%% file: compress_body.tex
\documentclass[11pt]{amsart}
\usepackage {amsmath, amssymb, amscd, mathrsfs, amsthm, stmaryrd,  bbm, diagbox, enumerate, slashed, graphicx, color, subfig, transparent, comment, float,bm}
\usepackage[dvipsnames]{xcolor}
\usepackage{url, hyperref}
\usepackage{footmisc}
\usepackage{tikz-cd}
\usepackage{csquotes}
\usetikzlibrary{decorations.pathmorphing}
\usepackage[margin=1.2in]{geometry}
\tikzset{
  squiggly/.style={
    decorate,
    decoration={snake, amplitude=0.6mm, segment length=3mm},
  }
}
\newtheorem{theorem}{Theorem}[section]
\newtheorem{proposition}[theorem]{Proposition}
\newtheorem{corollary}[theorem]{Corollary}
\newtheorem{lemma}[theorem]{Lemma}

\theoremstyle{remark}

\newtheorem*{remark*}{Remark}

\theoremstyle{definition}
\newtheorem{definition}[theorem]{Definition}

\newtheorem {remark}[theorem]{Remark}

\newcommand{\bp}{\mathbf{p}}
\newcommand{\bq}{\mathbf{q}}
\newcommand{\bu}{\mathbf{u}}
\newcommand{\bv}{\mathbf{v}}
\newcommand{\bV}{\mathbf{V}}
\newcommand{\bw}{\mathbf{w}}
\newcommand{\bz}{\mathbf{z}}
\newcommand{\be}{\mathbf{e}}
\newcommand{\bdel}{\bm{\delta}}
\newcommand{\bal}{\bm{\alpha}}
\newcommand{\bbe}{\bm{\beta}}
\newcommand{\bga}{\vec{\gamma}}
\newcommand{\cG}{\mathcal{G}}
\newcommand{\cP}{\mathcal{P}}
\newcommand{\cS}{\mathcal{S}}

\newcommand{\iso}[1]{\langle #1 \rangle}

\begin{document}
\title{loops of handleslides for sutured diagrams}
\author{Qianhe Qin}
\address {Department of Mathematics, Stanford University\\
Stanford, California, 94305, United States of America}
\email {\href{mailto:qqhe@stanford.edu}{qqhe@stanford.edu}}
\maketitle 
\begin{abstract}
We prove that the cut-system complex of a sutured compression body, with vertices representing cut-systems and edges corresponding to handleslides, becomes simply connected when six kinds of 2-cells are attached. Moreover, we define tight Heegaard invariants and show that each admits a unique extension to a strong Heegaard invariant. This gives a new framework for proving naturality results for Floer homology theories associated to sutured manifolds.
\end{abstract}
\section{introduction}
In \cite{JTZ}, Juh{\'a}sz, Thurston, and Zemke showed that all flavors of Heegaard Floer homology, link Floer homology, and sutured Floer homology are natural over \( \mathbb{F}_2 \). These invariants were originally defined only up to isomorphism (cf. Heegaard Floer homology in \cite{OS26, OS27}, link Floer homology in \cite{OSknot, Raknot, OSlink}, and sutured Floer homology in \cite{J}); they are examples of \emph{weak Heegaard invariants} in the terminology of \cite{JTZ}. To obtain more powerful results, we need a naturally associated object, not merely one defined up to isomorphism. In \cite{JTZ}, the authors considered a graph $\mathcal{G}$ that encodes the structure of the ``space of sutured Heegaard diagrams,'' related by certain moves. More precisely, the vertices of $\mathcal{G}$ correspond to isotopy diagrams of sutured manifolds, and the edges indicate whether two such diagrams are related by an $\alpha$/$\beta$-equivalence, a stabilization, or a diffeomorphism. For establishing naturality, they identified a simple generating set for the fundamental group of this graph and showed that Heegaard Floer homology has no monodromy around these generators. To formalize this, they introduced the notion of \emph{strong Heegaard invariants}. 

A strong Heegaard invariant is required to satisfy functoriality when restricted to the subgraph \( \mathcal{G}_{\alpha} \) (resp.\ \( \mathcal{G}_{\beta} \)), whose edges correspond to \( \alpha \)-equivalences (resp.\ \( \beta \)-equivalences). When working with coefficients in $\mathbb{F}_2$, given any \( \beta \)-equivalence \( (\Sigma, \bal, \bbe_1) \rightarrow (\Sigma, \bal, \bbe_2) \), the highest non-zero homological grading part of the Floer homology \( HF^{-}(\Sigma,\bbe_1,\bbe_2, \mathfrak{s}_0)\) is isomorphic to \( \mathbb{F}_2 \) (cf.\ \cite[Lemma 10.2]{JTZ}). Therefore, there exists a unique generator \( \Theta_{\beta_1,\beta_2} \) of this highest graded part, which is used to specify a ``canonical'' isomorphism
\[
\Psi^{\bal}_{\bbe_1\rightarrow\bbe_2} \colon HF^{\circ}(\Sigma, \bal, \bbe_1) \rightarrow HF^{\circ}(\Sigma, \bal, \bbe_2),
\]
as in \cite[Definition 10.3]{JTZ}. The same construction applies to \( \alpha \)-equivalences. However, over \( \mathbb{Z} \), the highest graded part is isomorphic to $\mathbb{Z}$, so there are two possible choices of the generator, and the map cannot be defined directly for \( \alpha \)/\( \beta \)-equivalences. Thus, for naturality over \( \mathbb{Z} \) (or for defining Heegaard Floer stable homotopy types), it is necessary to work with \( \alpha \)/\( \beta \)-handleslides rather than \( \alpha \)/\( \beta \)-equivalences. In fact, in the appendix of \cite{JTZ}, they consider an equivalent description of strong Heegaard invariants for classical single pointed Heegaard diagrams, using handleslides instead of equivalences.

In this paper, we generalize Theorem A.6 from \cite{JTZ} to three-manifolds and links with multiple basepoints, as well as to balanced sutured manifolds, and we provide a full proof of the result sketched in \cite[Appendix A]{JTZ}. Our argument is carried out in the setting of sutured manifolds, and we then apply the result to the classical cases of multi-based $3$-manifolds and multi-based links.

We first focus on handleslides involving only $\alpha$-curves or only $\beta$-curves. For this purpose, it suffices to consider one side of the sutured manifold, namely, a sutured compression body. Let $C(\bdel)$ denote the sutured compression body obtained by compressing the Heegaard surface along an attaching set $\bdel$.

We construct a $2$-complex \(Y_2(C(\bdel))\) following the approach of \cite{JTZ}. The vertices of \(Y_2(C(\bdel))\) are cut-systems of the sutured compression body \(C(\bdel)\), and the edges correspond to handleslides. We then attach triangular, rectangular, and pentagonal $2$-cells along six particular kinds of loops (see Definition \ref{def:hslide}).

\begin{theorem}
The $2$-complex \(Y_2(C(\bdel))\) is connected and simply-connected.
\end{theorem}

The fact that any two compression-equivalent attaching sets are related by a sequence of handleslides implies that the complex is connected (cf.~\cite[Lemma~2.11]{JTZ}). To prove simple-connectivity, we adapt the argument outlined in \cite{JTZ}. We begin by constructing another $2$-complex \(X_2(C(\bdel))\), whose $1$-skeleton contains that of \(Y_2(C(\bdel))\) as a spanning subgraph.

In \cite{Wa1998}, Wajnryb constructed a $2$-complex \(X_2(H_g)\) for a genus-$g$ handlebody \(H_g\), where the vertices are cut-systems and the edges correspond to \emph{simple moves}. We generalize this construction to the setting of sutured compression bodies and define a corresponding complex \(X_2(C(\bdel))\) for the sutured compression body \(C(\bdel)\). By modifying Wajnryb’s proof, we establish the following result by induction on the number of curves in the attaching set \(\bdel\).

\begin{theorem}
The $2$-complex \(X_2(C(\bdel))\) is connected and simply-connected.
\end{theorem}

Next, we adapt the method of \emph{minimal resolution}, originally introduced in \cite{JTZ}, to our setting. This method provides a way to resolve each simple move into a sequence of handleslides, unique up to homotopy in \(Y_2(C(\bdel))\). Moreover, we show that each $2$-cell \(D\) in \(X_2(C(\bdel))\) can be expressed as a sum of $2$-cells in \(Y_2(C(\bdel))\), where each edge on the boundary of \(D\) is replaced by its minimal resolution. Combining these results, and using the fact that \(X_2(C(\bdel))\) is simply-connected, we conclude that \(Y_2(C(\bdel))\) is also simply-connected. (See Figure \ref{sea_leaf}.)

\begin{figure}[h]
{
   \fontsize{10pt}{11pt}\selectfont
   \def\svgwidth{4.2in}
   \begin{center}
   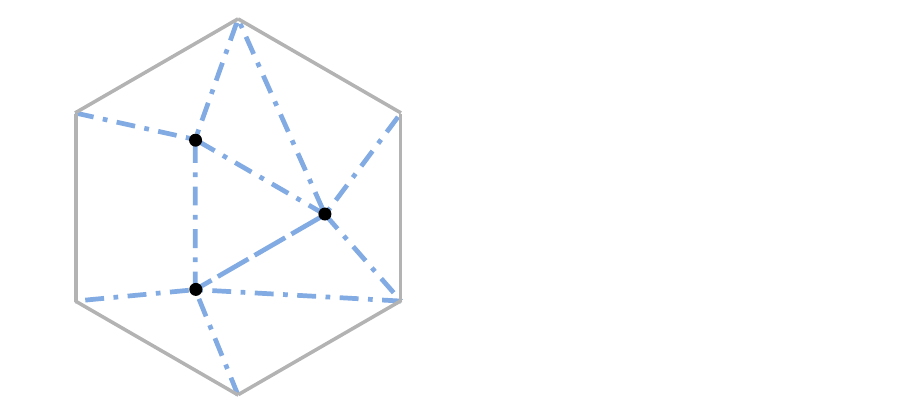
   \end{center}
   \caption{}
   \label{sea_leaf}
}
\end{figure}

Let \( \mathcal{G}' \) be the subgraph of \( \mathcal{G} \) obtained by replacing \( \alpha \)/\( \beta \)-equivalences with \( \alpha \)/\( \beta \)-handleslides. As in Appendix A of \cite{JTZ}, we can restrict the defining properties of weak and strong Heegaard invariants to \( \mathcal{G}' \). We refer to the corresponding Heegaard invariants defined on \( \mathcal{G}' \) as \emph{loose} and \emph{tight Heegaard invariants}. Let \(\cS\) be a set of diffeomorphism types of sutured manifolds, and let \(\mathcal{C}\) be any category. We denote by \(\cG(\cS)\) the full subgraph of \(\cG\) whose vertices are sutured isotopy diagrams \(H\) satisfying \(S(H) \in \cS\). Similarly, \(\cG'(\cS)\) is the full subgraph of \(\cG'\) spanned by sutured isotopy diagrams \(H\) for which \(S(H) \in \cS\). 

\begin{theorem}[Generalization of Theorem A.6 in \cite{JTZ}] 
Every loose Heegaard invariant $F': \cG'(\cS) \rightarrow \mathcal{C}$ extends to a weak Heegaard invariant $F: \cG(\cS) \rightarrow \mathcal{C}$. If $F'$ is tight, then $F'$ uniquely extends to a strong Heegaard invariant $F: \cG(\cS) \rightarrow \mathcal{C}$.
\end{theorem}

We can apply the theorem above to subsets of diffeomorphism types of balanced sutured manifolds, and obtain corollaries for classical (i.e., non-sutured) Heegaard invariants of 3-manifolds with multiple basepoints and of multi-based links in 3-manifolds (cf. Definition 2.4 and Definition 2.5 in \cite{JTZ}).

\begin{corollary} Let $\cS=\cS_{\mathrm{man*s}}$ be the set of all diffeomorphism types of sutured manifold associated to a closed oriented $3$-manifold with basepoints and an oriented tangent $2$-plane for each basepoint. Then every loose Heegaard invariant $F': \cG'(\cS) \rightarrow \mathcal{C}$ extends to a weak Heegaard invariant $F: \cG(\cS) \rightarrow \mathcal{C}$.
If, furthermore, $F'$ is a tight Heegaard invariant, then $F'$ uniquely extends to a strong Heegaard invariant $F: \cG(\cS) \rightarrow \mathcal{C}$.
\end{corollary}

\begin{corollary} Let $\cS=\cS_{\mathrm{link*s}}$ be the set of all diffeomorphism types of sutured manifold associated to a closed oriented $3$-manifold \( Y \) and a multi-based oriented link \((L, \bz, \bw)\) in \( Y \) (with at least one $\bz$ and one $\bw$ marking on each component of $L$). Then every loose Heegaard invariant $F': \cG'(\cS) \rightarrow \mathcal{C}$ extends to a weak Heegaard invariant $F: \cG(\cS) \rightarrow \mathcal{C}$.
If, furthermore, $F'$ is a tight Heegaard invariant, then $F'$ uniquely extends to a strong Heegaard invariant $F: \cG(\cS) \rightarrow \mathcal{C}$.
\end{corollary}

\medskip
\textbf{Organization of the paper.} In Section \ref{cut-systems}, we define cut-systems for sutured compression bodies. In Sections \ref{x2} and \ref{y2}, we construct the 2-complexes \( X_2(C(\boldsymbol{\delta})) \) and \( Y_2(C(\boldsymbol{\delta})) \), respectively, and show that they are connected and simply connected for any sutured compressionbody \( C(\boldsymbol{\delta}) \). In Section \ref{strong}, we define loose and tight Heegaard invariants and prove that they extend uniquely to weak and strong Heegaard invariants, respectively.

\medskip
\textbf{Acknowledgments.}
The author thanks Ciprian Manolescu for suggesting this problem and for his ongoing guidance and support.

\section{sutured compression body and cut-systems}\label{cut-systems}
Let $\Sigma$ be a compact, oriented (possibly disconnected) surface such that the boundary of each connected component of $\Sigma$ is nonempty. We also decorate $\Sigma$ with $n$ distinguished disks ${D_1, \dots, D_n} \subset \operatorname{Int}(\Sigma)$. All isotopies of $\Sigma$ are required to fix these distinguished disks. A curve on \(\Sigma\) refers to a simple closed curve on \(\operatorname{Int}(\Sigma)\) that is disjoint from the disks \(D_1, \dots, D_n\).

\begin{definition} [cf. Definition 2.6 in \cite{JTZ}]
An \emph{attaching set} \( \bdel \) is a collection of disjoint curves on \( \Sigma \) such that every connected component of \( \Sigma \setminus \bdel \) contains at least one component of \( \partial \Sigma \). We denote the isotopy class of \( \bdel \) by \( \iso{ \bdel } \). The number of curves in \( \bdel \) is denoted by \( |\bdel| \). See Figure \ref{attaching_set} for an example.
\end{definition}

\begin{figure}[h]
{
   \fontsize{12pt}{11pt}\selectfont
   \def\svgwidth{2.8in}
   \begin{center}
   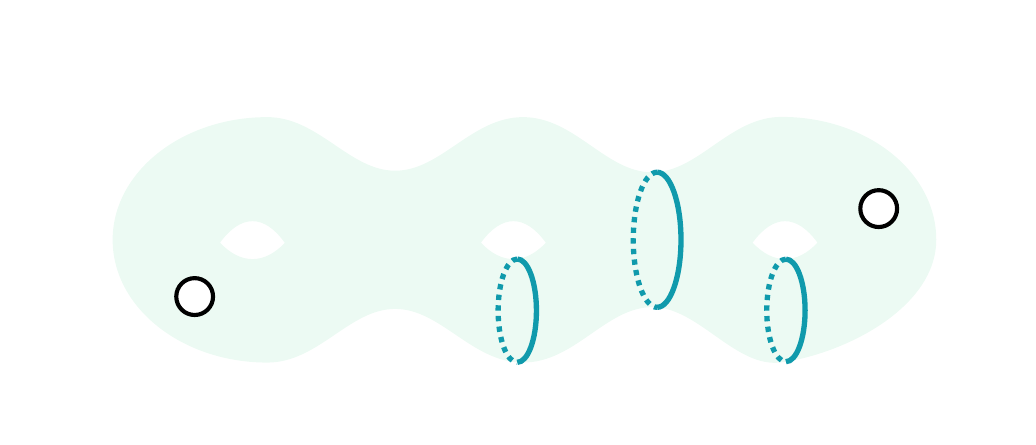
   \end{center}
   \caption{}
   \label{attaching_set}
}
\end{figure}

\begin{proposition}\label{prop:lid}
A collection of disjoint curves $\left(\alpha_1,\dots, \alpha_k\right)$ on $\Sigma$ forms an attaching set if and only if the homology classes $[\alpha_1],\dots, [\alpha_k]$ are linearly independent in $H_1(\Sigma;\mathbb{Z}/2)$.
\end{proposition}
\begin{proof}
Suppose $\left(\alpha_1,\dots, \alpha_k\right)$ is not an attaching set. Then, cutting the surface $\Sigma$ along the curves $\alpha_1,\dots, \alpha_k$ produces a connected component $F\subset \Sigma$ which contains no component of $\partial \Sigma$. Since $\partial F$ contains some curve $\alpha_i$ with multiplicity one mod 2, we obtain a nontrivial linear relation among the homology classes $[\alpha_1],\dots, [\alpha_k]$ in $H_1(\Sigma;\mathbb{Z}/2)$.

{Conversely}, suppose \((\alpha_1, \dots, \alpha_k)\) is an attaching set. Cutting the surface \( \Sigma \) along the curves \(\alpha_1, \dots, \alpha_k\) divides \( \Sigma \) into connected components \( F_1, \dots, F_m \). Choose a triangulation of \( \Sigma \) such that each $\alpha$-curve is a union of edges. Suppose there are coefficients \( c_i \in \mathbb{Z}/2 \) such that \(\sum_{i=1}^k c_i[\alpha_i] = 0\) in \( H_1(\Sigma; \mathbb{Z}/2) \). Thus, there exists a 2-chain \( C \) such that \(\sum_{i=1}^k c_i[\alpha_i] = \partial C\). Since each $F_j$ is connected, the coefficients of the triangles in $F_j$ must agree. Thus, we have that $C=\sum_{j=1}^m d_j F_j$ for some $d_j\in\mathbb{Z}/2$, and
$$\sum_{i=1}^k c_i[\alpha_i]=\partial C=\sum_{j=1}^m d_j \partial F_j$$
as 1-chains. Since \((\alpha_1, \dots, \alpha_k)\) is an attaching set, each $F_j$ contains at least one boundary component $\ell_j\subset \partial \Sigma$. Since $\partial C$ does not contain any $[\ell_j]$, we have that $d_j=0$ for all $i$, which implies that $c_i=0$ for all $i$. Hence, we showed that the homology classes are linearly independent in \( H_1(\Sigma; \mathbb{Z}/2) \).
\end{proof}

Given an attaching set $\bdel$ on $\Sigma$, let \( C(\bdel) \) be the sutured compression body obtained by taking \( \Sigma \times [0,1] \) and attaching 3-dimensional 2-handles along \( \bdel \times \{1\} \), with \( \gamma = \partial \Sigma \times [0,1] \) (see Figure \ref{boundary}). We define \( C_-(\bdel) = \Sigma \times \{0\} \) and  
\[
C_+(\bdel) = \partial C(\bdel) \setminus \text{Int} (C_-(\bdel) \cup \gamma).
\]
We abuse the notation and denote \( C_-(\bdel) \) by \( \Sigma \) for brevity.

\begin{figure}[h]
{
   \fontsize{10pt}{11pt}\selectfont
   \def\svgwidth{3.5in}
   \begin{center}
   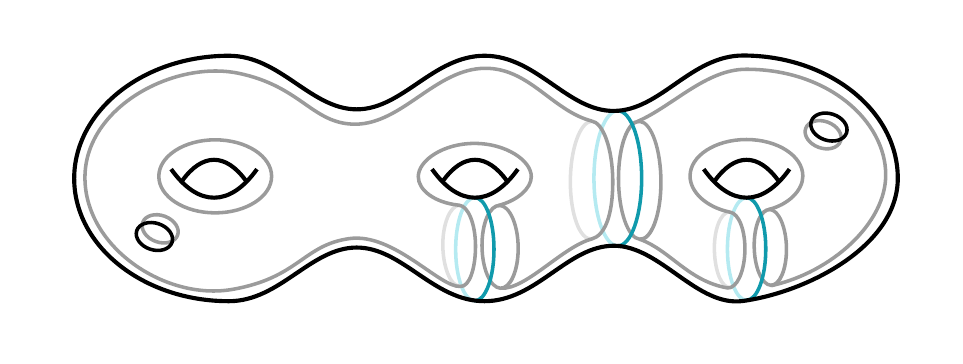
   \end{center}
   \caption{}
   \label{boundary}
}
\end{figure}

A curve on $\Sigma$ is called a \emph{meridian curve} if it bounds a properly-embedded disk in $C(\bdel)$. The proof of Lemma 2.9 in \cite{JTZ} shows that the sutured compression body $C(\bdel)$ is irreducible. Thus, given disjoint meridian curves \( \bal=(\alpha_1, \dots, \alpha_{k}) \) on \( \Sigma \), we can construct properly-embedded disjoint meridian disks \( D_{\alpha_1}, \dots, D_{\alpha_{k}} \) in the sutured compression body $C(\bdel)$. We refer to the process of removing disjoint tubular neighborhoods of meridian disks \( D_{\alpha_1}, \dots, D_{\alpha_{k}} \) as \emph{cutting along meridian curves} $\bal$. Notice that cutting along $\bal$ changes $C_-(\bdel)$ by a surgery along \( \bal \) and leaves both $C_+(\bdel)$ and the suture $\gamma$ unchanged. 

\begin{definition}[cf. Definition 2.8 in \cite{JTZ}]
Let \( \bdel \) and \( \bdel' \) be attaching sets on \( \Sigma \). We say they are \emph{compression equivalent}, denoted \( \bdel \sim \bdel' \), if there exists a diffeomorphism \( d: C(\bdel) \to C(\bdel') \) such that \( d|_{C_{-}(\bdel)} \) is the identity. This is an equivalence relation that descends to isotopy classes of attaching sets, and thus we write \( \iso{ \bdel } \sim \iso{ \bdel' } \) if \( \bdel \sim \bdel' \).
\end{definition}

\begin{proposition}\label{prop:equiv}
Let \(\bal = (\alpha_1, \dots, \alpha_{k})\) be an attaching set on $\Sigma$. Then, $\bal \sim\bdel$ if and only if the curves in $\bal$ are meridian curves of $C(\bdel)$ such that cutting \( C(\bdel) \) along \( \bal \) yields a sutured manifold diffeomorphic to the product sutured manifold \( C_+(\bdel) \times I \).
\end{proposition}
\begin{proof}
Suppose $\bal\sim\bdel$. Then, by definition, there exists a diffeomorphism \( d: C(\bal) \to C(\bdel) \) such that \( d|_{\Sigma} \) is the identity. For each $i$, let $D_{\alpha_i}$ be the core disk of the $2$-handle attached to $\alpha_i$ in $C(\bal)$. Choose tubular neighbourhoods $\nu(D_{\alpha_i})$ such that their images under $d$ are disjoint in $C(\bdel)$ and $C(\bal)\setminus \nu(D_{\bal}) \cong C_+(\bal) \times I$. Since \( d|_{\Sigma} \) is the identity, we have that $\partial(d(D_{\alpha_i})) =\alpha_i$. Thus, $\bal$ consists of meridian curves of $C(\bdel)$. Moreover, cutting $C(\bdel)$ along $\bal$ results in $C(\bdel)\setminus d(\nu(D_{\bal}))$. Since $C(\bal)\setminus \nu(D_{\bal}) \cong C_+(\bal) \times I$ and $C_+(\bal) \cong C_+(\bdel)$, we have that 
$$C(\bdel)\setminus d(\nu(D_{\bal})) \cong C(\bal)\setminus \nu(D_{\bal}) \cong C_+(\bal) \times I \cong C_+(\bdel) \times I$$
as sutured manifolds.

Conversely, suppose that the curves in \( \bal \) are meridian curves of \( C(\bdel) \) such that cutting \( C(\bdel) \) along \( \bal \) results in \( C_+(\bdel) \times I \). Since \( \bal \) consists of disjoint meridian curves, we choose disjoint meridian disks \( D_{\bal} \) for \( \bal \) and take a tubular neighborhood \( N \) of \( \Sigma \cup D_{\bal} \) in \( C(\bdel) \). Then, \( N \) is diffeomorphic to \( C(\bal) \), and \( C(\bdel)\setminus N \) is diffeomorphic to \( C_+(\bdel) \times I \). Thus, there exists an embedding \( \phi: C(\bal) \hookrightarrow C(\bdel) \), which induces a diffeomorphism from \( C(\bal) \) to \( C(\bdel) \) that restricts to the identity on \( \Sigma \), implying that \( \bal \sim \bdel \).
\end{proof}
We say that \( \iso{ \bal } \) is a \emph{cut-system} of \( C(\bdel) \) if \( \bal \sim \bdel \). In particular, we have \( |\bal| = |\bdel| \) (cf. \cite[Definition 2.8]{JTZ}). Equivalently, we will use the following definition, which offers a more direct approach by keeping the sutured compression body $C(\bdel)$ fixed, as follows from Proposition \ref{prop:equiv}.
\begin{definition}
A \emph{cut-system} of \( C(\bdel) \) is an isotopy class of a collection of disjoint meridian curves \( \bal = (\alpha_1, \dots, \alpha_{|\bdel|}) \) on \( \Sigma \), such that cutting \( C(\bdel) \) along \( \bal \) yields a sutured manifold diffeomorphic to the product sutured manifold \( C_+(\bdel) \times I \). We denote a cut-system by \( \iso{ \bal } \) or explicitly by \( \iso{ \alpha_1, \dots, \alpha_{|\bdel|} } \).
\end{definition}

\begin{definition}
A collection of disjoint meridian curves $\left(\alpha_1,\dots, \alpha_k\right)$ is called a \emph{pre-cut-system} if their homology classes are linearly independent in $H_1(\Sigma;\mathbb{Z}/2)$. A meridian curve $\alpha$ is called a \emph{good cut} if its homology class $[\alpha]$ is nonzero in $H_1(\Sigma;\mathbb{Z}/2)$. 
\end{definition}

\section{the $2$-complex of cut-systems and simple moves}\label{x2}
In this section, we construct a 2-complex \( X_2(C(\bdel)) \) associated to the sutured compression body \( C(\bdel) \), and prove that it is both connected and simply-connected. Our approach follows the conventions and methods developed in \cite{Wa1998}.

\begin{definition}
Two cut-systems \( \bu \) and \( \bv \) of \( C(\bdel) \) are said to be \emph{related by a simple move}, denoted by \( \bu \rightarrow \bv \), if there exists a set of disjoint meridian curves \( \{ \alpha, \alpha', \alpha_1, \dots, \alpha_{|\bdel|-1} \} \) such that \( \bu = \iso{ \alpha, \vec{\alpha} } \) and \( \bv = \iso{ \alpha', \vec{\alpha} } \), where \( \vec{\alpha} = (\alpha_1, \dots, \alpha_{|\bdel|-1}) \). For simplicity, we may omit the curves that remain unchanged and denote the simple move $\bu\rightarrow\bv$ by \( \iso{ \alpha } \rightarrow \iso{ \alpha' } \).
\end{definition}

We begin with the construction of the $1$-skeleton of $X_2(C(\bdel))$. The vertices of \( X_2(C(\bdel)) \) correspond to cut-systems of \( C(\bdel) \), and two vertices are connected by an edge if the corresponding cut-systems are related by a simple move. We denote this $1$-skeleton by $X_1(C(\bdel))$.

Next, we define the 2-cells of $X_2(C(\bdel))$. There are two types of triangular edge-paths in $X_1(C(\bdel))$:
\begin{itemize}
\item Type I: the triangle formed by cut-systems $\bv_i=\left\iso{\alpha_i, \vec{\alpha}\right}$ for $i \in \mathbb{Z}_3$.
\item Type II: the triangle formed by cut-systems $\bv_i=\left\iso{\alpha_{i-1}, \alpha_{i+1}, \vec{\alpha}\right}$ for $i \in \mathbb{Z}_3$. 
\end{itemize}
Unlike the case for handlebodies, since a cut-system of $C(\bdel)$ may include separating curves when $\partial \Sigma$ is disconnected (see Figure \ref{attaching_set}), we need to handle the case where $\Sigma$ itself is disconnected during the process of induction. Therefore, we also consider the following type of square loops that did not appear in the $2$-complex of cut-systems for handlebodies in \cite{Wa1998}.

\begin{definition}
A closed edge-path
\[\begin{tikzcd}[row sep=large, column sep=large]
	{\iso{ \alpha_1, \alpha_3, \vec{\alpha} }} & {\iso{ \alpha_1, \alpha_4, \vec{\alpha} }} \\
	{\iso{ \alpha_2, \alpha_3, \vec{\alpha} }} & {\iso{ \alpha_2, \alpha_4, \vec{\alpha} }}
	\arrow[no head, from=1-1, to=1-2]
	\arrow[no head, from=1-2, to=2-2]
	\arrow[no head, from=2-1, to=1-1]
	\arrow[no head, from=2-2, to=2-1]
\end{tikzcd}\]
is called a \emph{self-separated square} if $\{\alpha_1,\alpha_2\}$ and $\{\alpha_3,\alpha_4\}$ are contained in two different connected components of $\Sigma\setminus \vec{\alpha}$. 
\end{definition}
For example, Figure \ref{self_sep_square} gives a self-separated square for the compression body in Figure \ref{boundary}

\begin{figure}[h]
{
   \fontsize{11pt}{11pt}\selectfont
   \def\svgwidth{4.2in}
   \begin{center}
   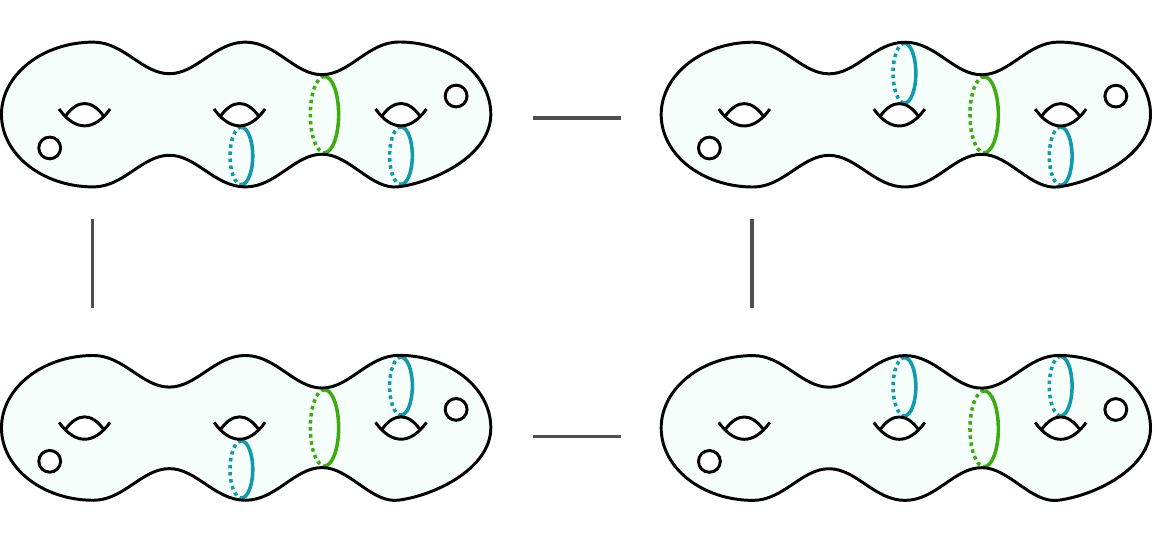
   \end{center}
   \caption{}
   \label{self_sep_square}
}
\end{figure}

We attach a $2$-cell to each triangular edge-path in \( X_1(C(\bdel)) \) and a $2$-cell to each self-separated square in \( X_1(C(\bdel)) \). The resulting $2$-complex, denoted by \( X_2(C(\bdel))\), is called the \emph{$2$-complex of cut-systems and simple moves} associated to the sutured compression body $C(\bdel)$. 

\begin{theorem}\label{thm:xsimpc}
The $2$-complex $X_2(C(\bdel))$ is connected and simply-connected.
\end{theorem}

We divide the proof into two parts: we first show that \( X_2(C(\bdel)) \) is connected, and then establish that it is simply-connected; both proceed by induction on \( |\bdel| \). We start with the proof of connectivity. For the base case $|\bdel|=0$, the complex $X_2(C(\bdel))$ consists of the empty cut, which is connected and simply-connected. For the inductive step, suppose $X_2(C(\bdel))$ is connected for $|\bdel|\le k$; we aim to show that $X_2(C(\bdel))$ is connected for $|\bdel| = k+1$. 

\begin{lemma}\label{lem:common}
Let $\bu,\bv$ be two vertices in $X_2(C(\bdel))$. If $\bu$ and $\bv$ have one or more curves in common, then $\bu$ and $\bv$ are connected by an edge-path in $X_2(C(\bdel))$ where every vertex contains the common curves.
\end{lemma}

\begin{proof}
Let $\bbe$ be common curves between $\bu$ and $\bv$. Cut the sutured compression body $C(\bdel)$ along $\bbe$ and obtain a new sutured compression body $C(\bdel')$. The cut-system $\bu$ restricts to a cut-system $\bu'$ on $C(\bdel')$, and the cut-system $\bv$ restricts to a cut-system $\bv'$ on $C(\bdel')$. By the induction hypothesis, since $|\bdel'|=|\bdel|-|\bbe|<|\bdel|$, there exists an edge-path $\bp'$ between $\bu'$ and $\bv'$ in $X_2(C(\bdel'))$. We augment the edge-path $\bp'$ by $\bbe$ and obtain an edge-path $\bp$ from the cut-system $\bu = \iso{ \bu', \bbe }$ to the cut-system $\bv = \iso{ \bv', \bbe }$ in $X_2(C(\bdel))$, such that each vertex of $\bp$ contains the common curves $\bbe$.
\end{proof}

\begin{lemma}\label{lem7}
Let $\alpha$ and $\beta$ be good cuts of $C(\bdel)$. Then, there exist vertices $\bv_1$, $\bv_2$ in $X_2(C(\bdel))$ (possibly $\bv_1=\bv_2$) such that $\alpha\in \bv_1$, $\beta\in \bv_2$ and $\bv_1$ and $\bv_2$ are connected by an edge-path in $X_2(C(\bdel))$.
\end{lemma}

\begin{proof}
We prove by induction on the intersection number $n$ between the curves $\alpha$ and $\beta$. First, consider the base case where $\alpha$ and $\beta$ are disjoint. If $(\alpha,\beta)$ is a pre-cut-system, we complete it to a cut-system $\bv\in X_2(C(\bdel))$. Thus, we choose $\bv_1=\bv_2=\bv$. 

Otherwise, cut along $\alpha$ and $\beta$. Since both $\alpha$ and $\beta$ are good cuts, we obtain a handlebody $H$ and a sutured compression body $C$. Choose a cut-system for $H$ as a handlebody, and a cut-sysem for $C$ as a sutured compression body. Together, we get disjoint meridian curves $\bga$. Cut the sutured compression body $C(\bdel)$ along $\{\alpha,\beta, \bga\}$, and end up with the product sutured manifold $C_+(\bdel)\times I$, together with a ball that comes from $H$. Thus, we have constructed two cut-systems $\bv_1=\iso{ \bga, \alpha }$ and $\bv_2=\iso{ \bga, \beta }$ that are connected by an edge in $X_2(C(\bdel))$.

Now, assume that \( n > 0 \). Since both \( \alpha \) and \( \beta \) are meridian curves of \( C(\bdel) \), they bound meridian disks \( D_{\alpha} \) and \( D_{\beta} \), respectively, which intersect transversely. Since \( C(\bdel) \) is irreducible, we can get rid of all circular intersections by moving the disks across 3-balls corresponding to innermost circles. Let \( d \) be an innermost arc of intersection on \( D_{\alpha} \). The arc \( d \) separates \( D_{\alpha} \) into two disks; let \( D_{\alpha}' \) denote the one whose interior is disjoint from \( D_{\beta} \). Let \( \varepsilon \) be the arc on \( \alpha \) such that \( d \cup \varepsilon \) bounds \( D_{\alpha}' \). See Figure \ref{inner_d}.

\begin{figure}[h]
{
   \fontsize{12pt}{11pt}\selectfont
   \def\svgwidth{2.5in}
   \begin{center}
   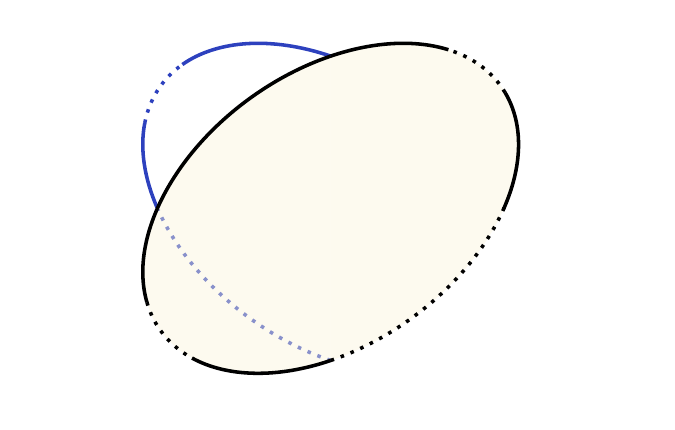
   \end{center}
   \caption{}
   \label{inner_d}
}
\end{figure}

The endpoints of $\varepsilon$ divide the meridian curve \( \beta \) into two subarcs, which we denote by \( \beta_1 \) and \( \beta_2 \). Let $\gamma_i$ be a push-off of $\beta_i+\varepsilon$ such that $|\gamma_i \cap \beta| = 0$ and $|\gamma_i \cap \alpha|< n$ for $i\in\{1,2\}$. Since $\beta$ is a good cut, we have that $[\gamma_1] + [\gamma_2] = [\beta] \neq 0 \in H_1(\Sigma;\mathbb{Z}/2)$. It follows that at least one of \(\gamma_i\) represents a nonzero homology class and therefore is a good cut. We denote this good cut by \(\gamma\). 

By the induction hypothesis, there exist vertices \(\bv_1\) and \(\bv_3\) in \(X_2(C(\bdel))\) containing \(\alpha\) and \(\gamma\), respectively, and connected by an edge-path in \(X_2(C(\bdel))\). Similarly, there exist vertices \(\bv_2\) and \(\bv_4\) containing \(\beta\) and \(\gamma\), also connected by an edge-path. Finally, the vertices $\bv_3$ and $\bv_4$ are connected by an edge-path by Lemma \ref{lem:common}. 
\[\begin{tikzcd}[row sep=large]
	{\alpha\in\bv_1} &&& {\bv_2 \ni \beta} \\
	\\
	{\gamma\in\bv_3} &&& {\bv_4 \ni \gamma}
	\arrow["{{\mathrm{induction}}}", shift right=4, squiggly, no head, from=1-4, to=3-4]
	\arrow["{{\mathrm{induction}}}", shift right=4, squiggly, no head, from=3-1, to=1-1]
	\arrow["{{\mathrm{Lemma\ \ref{lem:common}}}}"', squiggly, no head, from=3-1, to=3-4]
\end{tikzcd}\]

\end{proof}

\begin{proposition}\label{prop8}
The complex \(X_2(C(\bdel))\) is connected.
\end{proposition}

\begin{proof}
Let $\bv_1$ be a vertex of \(X_2(C(\bdel))\) containing some curve $\alpha$ and $\bv_2$ be a vertex of \(X_2(C(\bdel))\) containing some curve $\beta$. By Lemma \ref{lem7}, there exist other vertices $\bv_3$ and $\bv_4$ (possibly $\bv_3=\bv_4$) connected by an edge-path and containing $\alpha$ and $\beta$ respectively. By Lemma \ref{lem:common}, we can connect $\bv_1$ to $\bv_3$ and $\bv_4$ to $\bv_2$ by edge-paths in \(X_2(C(\bdel))\).
\[\begin{tikzcd}[row sep=large]
	{\alpha\in\bv_1} &&& {\bv_2 \ni \beta} \\
	\\
	{\alpha\in\bv_3} &&& {\bv_4 \ni \beta}
	\arrow["{\mathrm{Lemma\ \ref{lem:common}}}", shift right=4, squiggly, no head, from=1-4, to=3-4]
	\arrow["{\mathrm{Lemma\ \ref{lem:common}}}", shift right=4, squiggly, no head, from=3-1, to=1-1]
	\arrow["{\mathrm{Lemma\ \ref{lem7}}}"', squiggly, no head, from=3-1, to=3-4]
\end{tikzcd}\]
\end{proof}

We proceed to prove that \( X_2(C(\bdel)) \) is simply-connected by induction on \( |\bdel| \). Following the approach in \cite{Wa1998}, we aim to show that any closed edge-path can be decomposed into a sum of triangles and squares. We choose a spanning tree for the closed edge-path and then fix a specific collection of curves to represent each vertex, such that every edge in the spanning tree is given by \( \iso{ \alpha } \rightarrow \iso{ \alpha' } \). We also assume that all curves in these collections intersect transversely at a finite number of points, with no point belonging to more than two curves. Throughout the proof, we will introduce new curves subject to the same conditions.
\begin{definition}
A \emph{segment} is an edge-path in which all vertices share a common curve. Specifically, an edge-path \( \bp \) is called an \( \iso{ \vec{\alpha} } \)-segment if all vertices of \( \bp \) contain \( \vec{\alpha} \), where \( \vec{\alpha} \) is a non-empty set of curves.
\end{definition}
In the following lemmas, we show that certain types of closed edge-paths are null-homotopic. These will serve as building blocks for decomposing general closed edge-paths.

\begin{lemma}\label{lem:seg}
A closed segment in $X_2(C(\bdel))$ is null-homotopic.
\end{lemma}
\begin{proof}
Let \(\bp\) be a closed \(\iso{ \alpha }\)-segment. We cut the sutured compression body $C(\bdel)$ along the meridian curve \(\alpha\), and obtain a new sutured compression body \(C(\bdel')\) with $|\bdel'| = |\bdel|-1$. {Remove \(\alpha\) from each vertex of \(\bp\), and obtain a closed edge-path \(\bp'\) in \(X_2(C(\bdel'))\).} By the induction hypothesis, the edge-path \(\bp'\) is a sum of 2-cells in \(X_2(C(\bdel'))\). Augmenting these 2-cells with the curve \(\alpha\) shows that the closed segment \(\bp\) is null-homotopic in \(X_2(C(\bdel))\).
\end{proof}
\begin{lemma}\label{lem:square0}
The closed path $\bp=\iso{ \alpha, \gamma } \rightarrow \iso{ \beta, \gamma } \rightarrow \iso{ \beta, \delta} \rightarrow \iso{ \alpha, \delta} \rightarrow \iso{ \alpha, \gamma }$ is null-homotopic in $X_2(C(\bdel))$, where the meridian curves $\alpha,\beta,\gamma$, and $\delta$ are pairwise disjoint.
\end{lemma}
As a first step toward the proof, we establish the following intermediate lemma.
\begin{lemma}\label{prop:for_lem}
Let \(\alpha, \beta, \delta, \gamma\) be four disjoint curves on a connected surface \(\Sigma\) such that their homology classes in \(H_1(\Sigma; \mathbb{Z}/2)\) are pairwise linearly independent except for the pairs \((\alpha, \beta)\) and \((\gamma, \delta)\). Then, there exists a connected component \(F\) of  
\[ \Sigma_0 := \Sigma \setminus (\alpha \cup \beta \cup \gamma \cup \delta) \]  
such that the boundary \(\partial F\) contains two curves from \(\{\alpha, \beta, \gamma, \delta\}\) that are linearly independent in \(H_1(\Sigma; \mathbb{Z}/2)\).
\end{lemma}
\begin{proof}
We cut \(\Sigma\) open along \(\alpha\) and \(\gamma\), and let \(n\) be the number of connected components in \[\Sigma' = \Sigma \setminus (\alpha \cup \gamma).\] Depending on the separating properties of \(\alpha\) and \(\gamma\), we have \(n = 1\), \(2\), or \(3\). 

For \( n = 1 \), we adapt the argument in the proof of \cite[Lemma 10]{Wa1998}. The surface \( \Sigma' \setminus (\beta \cup \delta) \) has three connected components and eight new boundary components. It follows that one of these components must have at least three boundary components from the set \( \{ \alpha, \beta, \gamma, \delta \} \). By the separation properties of \( \alpha, \beta, \gamma \), and \( \delta \), the boundary of this component contains three distinct curves from \( \{ \alpha, \beta, \gamma, \delta \} \). Hence, at least two of them are linearly independent.

For \(n = 2\), each connected component of \(\Sigma'\) has \(\alpha \cup \gamma\) as part of its boundary. If there exists a connected component \(F \subset \Sigma'\) that does not intersect \(\beta \cup \delta\), then \(F\) is a connected component of \(\Sigma_0\) such that the boundary $\partial F$ contains \(\alpha\) and \(\gamma\). Otherwise, consider the connected component of \(\Sigma'\) containing \(\beta\), and choose a connected componet \(F \subset \Sigma_0\) whose boundary contains \(\beta \cup \gamma\).

For \(n = 3\), let \(F'\) be the connected component of \(\Sigma'\) with \(\alpha \cup \gamma \subset \partial F'\). If \(F'\) does not intersect \(\beta\) or \(\delta\), we can find \(F \subset F'\) satisfying the desired property, by the same argument as for \(n = 2\). Otherwise, \(F'\) contains both \(\beta\) and \(\delta\). Choose a connected component \(F'' \subset F'\) whose boundary contains \(\beta \cup \gamma\). If $F''$ does not intersect $\delta$, then $F''$ has desired properties. Otherwise, we cut \(F''\) along \(\delta\) and select a connected component \(F \subset F''\) such that \(\partial F = \beta \cup \delta\).
\end{proof}
\begin{proof}[Proof of Lemma \ref{lem:square0}]
If \(|\bdel| > 2\), the path \(\bp\) is a segment and is null-homotopic by Lemma \ref{lem:seg}. Thus, it suffices to consider the case \(|\bdel| = 2\). Suppose \(\iso{ \alpha, \beta }\) (or \(\iso{ \gamma, \delta }\)) is a cut-system of \(C(\bdel)\). Observe that \(\iso{ \alpha, \beta }\) (or \(\iso{ \gamma, \delta }\)) is connected by an edge to every vertex of the path \(\bp\) in \(X_2(C(\bdel))\). Thus, the path \(\bp\) can be decomposed into a sum of four triangles, each containing \(\iso{ \alpha, \beta }\) (or \(\iso{ \gamma, \delta }\)). 

The remaining case to consider is when neither \(\{\alpha, \beta\}\) nor \(\{\gamma, \delta\}\) forms a cut-system. Since \(\{\alpha, \beta\}\) (respectively, \(\{\gamma, \delta\}\)) does not form a cut-system, we have that \(\{\alpha, \beta\}\) (respectively, \(\{\gamma, \delta\}\)) belong to the same connected component of \(\Sigma\). If \(\{\alpha, \beta\}\)  and \(\{\gamma, \delta\}\) belong to different connected components of \(\Sigma\), then the path \(\bp\) is a self-separated square. 

Otherwise, let \({\Sigma}'\) be the connected component containing all four curves. By Lemma \ref{prop:for_lem}, there exists a connected component $F$ in ${\Sigma}'\setminus (\alpha\cup\beta\cup\gamma\cup\delta)$ such that $\partial F$ contains two curves which forms a vertex in $\bp$, say $\alpha$ and $\gamma$. We take a connect sum $\alpha+\gamma$ inside $F$ and get a new meridian curve $\varepsilon$ disjoint from all four curves. Since $\varepsilon$ forms a vertex with each of the four curves, the square path \(\bp\) can be decomposed into a sum of ten triangles as follows (cf. Figure 4 of \cite{Wa1998}).
\[\begin{tikzcd}[cramped, row sep=small, column sep=small]
	&&& {\iso{\alpha,\gamma}} \\
	\\
	\\
	\\
	&& {\iso{\alpha,\varepsilon}\hspace{-1em}} && {\hspace{-1em}\iso{\gamma,\varepsilon}} \\
	{\iso{\alpha,\delta}} &&&&&& {\iso{\beta,\gamma}} \\
	&& {\iso{\delta,\varepsilon}\hspace{-1em}} && {\hspace{-1em}\iso{\beta,\varepsilon}} \\
	\\
	\\
	\\
	&&& {\iso{\beta,\delta}}
	\arrow[no head, from=1-4, to=5-5]
	\arrow[no head, from=1-4, to=6-1]
	\arrow[no head, from=1-4, to=6-7]
	\arrow[no head, from=5-3, to=1-4]
	\arrow[no head, shorten <=10pt, shorten >=10pt, from=5-3, to=5-5]
	\arrow[no head, shift left=1, from=5-3, to=7-3]
	\arrow[no head, shift right=1, from=5-5, to=7-5]
	\arrow[no head, from=6-1, to=5-3]
	\arrow[no head, from=6-1, to=7-3]
	\arrow[no head, from=6-1, to=11-4]
	\arrow[no head, from=6-7, to=5-5]
	\arrow[no head, shorten <=10pt, shorten >=10pt, from=7-3, to=5-5]
	\arrow[no head, from=7-3, to=11-4]
	\arrow[no head, from=7-5, to=6-7]
	\arrow[no head, shorten <=10pt, shorten >=10pt, from=7-5, to=7-3]
	\arrow[no head, from=11-4, to=6-7]
	\arrow[no head, from=11-4, to=7-5]
\end{tikzcd}\]
\end{proof}

\begin{lemma}\label{lem11}
Let $\alpha$ and $\beta$ be good cuts. Let $\bv_1$ and $\bv_2$ be vertices of $X_2(C(\bdel))$ connected by an edge $\be_1$ and containing $\alpha$ and $\beta$ respectively (or $\bv_1=\bv_2$ contains both $\alpha$ and $\beta$). Let $\bv_3$ and $\bv_4$ be another pair of vertices of $X_2(C(\bdel))$ connected by an edge $\be_2$ and containing $\alpha$ and $\beta$ respectively. Then there exists an $\iso{\alpha}$-segment connecting $\bv_1$ and $\bv_3$, and a $\iso{\beta}$-segment connecting $\bv_2$ and $\bv_4$ such that the closed edge-path they form is null-homotopic in $X_2(C(\bdel))$.
\end{lemma}

\begin{proof}
Suppose \((\alpha, \beta)\) is a pre-cut-system. Then, we select \(|\bdel|-2\) curves \(\vec{\gamma_1}\) from the common curves of $\bv_1$ and $\bv_2$ such that \(\bv_5 =\iso{ \alpha, \beta, \vec{\gamma_1} }\) is a cut-system. Similarly, we choose \(|\bdel|-2\) curves \(\vec{\gamma_2}\) from the common curves of $\bv_3$ and $\bv_4$ such that \(\bv_6 =\iso{ \alpha, \beta, \vec{\gamma_2} }\) is also a cut-system. By Lemma \ref{lem:common}, we have that \(\bv_5 \) and \(\bv_6\) are connected by an \(\iso{ \alpha, \beta }\)-segment. We extend the \( \iso{\alpha, \beta} \)-segment to an \( \alpha \)-segment \( \bp_1 \) from \( \bv_1 \) to \( \bv_3 \), and to a \( \beta \)-segment \( \bp_2 \) from \( \bv_2 \) to \( \bv_4 \). Since \( (\bv_1, \bv_2, \bv_5) \) forms one triangle and \( (\bv_3, \bv_4, \bv_6) \) forms another, the union of the segments \( \bp_1 \), \( \bp_2 \), and the two edges \( \be_1 \), \( \be_2 \) is a sum of two triangles and a trivial path, as illustrated below (cf. \cite[Figure 5 (right)]{Wa1998}).

\[\begin{tikzcd}
	{\bv_1} && {\bv_2} \\
	& {\bv_5 = \iso{\alpha,\beta,\vec{\gamma_1}}} \\
	\\
	\\
	& {\bv_6 = \iso{\alpha,\beta,\vec{\gamma_2}}} \\
	{\bv_3} && {\bv_4}
	\arrow["{\be_1}", no head, from=1-1, to=1-3]
	\arrow[no head, from=1-1, to=2-2]
	\arrow["{\bp_1}"{description}, dashed, no head, from=1-1, to=6-1]
	\arrow[no head, from=1-3, to=2-2]
	\arrow["{\bp_2}"{description}, dashed, no head, from=1-3, to=6-3]
	\arrow["{\iso{\alpha,\beta}-\mathrm{segment}}"{description}, squiggly, no head, from=2-2, to=5-2]
	\arrow[no head, from=6-1, to=5-2]
	\arrow["{\be_2}"', no head, from=6-1, to=6-3]
	\arrow[no head, from=6-3, to=5-2]
\end{tikzcd}\]

Suppose \( (\alpha, \beta) \) is not a pre-cut-system. Cut \( C(\bdel) \) along the curves \( \alpha \) and \( \beta \) and obtain a sutured compression body $C(\bdel')$ and a handlebody $H$. The common curves in $\bv_1$ and $\bv_2$ split into a cut-system $\bu_1$ of  $C(\bdel')$ and a cut-system $\bu_2$ of $H$. Similarly, the common curves in the vertices $\bv_3$ and $\bv_4$ split into a cut-system $\bw_1$ of $C(\bdel')$ and a cut-system $\bw_2$ of $H$. Connect \( \bu_1 \) and \( \bw_1 \) with an edge-path in \( X_2(C(\bdel')) \), and connect $\bu_2$ and $\bw_2$ with an edge-path in $X_2(H)$. 

Connect \( \bv_1 \) and \( \bv_3 \) with a \( \iso{ \alpha } \)-segment \( \bp_{\alpha} \) as in the diagram below. Replacing \( \alpha \) with \( \beta \) in the edge-path \( \bp_{\alpha} \) yields a \( \iso{ \beta } \)-segment \( \bp_{\beta} \), which connects \( \bv_2 \) and \( \bv_4 \). The edge-paths \( \bp_\alpha \) and \( \bp_\beta \), together with the edges connecting their endpoints, decompose into squares that are null-homotopic by Lemma \ref{lem:square0}, as shown below (cf. \cite[Figure 5 (left)]{Wa1998}).
\[\begin{tikzcd}[cramped, row sep=small, column sep=tiny]
	{\bv_1=\hspace{-1em}} & {\iso{\bu_1,\bu_2,\alpha}} && {\iso{\bu_1,\bu_2,\beta}} & {\hspace{-1em}=\bv_2} \\
	& {\iso{\bu_1',\bu_2,\alpha}} && {\iso{\bu_1',\bu_2,\beta}} \\
	\\
	& {\iso{\bw_1,\bu_2,\alpha}} && {\iso{\bw_1,\bu_2,\beta}} \\
	& {\iso{\bw_1,\bu_2',\alpha}} && {\iso{\bw_1,\bu_2',\beta}} \\
	\\
	{\bv_3=\hspace{-1em}} & {\iso{\bw_1,\bw_2,\alpha}} && {\iso{\bw_1,\bw_2,\beta}} & {\hspace{-1em}=\bv_4}
	\arrow[no head, from=1-2, to=1-4]
	\arrow[no head, from=1-2, to=2-2]
	\arrow[no head, from=1-4, to=2-4]
	\arrow[no head, from=2-2, to=2-4]
	\arrow[squiggly, no head, from=2-2, to=4-2]
	\arrow[squiggly, no head, from=2-4, to=4-4]
	\arrow[no head, from=4-2, to=4-4]
	\arrow[no head, from=4-2, to=5-2]
	\arrow[no head, from=4-4, to=5-4]
	\arrow[no head, from=5-2, to=5-4]
	\arrow[squiggly, no head, from=5-2, to=7-2]
	\arrow[squiggly, no head, from=5-4, to=7-4]
	\arrow[no head, from=7-2, to=7-4]
\end{tikzcd}\]
\end{proof}

We define the distance between two curves $\alpha$ and $\beta$ as their intersection number $\left|\alpha \cap \beta\right|$. The distance between a curve $\gamma$ and a collection of curves $\left(\alpha_1, \ldots, \alpha_{k}\right)$ is equal to the minimum of the distances $\left|\gamma \cap \alpha_i\right|$. The radius of the edge-path around a curve $\alpha$ is equal to the maximum of the distances from $\alpha$ to the vertices of the edge-path. 

For a general closed edge-path, the goal is to reduce it to a sum of simpler building blocks from the previous lemmas. This is achieved by induction on the radius of the edge-path around a fixed curve \( \alpha \), and by induction on the number of segments at the maximal distance from \( \alpha \). 
\begin{lemma}\label{lem12}
A closed edge-path ${\bp}$ of radius zero around a curve $\alpha$ of some vertex of ${\bp}$ is null-homotopic in $X_2(C(\bdel))$.
\end{lemma}

\begin{proof}
The proof follows the same idea as Lemma 12 in \cite{Wa1998}; we adapt their argument to our setting. We prove the lemma by induction on the number of segments sharing a common curve disjoint from \( \alpha \). If all vertices contain \( \alpha \), then we have a closed segment, which is null-homotopic by Lemma \ref{lem:seg}. 

Otherwise, consider a maximal segment containing \( \alpha \). Let \( \bu_1 \) be its first vertex and \( \bv_1 \) its last. Let \( \bu_2 \) be the next vertex of \( \bp \); then $\bu_2$ contains a curve $\beta$ disjoint from $\alpha$, the common curve of the second segment. Let \( \bv_2 \) be the last vertex of this second segment. Suppose $\bp$ consists of exactly two such segments. By Lemma \ref{lem11}, we can connect the pairs \( \bv_1, \bu_2 \) and \(\bu_1,\bv_2 \) by parallel segments, so that \( \bp \) splits into a sum of two closed segments and a null-homotopic closed edge-path (see Figure \ref{melon}). 

\begin{figure}[h]
{
   \fontsize{10pt}{11pt}\selectfont
   \def\svgwidth{2.8in}
   \begin{center}
   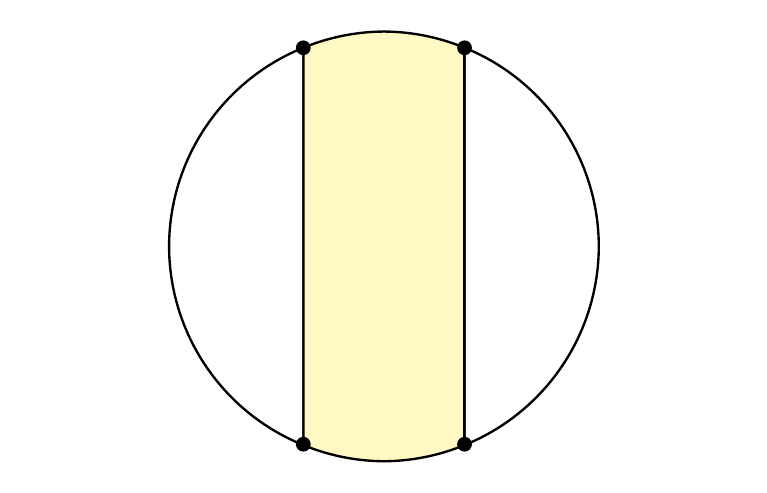
   \end{center}
   \caption{}
   \label{melon}
}
\end{figure}
Now suppose that \( \bp \) contains more than two segments. In the inductive step of Lemma 12 in \cite{Wa1998}, part of the path is replaced by a homotopic subpath to reduce the number of segments. We adapt this idea to our setting by modifying the choice of vertices \( w_1, w_2, w_3 \). Depending on the dimension of $$V:=\text{span} \{[\alpha],[\beta],[\gamma]\}\subset H_1(\Sigma;\mathbb{Z}/2),$$ there are three different cases. 

Suppose $\dim V = 3$. Since the curves $(\alpha,\beta, \gamma)$ forms a pre-cut-system, we complete it into a cut-system by adding $|\bdel|-3$ curves. Thus, we obtain a vertex $\bv$ which contains $\alpha,\beta, \gamma$. Let $w_1=w_2=w_3= \bv$.

Suppose $\dim V = 2$. Without loss of generality, we may assume that $(\alpha,\beta)$ forms a pre-cut-system. Cutting $C(\bdel)$ along $(\alpha,\beta,\gamma)$ gives us a handlebody $H$ and a sutured compression body $C$. Choose cut-systems for both $H$ and $C$, and we get $|\bdel|-2$ disjoint meridian curves. We can find vertices $\bv_1$ and $\bv_2$ which are connected by an edge, such that $\alpha, \beta \in \bv_1$ and $\gamma\in \bv_2$. Let $w_1 = w_2 = \bv_1$ and $w_3 = \bv_2$.

Suppose $\dim V=1$. Cutting $C(\bdel)$ along $(\alpha,\beta,\gamma)$ results in two handlebodies $H_1, H_2$ and a sutured compression body $C$. We choose cut-systems for $H_1$, $H_2$ and $C$, and obtain $|\bdel|-1$ disjoint meridian curves $\vec{\eta}$. Let $w_1 = \iso{ \alpha,\vec{\eta}}, w_2 = \iso{ \beta,\vec{\eta}}, w_3 = \iso{ \gamma,\vec{\eta}}$. 

Finally, the closed edge-path shown in \cite[Figure 6]{Wa1998} remains null-homotopic in \( X_2(C(\bdel)) \). This follows by replacing the use of Lemma 9 and Lemma 11 in their argument with Lemma \ref{lem:seg} and Lemma \ref{lem11}, respectively. With these substitutions, the inductive argument goes through, and the result follows by induction on the number of segments.
\end{proof}

\begin{lemma} \label{lem5}
Let $\alpha, \beta, \gamma_1, \gamma_2$ be good cuts such that $\left|\gamma_1 \cap \gamma_2\right|=k>0,\left|\gamma_1 \cap \alpha\right|<n$, and $\left|\gamma_2 \cap \alpha\right| \leq n$. Then there exists a good cut $\delta$ such that $\left|\delta \cap \gamma_1\right|<k,\left|\delta \cap \gamma_2\right|<k$, and $\left|\delta \cap \alpha\right|<n$. If also $\left|\gamma_1 \cap \beta\right|=\left|\gamma_2 \cap \beta\right|=0$, then $|\delta \cap \beta|=0$.
\end{lemma}
\begin{proof}
The proof of \cite[Lemma 5]{Wa1998} applies with \( S \) replaced by \( \Sigma \) and \emph{non-separating} replaced by \emph{a good cut}.
\end{proof}

\begin{lemma}\label{lem13}
Let $\alpha, \beta, \gamma_1, \gamma_2$ be good cuts such that $\left|\gamma_1 \cap \gamma_2\right|=k,\left|\gamma_1 \cap \alpha\right|<m,\left|\gamma_2 \cap \alpha\right| \leq m$ and $\left|\gamma_1 \cap \beta\right|=\left|\gamma_2 \cap \beta\right|=0$. Let $\bu_1$ and $\bu_2$ be vertices of $X_2(C(\bdel))$ containing $\gamma_1$ and $\gamma_2$ respectively. Then there exists an edge-path ${\bq}$ connecting $\bu_1$ and $\bu_2$ such that all vertices of ${\bq}$ have distance zero from $\beta$ and have distance less than $m$ from $\alpha$, except possibly for a final $\gamma_2$-segment of ${\bq}$ which ends at $\bu_2$.
\end{lemma}

\begin{proof}
The proof of \cite[Lemma 13]{Wa1998} holds with \cite[Lemma 5]{Wa1998} replaced by Lemma \ref{lem5} and \cite[Lemma 6]{Wa1998} replaced by Lemma \ref{lem:common}.
\end{proof}

\begin{proof}[Proof of Theorem \ref{thm:xsimpc}]
The $2$-complex \( X_2(\bdel) \) is connected by Proposition \ref{prop8}. Let \( \bp \) be a closed edge-path in \( X_2(\bdel) \) with radius \( m \) around some curve \( \alpha \) of some vertex of $\bp$. If \( m = 0 \), then \( \bp \) is null-homotopic by Lemma \ref{lem12}. The induction argument in the proof of \cite[Theorem 1]{Wa1998} applies with \cite[Lemma 13]{Wa1998} replaced by Lemma \ref{lem13}.
\end{proof}

\section{the $2$-complex of cut-systems and handleslides}\label{y2}
We fix a surface \(\Sigma\) and a sutured compression body $C(\bdel)$ as previously. In this section, we assume that \(\Sigma\) does not contain any distinguished disks. 

Let $\alpha_1$ and $\alpha_2$ be two disjoint curves in $\Sigma$, and fix an embedded arc $a$ from $\alpha_1$ to $\alpha_2$ whose interior is disjoint from $\alpha_1\cup\alpha_2$ and from $\partial \Sigma$. Then a regular neighborhood of the graph $\alpha_1 \cup a \cup \alpha_2$ is a pair of pants with three boundary components: one is isotopic to $\alpha_1$, the other is isotopic to $\alpha_2$, and the third is a new curve $\alpha_1'$, which we call the curve obtained by \emph{handle-sliding} $\alpha_1$ over $\alpha_2$ along the arc $a$.

\begin{definition}[cf. Definition 2.10 in \cite{JTZ}]
Suppose $\bal$ and $\bal'$ are two attaching sets on $\Sigma$. We say that $\bal$ and $\bal'$ are \emph{related by a handleslide}, if there are components $\alpha_1$ and $\alpha_2$ of $\bal$ and a component $\alpha_1'$ of $\bal'$ such that $\alpha_1'$ can be obtained by handle-sliding $\alpha_1$ over $\alpha_2$ along some arc whose
interior is disjoint from $\bal$, and $\bal'=\left(\bal \backslash \alpha_1\right) \cup \alpha_1'$. 

If $\bu$ and $\bv$ are two cut-systems of $C(\bdel)$, then we say that they are related by a handleslide if they have representatives $\bal$ and $\bbe$, respectively, such that $\bal$ and $\bbe$ are related by a handleslide.
\end{definition}

Let \(Y_1(C(\bdel))\) denote the 1-complex whose vertices correspond to the cut-systems of \(C(\bdel)\), and whose edges correspond to handleslides between cut-systems of \(C(\bdel)\) on \(\Sigma\). 
\begin{definition}\label{def:hslide}
A \emph{handleslide loop} is one of the following sequences of cut-systems connected by handleslides. Cases (2)–(5) are illustrated in Figure \ref{slide_square}, and case (6) is shown in Figure \ref{pentagon}.
\begin{enumerate}
\item A \emph{slide triangle}, formed by $\left\iso{\alpha_1, \alpha_2, \vec{\alpha}\right},\left\iso{\alpha_2, \alpha_3, \vec{\alpha}\right}$, and $\left\iso{\alpha_3, \alpha_1, \vec{\alpha}\right}$, where $\alpha_1$, $\alpha_2$, and $\alpha_3$ bound a pair-of-pants; see Figure \ref{slide_tri}.
\item A \emph{Type-I square}, involving four distinct $\alpha$-curves, as in the link of a singularity of type (A1a) in \cite[Figure 21]{JTZ}.
\item A \emph{Type-III square} formed by sliding $\alpha_1$ over $\alpha_2$ and/or $\alpha_3$, as in case (A1b) in \cite[Figure 21]{JTZ}.
\item A \emph{Type-II square} formed by sliding $\alpha_1$ and/or $\alpha_3$ over $\alpha_2$, with $\alpha_1$ and $\alpha_3$ sliding over $\alpha_2$ from opposite sides, as in case (A1c) in \cite[Figure 21]{JTZ}.
\item A \emph{Type-IV square} formed by sliding $\alpha_1$ over $\alpha_2$ in two different ways, approaching $\alpha_2$ from opposite sides, as in case (A1d)  in \cite[Figure 21]{JTZ}.
\item A slide pentagon as in case (A2) in \cite[Figure 21]{JTZ}.
\end{enumerate}

\begin{figure}[h]
{
   \fontsize{12pt}{11pt}\selectfont
   \def\svgwidth{3.2in}
   \begin{center}
   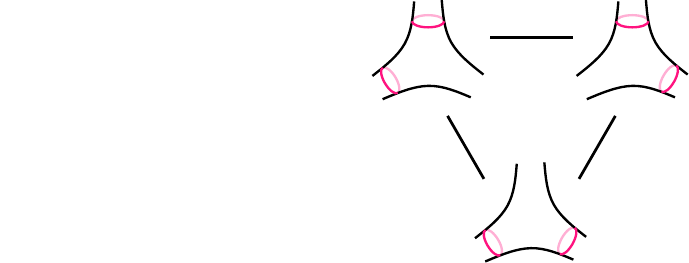
   \end{center}
   \caption{A slide triangle}
   \label{slide_tri}
}
\end{figure}

\begin{figure}[h]
{
   \fontsize{12pt}{11pt}\selectfont
   \def\svgwidth{5.2in}
   \begin{center}
   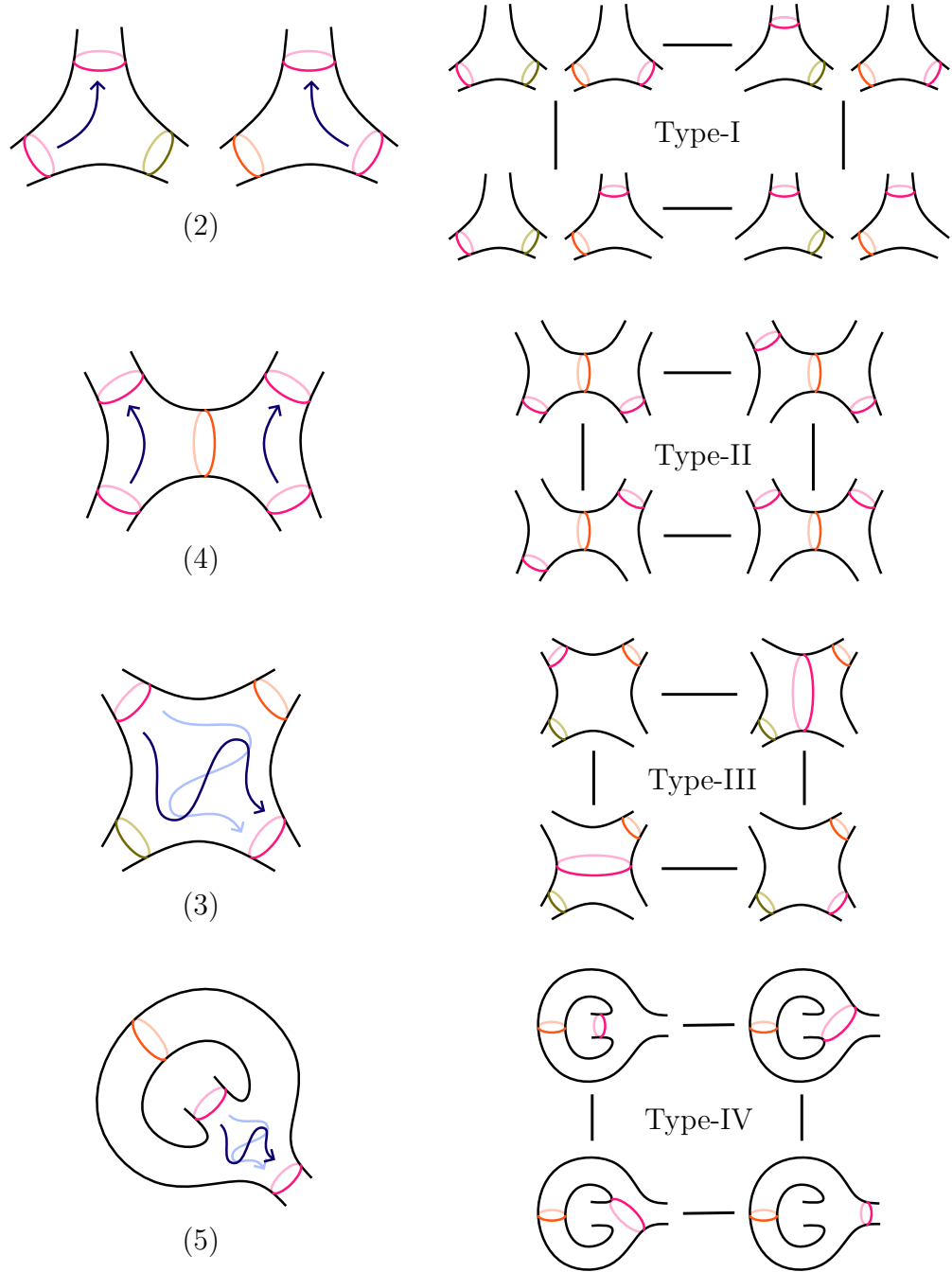
   \end{center}
   \caption{Four types of squares}
   \label{slide_square}
}
\end{figure}

\begin{figure}[h]
{
   \fontsize{12pt}{11pt}\selectfont
   \def\svgwidth{3.2in}
   \begin{center}
   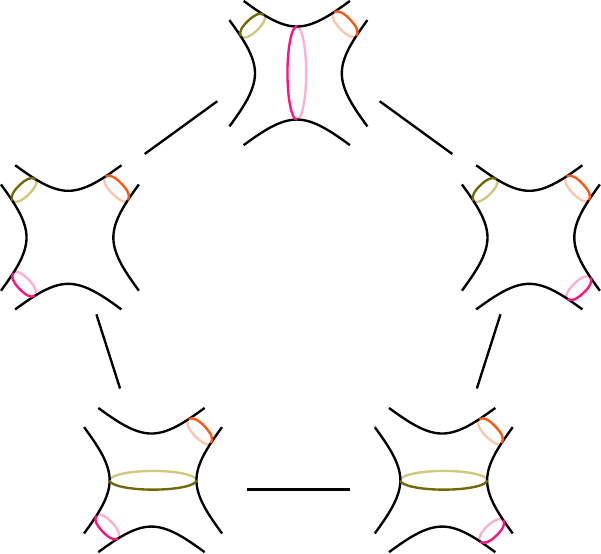
   \end{center}
   \caption{}
   \label{pentagon}
}
\end{figure}

\end{definition}
We attach $2$-cells to $Y_1(C(\bdel))$ along handleslide loops, and denote the resulting $2$-complex by $Y_2(C(\bdel))$. This space is called the \emph{$2$-complex of cut-systems and handleslides} associated to the sutured compression body $C(\bdel)$.
\begin{proposition}\label{prop:y0}
The $2$-complex \(Y_2(C(\bdel))\) is connected.
\end{proposition}

\begin{proof}
Let \( \iso{ \bal } \) and \( \iso{ \bbe } \) be two cut-systems of \( C(\bdel) \). By Proposition \ref{prop:equiv}, we have \( \bal \sim \bdel \sim \bbe \). It follows from \cite[Lemma 2.11]{JTZ} that \( \iso{ \bal } \) and \( \iso{ \bbe } \) are related by a sequence of handleslides.
\end{proof}

We now proceed to prove that \( Y_2(C(\bdel)) \) is simply connected. To prepare for this, we first establish the following lemma and proposition.

\begin{lemma}\label{lem:D}
Suppose $\iso{\bbe}$ is a cut-system of $C(\bdel)$, and \( \alpha \) is a meridian curve of $C(\bdel)$ disjoint from \( \bbe \). Then, \( \Sigma \setminus (\bbe \cup \alpha) \) has a connected component that is a punctured sphere whose boundary does not meet \( \partial \Sigma \) and in which \( \alpha \) appears exactly once.
\end{lemma}

\begin{proof}
Choose disjoint meridian disks \( D_{\alpha} \) and \( D_{\bbe} \) in \( C(\bdel) \), and let \( \Sigma' \) be the surface obtained by compressing \( \Sigma \) along \( D_{\bbe} \). Since \( C(\bdel) \setminus \nu(D_{\bbe}) \) is diffeomorphic to \( \Sigma' \times I \), the curve \( \alpha \) bounds a disk in this space, implying that \( \alpha \) is null-homotopic in \( \Sigma' \). Thus, there exists a disk \( D \subset \Sigma' \) with \( \partial D = \alpha \). Since \( F=D \cap (\Sigma \setminus \bbe) \) is a punctured disk obtained by removing from \( D \) the disks attached during the compression, the connected component of \( \Sigma \setminus (\bbe \cup \alpha) \) containing $F$ has the desired properties.
\end{proof}

\begin{proposition}\label{prop:F}
Let $\alpha_1,\alpha_2$ and $\vec{\beta}$ be disjoint meridian curves on $\Sigma$ such that $\iso{ \alpha_1, \vec{\beta} }$ and $\iso{ \alpha_2, \vec{\beta} }$ are cut-systems of $C(\bdel)$. Then, there exists exactly one connected component \(F\) of \[\Sigma \setminus (\alpha_1 \cup \alpha_2 \cup \vec{\beta})\] such that $g(F) = 0$, \(\partial F \cap \partial \Sigma = \varnothing\) and \(\partial F\) contains $\alpha_1$ and $\alpha_2$, each with multiplicity one.
\end{proposition}

\begin{proof}  
Since \(\alpha_1\) is a meridian curve and \(\iso{ \alpha_2, \vec{\beta} }\) is a cut-system, Lemma \ref{lem:D} implies the existence of a connected surface \(F \subset \Sigma \setminus (\alpha_1 \cup \alpha_2 \cup \vec{\beta})\) such that \(g(F) = 0\), \(\partial F \cap \partial \Sigma = \varnothing\), and \(\alpha_1\) appears in \(\partial F\) with multiplicity one. It remains to show that \(\partial F\) must contain \(\alpha_2\) with multiplicity one.

Gluing \(\Sigma \setminus (\alpha_1 \cup \alpha_2 \cup \vec{\beta})\) along two copies of \(\alpha_2\) results in \(\Sigma \setminus (\alpha_1 \cup \vec{\beta})\). Suppose \(\alpha_2\) does not appear in \(\partial F\). Then, \(F\) is also a connected component of \(\Sigma \setminus (\alpha_1 \cup \vec{\beta})\). However, since \(\iso{ \alpha_1, \vec{\beta} }\) is a cut-system, each connected component of \(\Sigma \setminus (\alpha_1 \cup \vec{\beta})\) must have a nonempty intersection with \(\partial \Sigma\). Since \(\partial F \cap \partial \Sigma = \varnothing\), this results in a contradiction.

Now, suppose \(\partial F\) contains two copies of \(\alpha_2\). Then, \(F\) induces a connected component \(F'\) of \(\Sigma \setminus (\alpha_1 \cup \vec{\beta})\). Since \(F'\) is obtained from \(F\) by gluing along two copies of \(\alpha_2\), we still have that \(F'\) does not intersect \(\partial \Sigma\). Again, this contradicts the fact that \(\iso{ \alpha_1, \vec{\beta} }\) is a cut-system. Thus, we conclude that \(\partial F\) must contain \(\alpha_2\) with multiplicity one.

For uniqueness, suppose there exists two such connected component $F_1,F_2$. Gluing $F_1$ and $F_2$ along $\alpha_1$ and $\alpha_2$, we have that $F'=F_1\cup F_2$ is a connected component of $\Sigma\setminus \bbe$ such that $\partial F'\cap \partial \Sigma = \varnothing$. This contradicts the fact that $\bbe$ is a pre-cut-system.
\end{proof}

As the first step in the proof that \( Y_2(C(\bdel)) \) is simply connected, we apply the \emph{minimal resolution} method, which replaces each edge in \( X_1(C(\bdel)) \) with an edge-path in \( Y_1(C(\bdel)) \). Although this technique was used implicitly in Propositions A.4 and A.5 of \cite{JTZ}, it was not explicitly defined there. For clarity and completeness, we now define it as follows.

Consider an edge \[ \be = \iso{ \alpha_1, \vec{\beta} } \rightarrow \iso{ \alpha_2, \vec{\beta} } \] in \( X_1(C(\bdel)) \). By Proposition \ref{prop:F}, there exists a unique connected component \( F \) of \( \Sigma \setminus (\alpha_1 \cup \alpha_2 \cup \vec{\beta}) \) such that \( F \) has genus zero and its boundary consists of a single copy of \( \alpha_1 \), a single copy of \( \alpha_2 \), and a collection of curves \( \vec{\beta_0} \), where each curve in \( \vec{\beta_0} \) arises from cutting along \( \vec{\beta} \).

Let \( \mathcal{P} \) be a pants decomposition of \( F \). We say that \( \mathcal{P} \) is \emph{regular} if, in each pair of pants, exactly one boundary component belongs to \( \vec{\beta}_0 \). Let \( \partial \mathcal{P} \) denote the set of boundary curves of the pairs of pants in \( \mathcal{P} \) that are not contained in \( \partial F \). Since each \( \gamma_i \in \partial \mathcal{P} \) separates \( \alpha_1 \) and \( \alpha_2 \), we specify an ordering \( \Gamma(\mathcal{P}) = (\gamma_1, \dots, \gamma_k) \) on $\partial \cP$ by requiring that for every \( j < i \), the curve \( \gamma_j \) lies in the connected component of \( F \setminus \gamma_i \) that contains \( \alpha_1 \) (See Figure \ref{regular}).

\begin{figure}[h]
{
   \fontsize{12pt}{11pt}\selectfont
   \def\svgwidth{5.8in}
   \begin{center}
   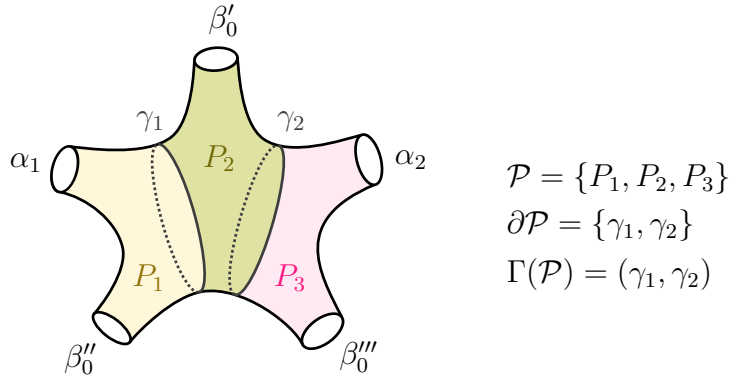
   \end{center}
   \caption{An example of regular pants decomposition}
   \label{regular}
}
\end{figure}

For each regular decomposition $P$, we define an edge-path $\bp$ in $Y_1(C(\bdel))$ to be
\[\bp =  \iso{ \alpha_1, \vec{\beta} } \rightarrow \iso{ \gamma_1, \vec{\beta} } \rightarrow \dots \rightarrow \iso{ \gamma_k, \vec{\beta} } \rightarrow \iso{ \alpha_2, \vec{\beta} },\]
where each edge is the elementary handleslide in the corresponding pair of pants over the single boundary component that belongs to $\vec{\beta_0}$. We say that \( \bp \) is the \emph{minimal resolution} of \( \be \) induced by the regular pants decomposition $\cP$.  

\begin{remark}
From this, we can deduce that the 2-complex \( Y_2(C(\bdel)) \) is connected, i.e., any two vertices are related by a sequence of handleslides as follows. First, since \( X_2(C(\bdel)) \) is connected, we choose an edge-path \( \bp \) in \( X_2(C(\bdel)) \) connecting the two vertices. Then, we replace each edge of \( \bp \) with its minimal resolution, which is an edge-path in \( Y_2(C(\bdel)) \). Finally, concatenating these minimal resolutions gives an edge-path in \( Y_2(C(\bdel)) \) connecting the original vertices.
\end{remark}

\begin{theorem}\label{thm:min}
All minimal resolutions of \(\be = \iso{ \alpha_1, \vec{\beta} } \rightarrow \iso{ \alpha_2, \vec{\beta} }\) are homotopic in $Y_2(C(\bdel))$. 
\end{theorem}

The proof will go by induction on $|\partial F|$. For the base case $|\partial F|=3$, the surface $F$ is a pair of pants with $\partial F = \alpha_1\cup \alpha_2 \cup \beta_0$. Then, the curves $\alpha_1$ and $\alpha_2$ are related by an elementary handleslide over $\beta_0$, which determines the unique minimal resolution of $\be$ in $Y_1(C(\bdel))$. Now, consider the case where $F$ has more than three boundary components, i.e., $|\partial F| = n > 3$.

Consider two regular pants decompositions \( \cP_1 \) and \( \cP_2 \) of \( F \), with \( \bp_1 \) and \( \bp_2 \) denoting the minimal resolutions induced by \( \cP_1 \) and \( \cP_2 \), respectively. For $i\in \{1,2\}$, let \( \gamma_i \) be the first curve of \( \Gamma(\cP_i) \), and let \( T_i \) be the pair of pants whose boundary includes \( \alpha_1 \) and \( \gamma_i \). Denote the third boundary component of \( T_i \) by \( \rho_i \), and choose a properly-embedded arc \( c_i \subset T_i \) such that \( \gamma_i \) is obtained by handle-sliding \( \alpha_1 \) over \( \rho_i \) along \( c_i \). See Figure \ref{bridge} for an example.

\begin{figure}[h]
{
   \fontsize{12pt}{12pt}\selectfont
   \def\svgwidth{2.1in}
   \begin{center}
   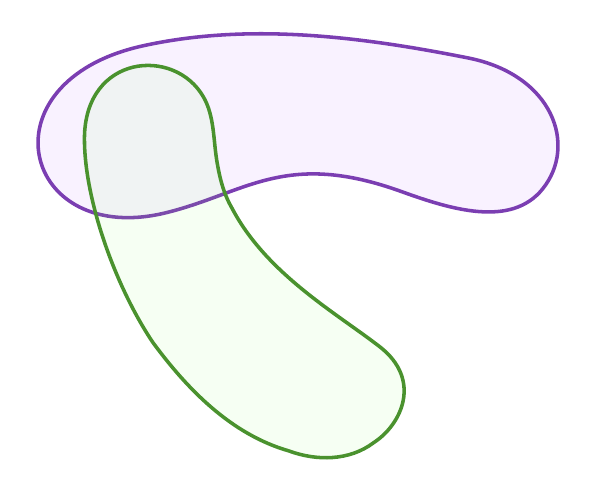
   \end{center}
   \caption{}
   \label{bridge}
}
\end{figure}

\begin{lemma}\label{lem:bridge}
Let \( \cP_1 \) and \( \cP_2 \) be two regular pants decompositions of \( F \). If \( \rho_1 \neq \rho_2 \) and \( c_1 \cap c_2 \subset \partial F \), then \( \bp_1 \) and \( \bp_2 \) are homotopic in \( Y_2(C(\bdel)) \).
\end{lemma}

\begin{proof}
Consider a tubular neighborhood $T_3$ of the graph $\rho_1 \cup c_1 \cup \alpha_1 \cup c_2 \cup \rho_2$ in $F$, and let $\gamma_3 = \partial T_3$. By isotoping \( \gamma_i \) while keeping the other curves in \( \Gamma(\cP_i) \) fixed, we may assume that \( T_i \) is contained in \( T_3 \) and that \( \gamma_1 \) and \( \gamma_2 \) intersect transversely at two points. See Figure \ref{bridge}.

Let $\bp_3$ be the minimal resolution of the edge $\iso{ \gamma_3,\vec{\beta}} \to \iso{ \alpha_2, \vec{\beta}}$ induced by some pants decomposition of $F \setminus T_3$, and let $\bp_i'$ be the minimal resolution induced by $\cP_i \setminus \{T_i\}$. This construction produces the following graph in $Y_1(C(\bdel))$:
\[
\begin{tikzcd}[column sep=small]
                                                            & {\iso{ \gamma_1,\vec{\beta}}} \arrow[rd] \arrow[rrrrd, "\bp_1'", bend left]   &                                                            &  &  &                                       \\
{\iso{ \alpha_1,\vec{\beta}}} \arrow[rd] \arrow[ru] &                                                                                       & {\iso{ \gamma_3,\vec{\beta}}} \arrow[rrr, "\bp_3"] &  &  & {\iso{ \alpha_2,\vec{\beta}}}, \\
                                                            & {\iso{ \gamma_2,\vec{\beta}}} \arrow[ru] \arrow[rrrru, "\bp_2'"', bend right] &                                                            &  &  &                                      
\end{tikzcd}
\]
Since \( |\partial (F\setminus T_i)| = n - 1 \) for $i=1$ or $2$, the induction hypothesis implies that all minimal resolutions of the edge \( \iso{ \gamma_i, \vec{\beta} } \to \iso{ \alpha_2, \vec{\beta} } \) are homotopic in \( Y_2(C(\bdel)) \). Consequently, the two loops on the right are null-homotopic in \( Y_2(C(\bdel)) \). Moreover, since the square on the left is either Type-III or Type-IV, it is null-homotopic in \( Y_2(C(\bdel)) \). Together, we conclude that \( \bp_1 \) and \( \bp_2 \) are homotopic in \( Y_2(C(\bdel)) \).
\end{proof}

\begin{lemma}\label{lem:rho}
Let \( \cP_1 \) and \( \cP_2 \) be two regular pants decompositions of \( F \) such that \( \rho_1 = \rho_2 = \rho \). If \( c_1 \) and \( c_2 \) do not intersect in the interior of \( F \), then \( \bp_1 \) and \( \bp_2 \) are homotopic in \( Y_2(C(\bdel)) \).
\end{lemma}

\begin{figure}[h]
{
   \fontsize{12pt}{11pt}\selectfont
   \def\svgwidth{1.8in}
   \begin{center}
   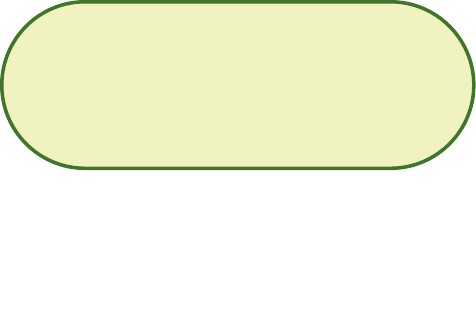
   \end{center}
   \caption{}
   \label{rho}
}
\end{figure}

\begin{proof}
Since \( c_1 \) and \( c_2 \) may intersect only along the boundary of \( F \), cutting \( F \) along \( c_1 \cup c_2 \) results in two punctured disks. In the disk that contains a boundary component \( \rho' \in \vec{\beta_0} \) with \( \rho' \neq \rho \), choose a properly-embedded arc \( c' \) connecting \( \alpha_1 \) to \( \rho' \). See Figure \ref{rho}.

Now, let \( T' \) be a tubular neighborhood of \( \alpha_1 \cup c' \cup \rho' \). Complete \( T' \) into a regular pants decomposition \( \cP' \) of \( F \) by choosing a regular pants decomposition for \( F \setminus T' \). Denote by \( \bp' \) the minimal resolution induced by \( \cP' \). By Lemma \ref{lem:bridge}, both \( \bp_1 \) and \( \bp_2 \) are homotopic to \( \bp' \) in \( Y_2(C(\bdel)) \), which implies that \( \bp_1 \) and \( \bp_2 \) are homotopic in \( Y_2(C(\bdel)) \).
\end{proof}

For the general cases, we first prove that any two arcs connecting $\alpha_1$ and $\rho$ are related by a sequence of arcs such that consecutive curves are disjoint.

\begin{lemma}\label{lem:inter}
Let $S$ be a cylinder decorated with finitely many distinguished disks. Suppose $\lambda$ and $\lambda'$ are two properly-embedded arcs connecting two boundary components of $S$. Then, there exists a sequence of properly-embedded arcs $\Lambda = (\lambda_1, \dots, \lambda_m)$, each connecting two boundary components of $S$, such that $\lambda_1 = \lambda$, $\lambda_m = \lambda'$, and consecutive arcs in the sequence are disjoint.
\end{lemma}

\begin{figure}[h]
{
   \fontsize{12pt}{11pt}\selectfont
   \def\svgwidth{2.9in}
   \begin{center}
   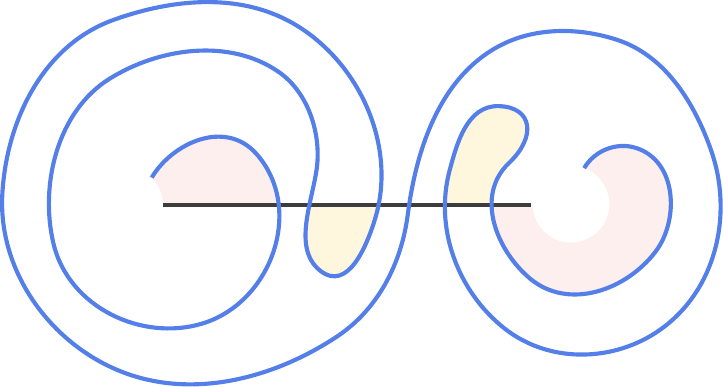
   \end{center}
   \caption{}
   \label{bigon}
}
\end{figure}

\begin{proof}
We construct \( \Lambda = (\lambda_1, \dots, \lambda_m) \) recursively by applying elementary moves that reduce interior intersections with $\lambda'$. Suppose $\lambda$ and $\lambda'$ are disjoint, then we set $m=2$ and let $\Lambda = (\lambda, \lambda')$.

Now, we consider the case where $\lambda$ and $\lambda'$ intersect (see Figure \ref{bigon} for example). Suppose we have constructed a sequence of arcs \((\lambda_2, \dots, \lambda_i) \) such that $\lambda_j$ and $\lambda_{j+1}$ are disjoint for $1\le j \le i-1$. If \( \lambda_i \) and \( \lambda' \) form a bigon or a half-bigon on \( S \), we isotope \( \lambda_i \) across the minimal (half-)bigon to obtain a new arc \( \lambda_{i+1} \) that intersects \( \lambda' \) at fewer points (see Figure \ref{gons}), such that  \( \lambda_i \) and \( \lambda_{i+1} \) do not intersect in the interior of $S$. 

\begin{figure}[h]
{
   \fontsize{12pt}{11pt}\selectfont
   \def\svgwidth{4.9in}
   \begin{center}
   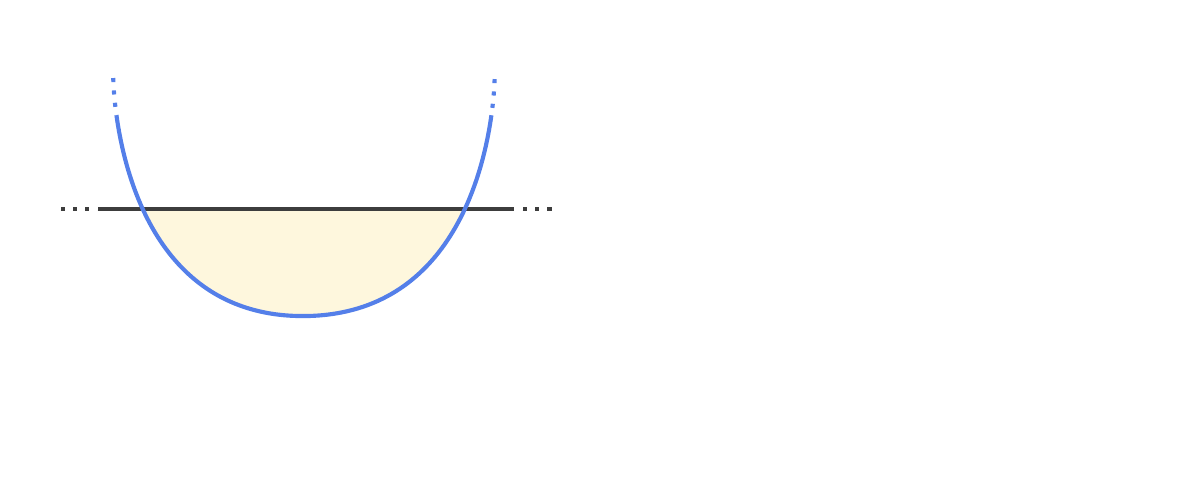
   \end{center}
   \caption{}
   \label{gons}
}
\end{figure}

Otherwise, it follows from the bigon criterion for arcs (cf. Section 1.2.7 in \cite{PMCG}) that if \( \lambda_i \) and \( \lambda' \) do not form bigons or half-bigons on $S$, then they are in minimal position under general isotopies (allowing moving end points along the boundary) and thus disjoint on $S$. Thus, we set $m=i+1$ and let $$\Lambda = (\lambda, \lambda_2,\dots, \lambda_{i}, \lambda'),$$ which satisfies the desired conditions.
\end{proof}

We enhance Lemma \ref{lem:rho} by interpolating between the arcs $c_1$ and $c_2$ in $F$ using Lemma \ref{lem:inter}.

\begin{proposition}\label{prop:rhorho}
Let \( \cP_1 \) and \( \cP_2 \) be two regular pants decompositions of \( F \). If \( \rho_1 = \rho_2 = \rho \), then \( \bp_1 \) and \( \bp_2 \) are homotopic in \( Y_2(C(\bdel)) \).
\end{proposition}

\begin{proof}
We first isotope the arcs \( c_1 \) and \( c_2 \) so that their endpoints are distinct. We treat all components of \( \partial F \) other than \( \alpha_1 \) and \( \rho \) as distinguished disks (by gluing disks to the punctured sphere \( F \) along those boundary components), and apply Lemma \ref{lem:inter} to the resulting cylinder with boundary \( \alpha_1 \cup \rho \). It follows that there exists a sequence of arcs \( (\lambda_1, \dots, \lambda_m) \), each connecting $\alpha_1$ and $\rho$, such that \( \lambda_1 = c_1 \), \( \lambda_m = c_2 \), and consecutive arcs in the sequence are disjoint.

Denote by $\delta_i$ the curve obtained by handle-sliding $\alpha_1$ over $\rho$ along the arc $\lambda_i$ for $1 < i < m$, and let $\cP_i$ be a regular pants decomposition such that $\delta_i$ is the first curve in $\Gamma(\cP_i)$. Denote by $\bp_i$ the minimal resolution induced by $\cP_i$ for $1 < i < m$. Then, Lemma \ref{lem:rho} implies that $\bp_1$ and $\bp_2$ are homotopic in $Y_2(C(\bdel))$ through minimal resolutions $\bp_i$ with $1<i<m$.
\end{proof}

\begin{proof}[Proof of Theorem \ref{thm:min}]
Let \( \cP_1 \) and \( \cP_2 \) be two regular pants decompositions of \( F \). We are left with the case where $\rho_1\neq \rho_2$ and the arcs $c_1$ and $c_2$ intersect in the interior of $F$. Since $F \setminus c_1$ is connected, we pick an arc $c_3$ connecting $\alpha_1$ and $\rho_2$ such that $c_1\cap c_3 \subset \partial F$. Denote by $\gamma_3$ the curve obtained by handle-sliding $\alpha_1$ over $\rho_2$ along the arc $c_3$, and let $\cP_3$ be a regular pants decomposition such that $\gamma_3$ is the first curve in $\Gamma(\cP_3)$. 

Denote by $\bp_i$ the minimal resolution induced by $\cP_i$ for $i\in\{1,2,3\}$. By Lemma \ref{lem:bridge}, $\bp_1$ is homotopic to $\bp_3$. By Proposition \ref{prop:rhorho}, $\bp_2$ is homotopic to $\bp_3$. Thus, we conclude that $\bp_1$ is homotopic to $\bp_2$ in $Y_2(C(\bdel))$.
\end{proof}

\begin{theorem}\label{thm:simpc}
The $2$-complex \(Y_2(C(\bdel))\) is simply-connected.
\end{theorem}

\begin{proof}
First, we show that if we convert edges $\be_0$, $\be_1$ and $\be_2$ of a triangle $\Delta$ in $X_2(C(\bdel))$ into minimal resolutions in $Y_1(C(\bdel))$, we obtain a closed edge-path that is null-homotopic in $Y_2(C(\bdel))$. Let $\bv_i$ be the vertex of $\Delta$ opposite the edge $\be_i$. We distinguish two cases based on the type of the triangle $\Delta$. 

Suppose the triangle $\Delta$ is of Type I, with vertices $\bv_i = \iso{ \alpha_i,\vec{\alpha}}$ for $i\in\mathbb{Z}_3$. By Proposition \ref{prop:F}, for each $i\in\mathbb{Z}_3$, there exists a connected component $F_i$ of the surface $\Sigma\setminus (\alpha_{i-1}\cup\alpha_{i+1}\cup \vec{\alpha})$ such that $g(F_i) = 0$, $\partial F_i\cap\partial \Sigma = \varnothing$, and $\alpha_{i-1}\cup\alpha_{i+1} \subset \partial F_i$, each with multiplicity one. Suppose \( F_i \cap \alpha_i = \varnothing \) for each \( i \in \mathbb{Z}_3 \). Then, the subsurfaces \( F_0, F_1, \) and \( F_2 \) are distinct connected components of  
\(
\Sigma \setminus (\alpha_0 \cup \alpha_1 \cup \alpha_2 \cup \vec{\alpha}),
\)
and their union  
\(
F := F_0 \cup F_1 \cup F_2
\)
forms a connected component of \( \Sigma \setminus \vec{\alpha} \) with \( \partial F \cap \partial \Sigma = \varnothing \). However, since \( \vec{\alpha} \) is a pre-cut system, every connected component of \( \Sigma \setminus \vec{\alpha} \) must intersect \( \partial \Sigma \), which contradicts our assumption. Therefore, \( \alpha_i \subset F_i \), and
\[
F_i = F_{i-1} \cup F_{i+1}
\]
for some \( i \in \mathbb{Z}_3 \).

Choose a regular pants decomposition \( \cP_{i-1} \) for the subsurface \( F_{i-1} \) and \( \cP_{i+1} \) for the subsurface \( F_{i+1} \). Then, convert \( \be_{i-1} \) into the minimal resolution \( \bp_{i-1} \) induced by \( \cP_{i-1} \) and similarly convert \( \be_{i+1} \) into the minimal resolution \( \bp_{i+1} \) induced by \( \cP_{i+1} \). Since the union \( \cP_{i-1} \cup \cP_{i+1} \) forms a regular pants decomposition \( \cP_i \) of \( F_i = F_{i-1} \cup F_{i+1} \), we can convert the edge \( \be_i \) into the concatenation of \( \bp_{i-1} \) and \( \bp_{i+1} \), which is the minimal resolution $\bp_i$ induced by \( \cP_i \).  By construction, the closed edge-path formed by $\bp_{i-1}$, $\bp_{i}$ and $\bp_{i+1}$ is null-homotopic in \(Y_2(C(\bdel))\). By Theorem \ref{thm:min}, any two minimal resolutions are homotopic in \(Y_2(C(\bdel))\), so it suffices to pick any particular one. (Compare \cite[Proposition A.5]{JTZ}.)

Next, suppose the triangle $\Delta$ is of Type II, with vertices $\bv_i=\iso{ \alpha_{i-1},\alpha_{i+1},\vec{\alpha}}$ for $i\in \mathbb{Z}_3$. Since $\iso{ \alpha_0, \alpha_2, \vec{\alpha}}$ and $\iso{ \alpha_1, \alpha_2, \vec{\alpha}}$ are both cut-systems of $C(\bdel)$, by Proposition \ref{prop:F}, there exists a connected component $F$ of the surface $\Sigma\setminus (\alpha_{0}\cup\alpha_{1}\cup \alpha_2\cup\vec{\alpha})$ such that $g(F)$=0, $\partial F \cap \partial \Sigma = \varnothing$, and $\partial F$ contains $\alpha_0$ and $\alpha_1$, each with multiplicity one. Since \( \iso{ \alpha_0, \alpha_1, \vec{\alpha} } \) is also a cut-system of \( C(\bdel) \), every connected component of \( \Sigma \setminus (\alpha_0 \cup \alpha_1 \cup \vec{\alpha}) \) must intersect \( \partial \Sigma \). Moreover, the surface \( \Sigma \setminus (\alpha_0 \cup \alpha_1 \cup \vec{\alpha}) \) can be obtained from \( \Sigma \setminus (\alpha_0 \cup \alpha_1 \cup \alpha_2 \cup \vec{\alpha}) \) by gluing along two copies of \( \alpha_2 \). If \( \partial F \) does not contain \( \alpha_2 \) with multiplicity one, then \( F \) induces a connected component \( F' \) of \( \Sigma \setminus (\alpha_0 \cup \alpha_1 \cup \vec{\alpha}) \) such that \( \partial F' \cap \partial \Sigma = \varnothing \), contradicting the fact that every component must intersect \( \partial \Sigma \). Thus, \( \alpha_i \) must appear in \( \partial F \) with multiplicity one for each \( i \in \mathbb{Z}_3 \). 

If there are no other components of $\partial F$, the cut-systems $\iso{ \alpha_{0},\alpha_{1},\vec{\alpha}}$,$\iso{ \alpha_{1},\alpha_{2},\vec{\alpha}}$ and $\iso{ \alpha_{0},\alpha_{2},\vec{\alpha}}$ form a slide triangle. Otherwise, we prove the desired result by induction on $k=|\partial F|-3$. Let \( \alpha_0' \) be the curve obtained from \( \alpha_0 \) by sliding it over one of the other \( k \) boundary components of \( \partial F \). Let \( F' \) denote the surface obtained by removing from \( F \) the pair of pants determined by this handleslide. See Figure \ref{gloves}.

\begin{figure}[h]
{
   \fontsize{11pt}{11pt}\selectfont
   \def\svgwidth{4.3in}
   \begin{center}
   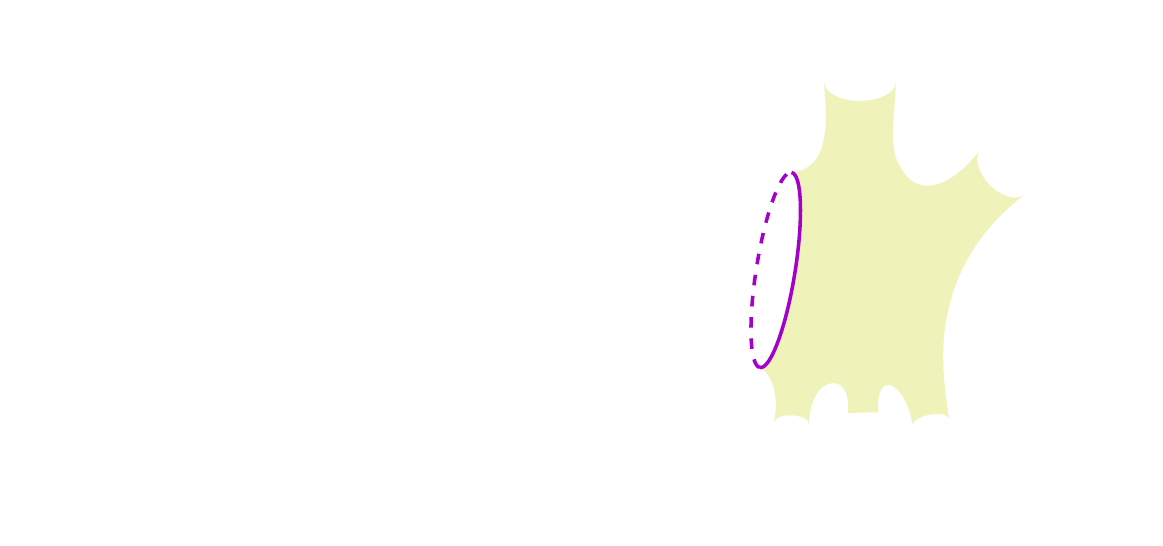
   \end{center}
   \caption{}
   \label{gloves}
}
\end{figure}

By induction, the triangular loop  
\[
\Delta' = \iso{ \alpha_{0}', \alpha_{1}, \vec{\alpha}} \rightarrow \iso{ \alpha_{1}, \alpha_{2}, \vec{\alpha}} \rightarrow \iso{ \alpha_{0}', \alpha_{2}, \vec{\alpha}} \rightarrow \iso{ \alpha_{0}', \alpha_{1}, \vec{\alpha}}
\]
can be decomposed into a sum of 2-cells in \( Y_2(C(\bdel)) \), such that the edges of \( \Delta' \) are resolved into minimal resolutions (induced by pants decompositions of \( F' \)); see Figure \ref{decomp}. Since all minimal resolutions of the edge 
\[ \be = \iso{ \alpha_{1}, \alpha_{0}', \vec{\alpha}} \rightarrow \iso{ \alpha_{2}, \alpha_{0}', \vec{\alpha}} \]
are homotopic in \( Y_2(C(\bdel)) \) (cf. Theorem \ref{thm:min}), we may assume that the first edge in the minimal resolution \( \bp_\be \) of \( \be \) is a handleslide of \( \alpha_1 \) over \( \alpha_0' \). Using a pentagon and a row of Type-I or Type-II squares, we can homotope the edge-path \( \bp_\be \) in \( Y_2(C(\bdel)) \) to a minimal resolution of \( \iso{ \alpha_{1}, \alpha_{0}, \vec{\alpha}} \rightarrow \iso{ \alpha_{2}, \alpha_{0}, \vec{\alpha}} \), as shown in Figure \ref{pinkpan}. Gluing Figure \ref{decomp} and Figure \ref{pinkpan} along $\bp_\be$, we get a decomposition of $\Delta$ in \( Y_2(C(\bdel)) \) (cf. Figure 64 in \cite{JTZ}).

\begin{figure}[h]
{
   \fontsize{11pt}{11pt}\selectfont
   \def\svgwidth{4.0in}
   \begin{center}
   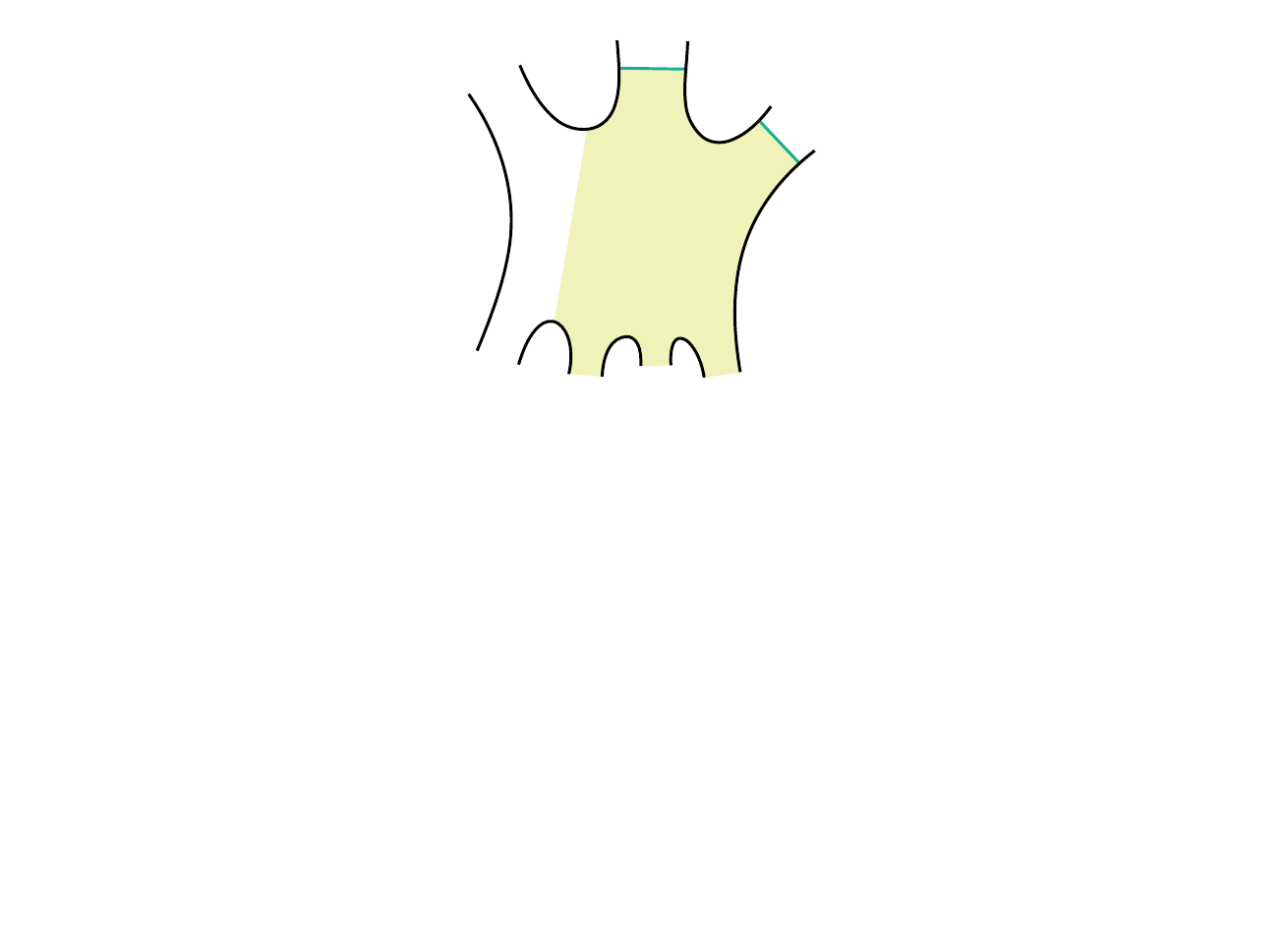
   \end{center}
   \caption{}
   \label{decomp}
}
\end{figure}

\begin{figure}[h]
{
   \fontsize{11pt}{11pt}\selectfont
   \def\svgwidth{6.0in}
   \begin{center}
   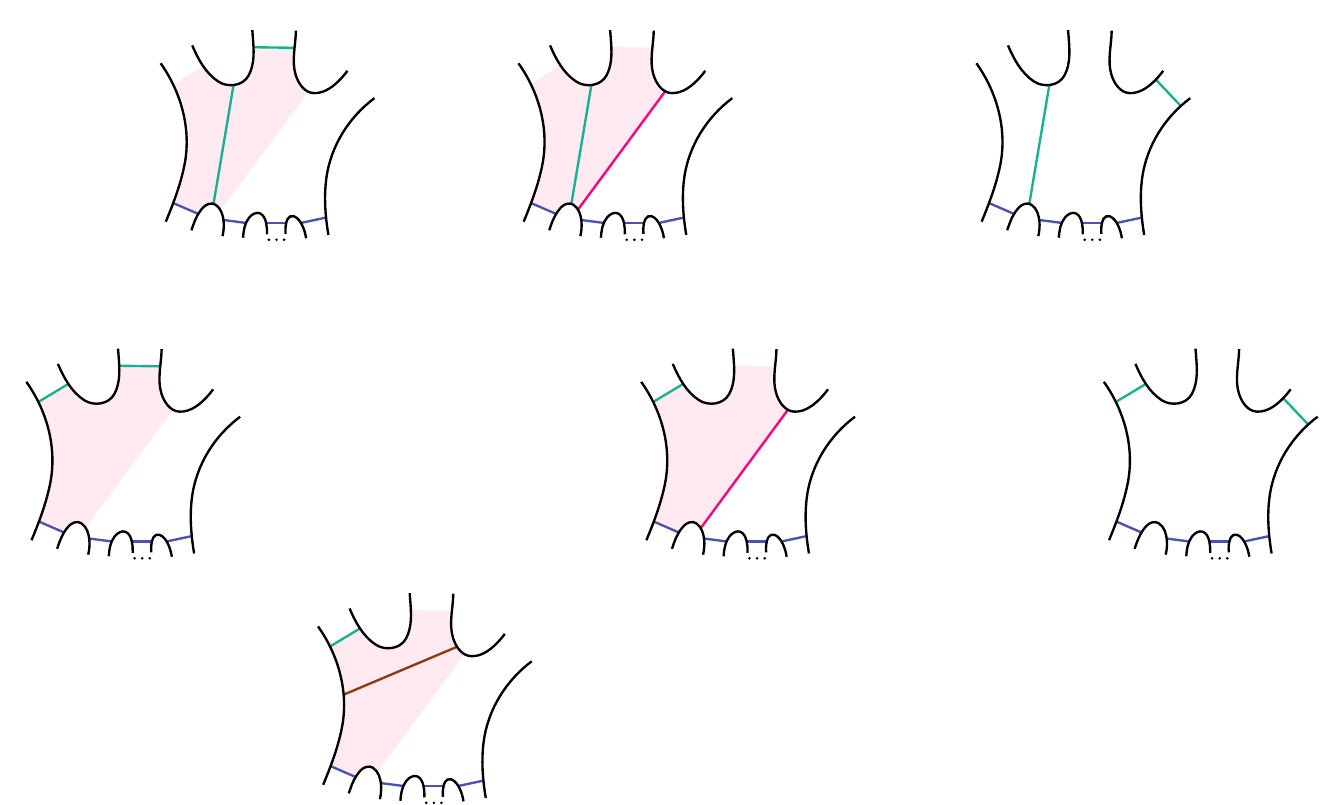
   \end{center}
   \caption{}
   \label{pinkpan}
}
\end{figure}

Secondly, we show that we converting edges of a self-separated square \(\Theta\) in \(X_2(C(\bdel))\) into minimal resolutions in \(Y_1(C(\bdel))\) results in a closed edge-path that is null-homotopic in \(Y_2(C(\bdel))\). Let the square \(\Theta\) be \(\iso{ \alpha_0, \beta_0, \vec{\alpha} } \rightarrow \iso{ \alpha_0, \beta_1, \vec{\alpha} } \rightarrow \iso{ \alpha_1, \beta_1, \vec{\alpha} } \rightarrow \iso{ \alpha_1, \beta_0, \vec{\alpha} } \rightarrow \iso{ \alpha_0, \beta_0, \vec{\alpha} }.\) Since $\Theta$ is self-separated, there exists two distinct connected components \(F_{\alpha}, F_{\beta}\subset\Sigma \setminus \vec{\alpha}\) such that \(\alpha_0, \alpha_1 \subset F_{\alpha}\) and \(\beta_0,\beta_1 \subset F_{\beta}\). Pick a sequence of handleslides \(\alpha(n)\) from \(\alpha_0\) to \(\alpha_1\) within \(F_{\alpha}\), and another sequence of handleslides \(\beta(m)\) from \(\beta_0\) to \(\beta_1\) within \(F_{\beta}\). Now, we decompose each edge of \(\Theta\) into minimal reolutions by interpolations using \(\alpha(n)\) and \(\beta(m)\), and then decompose this resolution of \(\Theta\)  into \(m \times n\) squares of either Type-I or Type-II in \(Y_2(C(\bdel))\). Again, by Theorem \ref{thm:min}, it suffices to show that this particular resolution is null-homotopic in  \(Y_2(C(\bdel))\).

Finally, we show that any closed edge-path \( \bp \) in \( Y_2(C(\bdel)) \) is null-homotopic as follows. Since \( X_2(C(\bdel)) \) is simply connected, we can decompose \( \bp \) into a sum of 2-cells in \( X_2(C(\bdel)) \). We then replace each edge with its minimal resolution. For each 2-cell, the resulting closed edge-path formed by these minimal resolutions is null-homotopic in \( Y_2(C(\bdel)) \). It follows that \( \bp \) is a sum of 2-cells in \( Y_2(C(\bdel)) \), and hence null-homotopic (see Figure \ref{sea_leaf}).
\end{proof}

\section{tight Heegaard invariants for sutured manifolds}\label{strong}
A \emph{sutured diagram} is a triple \((\Sigma, \bal, \bbe)\), where \(\Sigma\) is a compact oriented surface with boundary, and \(\bal\) and \(\bbe\) are two attaching sets on \(\Sigma\). Each sutured diagram \((\Sigma, \bal, \bbe)\) determines a \emph{sutured manifold} \((M, \gamma)\) by gluing the sutured compression bodies \(C(\bal)\) and \(C(\bbe)\) along \(\Sigma\),
\[
(M, \gamma) = \left(-C(\bal) \cup_{\Sigma} C(\bbe), \partial \Sigma \times [-1,1]\right),
\]
with sutures \(s(\gamma) = \partial \Sigma \times \{0\}\). A \emph{balanced sutured manifold} is a proper sutured manifold \((M, \gamma)\) such that \(\chi(R_{+}(\gamma)) = \chi(R_{-}(\gamma))\). Equivalently,  by \cite[Proposition 2.9]{J}, it is a proper sutured manifold that has a sutured diagram \((\Sigma, \bal, \bbe)\) with \(|\bal| = |\bbe|\).

A \emph{sutured isotopy diagram} is a triple \((\Sigma, \iso{ \bal }, \iso{ \bbe })\), where \((\Sigma, \bal, \bbe)\) is a sutured diagram (cf. \cite[Definition 2.13]{JTZ}). If \(\bal'\) and \(\bbe'\) are isotopic to \(\bal\) and \(\bbe\), respectively, then the sutured diagrams \((\Sigma, \bal, \bbe)\) and \((\Sigma, \bal', \bbe')\) define diffeomorphic sutured manifolds (relative to \(\Sigma\)). Therefore, we denote by \( S(H) \) the diffeomorphism type of the sutured manifold associated to the sutured isotopy diagram \( H \).

Following the convention in \cite[Appendix A]{JTZ}, we use graphs of sutured isotopy diagrams to define strong Heegaard invariants. First, we recall relations between sutured isotopy diagrams as follows. 

\begin{definition}[cf. Definition 2.17 in \cite{JTZ}]
Let \(\left(\Sigma, \iso{\bal}, \iso{\bbe}\right)\) and \(\left(\Sigma', \iso{\bal'}, \iso{\bbe'}\right)\) be sutured isotopy diagrams. We say they are \emph{\(\alpha\)-equivalent} if \(\Sigma = \Sigma'\), \(\iso{\bbe} = \iso{\bbe'}\), and \(\iso{\bal} \sim \iso{\bal'}\). Similarly, they are \emph{\(\beta\)-equivalent} if \(\Sigma = \Sigma'\), \(\iso{\bal} = \iso{\bal'}\), and \(\iso{\bbe} \sim \iso{\bbe'}\).
\end{definition}

\begin{definition}[cf. Definition 2.18 in \cite{JTZ}]
The sutured diagram $\left(\Sigma_2, \bal_2, \bbe_2\right)$ is obtained from the diagram $\left(\Sigma_1, \bal_1, \bbe_1\right)$ by a \emph{stabilization} if
\begin{itemize}
\item there is a disk $D \subset \Sigma_1$ and a punctured torus $T \subset \Sigma_2$ such that
$$
\Sigma_1 \backslash D=\Sigma_2 \backslash T,
$$
\item $\bal_1=\bal_2 \cap\left(\Sigma_2 \backslash T\right)$ and $\bbe_1=\bbe_2 \cap\left(\Sigma_2 \backslash T\right)$,
\item $\bal_2 \cap T$ and $\bbe_2 \cap T$ are simple closed curves that intersect each other transversely in a single point.
\end{itemize}
In this case, we also say that $\left(\Sigma_1, \bal_1, \bbe_1\right)$ is obtained from $\left(\Sigma_2, \bal_2, \bbe_2\right)$ by a \emph{destabilization}. For an illustration, see Figure 3 in \cite{JTZ}. 

Let $H_1$ and $H_2$ be sutured isotopy diagrams. We say that $H_2$ is obtained from $H_1$ by a \emph{(de)stabilization} if they have representatives  $\left(\Sigma_2, \bal_2, \bbe_2\right)$ and $\left(\Sigma_1, \bal_1, \bbe_1\right)$, respectively, such that $\left(\Sigma_2, \bal_2, \bbe_2\right)$ is obtained from $\left(\Sigma_1, \bal_1, \bbe_1\right)$ by a (de)stabilization.
\end{definition}

Let \(\cG\) be the graph whose vertices \(|\cG|\) consist of sutured isotopy diagrams. The edges \(\cG(H_1, H_2)\) connecting isotopy diagrams \(H_1, H_2 \in |\cG|\) are the union of four sets:
\[
\cG(H_1, H_2) = \cG_\alpha(H_1, H_2) \cup \cG_\beta(H_1, H_2) \cup \cG_{\text{stab}}(H_1, H_2) \cup \cG_{\text{diff}}(H_1, H_2).
\]
The set \(\cG_\alpha(H_1, H_2)\) (resp. \(\cG_\beta(H_1, H_2)\)) consists of a single arrow if \(H_1\) and \(H_2\) are \(\alpha\)-equivalent (resp. \(\beta\)-equivalent), and is empty otherwise. The set \(\cG_{\text{stab}}(H_1, H_2)\) consists of a single arrow if \(H_2\) is obtained from \(H_1\) by a stabilization or a destabilization, and is empty otherwise. Finally, \(\cG_{\text{diff}}(H_1, H_2)\) consists of all diffeomorphisms from \(H_1\) to \(H_2\). Let \(\cG'\) be the graph defined similarly to \(\cG\) with \(\alpha/\beta\)-equivalence replaced by \(\alpha/\beta\)-handleslide. Since every handleslide is an \(\alpha\)-equivalence or a \(\beta\)-equivalence, \(\cG'\) is a subgraph of \(\cG\).

Let \(\cS\) be a set of diffeomorphism types of sutured manifolds, and let \(\mathcal{C}\) be any category. We denote by \(\cG(\cS)\) the full subgraph of \(\cG\) whose vertices are sutured isotopy diagrams \(H\) satisfying \(S(H) \in \cS\). Similarly, \(\cG'(\cS)\) is the full subgraph of \(\cG'\) spanned by sutured isotopy diagrams \(H\) for which \(S(H) \in \cS\). 

A \emph{weak Heegaard invariant} of \(\cS\) is a morphism of graphs $F: \cG(\cS) \rightarrow \mathcal{C}$ such that, for every arrow $e$ of $\cG(\cS)$, the image $F(e)$ is an isomorphism (cf. \cite[Definition 2.24]{JTZ}). A \emph{strong Heegaard invariant} of \( \cS \) is a weak Heegaard invariant \( F: \cG(\cS) \rightarrow \mathcal{C} \) that satisfies the following axioms: functoriality, commutativity, continuity, and handleswap invariance (cf. \cite[Definition 2.32]{JTZ}).

\begin{definition}
A \emph{loose Heegaard invariant} of \( \cS \) is a morphism of graphs $F:\cG'(\cS)\rightarrow \mathcal{C}$ such that every arrow $e$ of $\cG'(\cS)$, the image $F(e)$ is an isomorphism. A \emph{tight Heegaard invariant} of \( \cS \) is a loose Heegaard invariant $F:\cG'(\cS)\rightarrow \mathcal{C}$ such that
\begin{itemize}
\item $F$ commutes along each handleslide loop from Definition \ref{def:hslide};
\item $F$ satisfies the second half of the functoriality axiom (\cite[Definition 2.32 (1)]{JTZ}), commutativity axiom (\cite[Definition 2.32 (2)]{JTZ}), continuity axiom(\cite[Definition 2.32 (3)]{JTZ}), and handleswap invariance axiom (\cite[Definition 2.32 (4)]{JTZ}), replacing “$\alpha$/$\beta$-equivalence” with “$\alpha$/$\beta$-handleslide”;
\item $F$ commutes along every stabilization slide (see \cite[Definition 7.7]{JTZ}).
\end{itemize}
\end{definition}

\begin{theorem}[Generalization of Theorem A.6 in \cite{JTZ}]\label{thm:mom}
Every loose Heegaard invariant $F': \cG'(\cS) \rightarrow \mathcal{C}$ extends to a weak Heegaard invariant $F: \cG(\cS) \rightarrow \mathcal{C}$. If $F'$ is tight, then $F'$ uniquely extends to a strong Heegaard invariant $F: \cG(\cS) \rightarrow \mathcal{C}$.
\end{theorem}

\begin{proof}
For the first part, we need to define \( F(\be) \) only for edges \( \be \) of \( \cG(\cS) \) corresponding to \( \alpha \)-equivalences or \( \beta \)-equivalences. Without loss of generality, suppose \( e \) is an \( \alpha \)-equivalence between sutured isotopy diagrams \( H = (\Sigma, \iso{\bal}, \iso{\bbe}) \) and \( H' = (\Sigma, \iso{\bal'}, \iso{\bbe}) \). Since \( \bal \sim \bal' \), both \( \iso{\bal} \) and \( \iso{\bal'} \) are cut-systems for the sutured compression body \( C(\bal) \). Therefore, by Proposition \ref{prop:y0}, since \( Y_2(C(\bal)) \) is connected, the cut-systems \( \iso{\bal} \) and \( \iso{\bal'} \) are related by a sequence of handleslides on $\Sigma$. Paired with \( \bbe \), this sequence of \( \alpha \)-handleslides gives rise to a sequence of handleslides \( h_1, \dots, h_k \) connecting \( H \) and \( H' \) in \( \cG'(\cS) \). We define \( F(\be) \) as the composite \( F'(h_k) \circ \dots \circ F'(h_1) \).

For the second part, we verify that if $F'$ satisfies the additional conditions, then a weak extension $F$ from the first part uniquely determines a strong Heegaard invariant. By Theorem \ref{thm:simpc}, the 2-complex $Y_2(C(\bal))$ is simply-connected. Since $F'$ commutes along every handleslide loop (i.e., along the boundary of $2$-cells of $Y_2(C(\bal))$, we have that the extension of $F'$ is independent of the choice of path $h_1,\dots, h_k$. Therefore, the extension is unique, and we denote it by \( F \). We now proceed to show that \( F \) is indeed strong.  

Since $F$ satisfies the functorality axiom over $\cG_{\alpha}(\cS)$ and $\cG_{\beta}(\cS)$ by construction, it remains to show that $F$ satisfies the commutativity axiom. As in the proof of \cite[Theorem A.6]{JTZ}, we can decompose each distinguished rectangle in $\cG(\cS)$ into distinguished rectangles in $\cG'(\cS)$. It suffices to consider distinguished rectangles of type (1), (2) and (3) (cf. Definition 2.29 in \cite{JTZ}). 

\begin{figure}[h]
{
   \fontsize{12pt}{11pt}\selectfont
   \def\svgwidth{3.6in}
   \begin{center}
   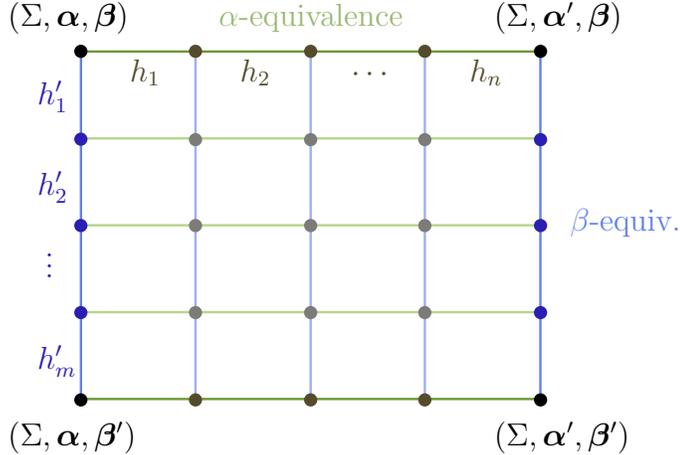
   \end{center}
   \caption{Decomposing a distinguished rectangle of type (1)}
   \label{grid}
}
\end{figure}

For type (1), each distinguished rectangle can be decompose into a grid of smaller rectangles with sides $\alpha$/$\beta$-handleslides as in Figure \ref{grid}. Similarly, for distinguished rectangles of type (3), we can decompose it into a row of smaller rectangles with sides consisting of handleslides and diffeomorphism of $\Sigma$ as in Figure \ref{row}. 

\begin{figure}[h]
{
   \fontsize{12pt}{11pt}\selectfont
   \def\svgwidth{3.2in}
   \begin{center}
   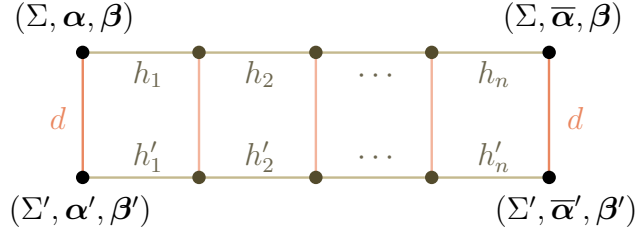
   \end{center}
   \caption{Decomposing a distinguished rectangle of type (3)}
   \label{row}
}
\end{figure}

We now decompose distinguished rectangles of type (2), whose vertical sides are stabilizations and whose horizontal sides either both $\alpha$-equivalences or both $\beta$-equivalences. Without loss of generality, we assume that the horizontal sides are $\alpha$-equivalences. Let $h_1,\dots,h_k$ be a handleslide resolution of the $\alpha$-equivalence on the destabilized side. Then, we stabilize the handleslide sequence above, and obtain a handleslide sequence $h_1',\dots,h_k'$ from $(\Sigma, \bal, \bbe) \# (T, \alpha_0, \beta_0)$ to $(\Sigma, \overline{\bal}, \bbe) \# (T, \alpha_0, \beta_0)$, where $T$ is a torus and $\alpha_0$ and $\beta_0$ intersects transversely at a single point. Since isotopies on $\Sigma$ may pass over the connected-sum disk of the stabilization, the sutured diagram $(\Sigma', \overline{\bal}', \bbe')$ in the lower-right corner may differ from $(\Sigma, \overline{\bal}, \bbe) \# (T, \alpha_0, \beta_0)$ by a sequence of handleslides over $\alpha_0$. To correct this, we attach a row of stabilization slides to the row of distinguished rectangles in $\cG'(\cS)$ of type (2), whose horizontal sides are $h_i$ and $h_i'$ (see Figure \ref{stab_slides}).

\begin{figure}[h]
{
   \fontsize{12pt}{11pt}\selectfont
   \def\svgwidth{4.4in}
   \begin{center}
   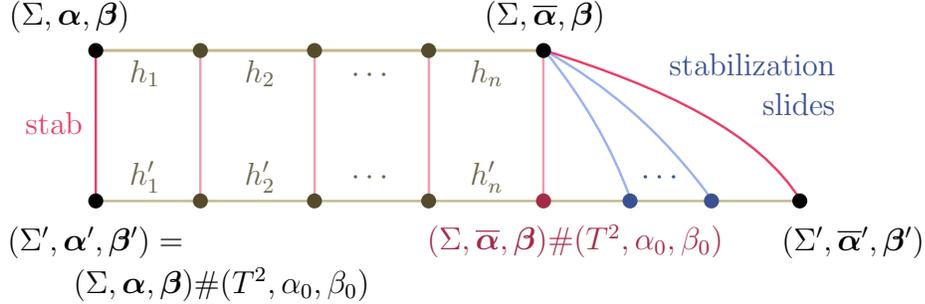
   \end{center}
   \caption{Decomposing a distinguished rectangle of type (2)}
   \label{stab_slides}
}
\end{figure}
\end{proof}

\begin{corollary} Let $\cS=\cS_{\mathrm{bal}}$ be the set of diffeomorphism types of balanced sutured manifolds, and let $\mathcal{C}$ be a category.  Then every loose Heegaard invariant $F': \cG'(\cS) \rightarrow \mathcal{C}$ extends to a weak Heegaard invariant $F: \cG(\cS) \rightarrow \mathcal{C}$. Moreover, if $F'$ is a tight Heegaard invariant, then $F'$ uniquely extends to a strong Heegaard invariant $F: \cG(\cS) \rightarrow \mathcal{C}$.
\end{corollary}

Each multi-based 3-manifold defines a balanced sutured manifold if we assign an oriented tangent 2-plane to each basepoint (cf. \cite[Definition 2.4]{JTZ}). Therefore, we can apply Theorem \ref{thm:mom} to the set of diffeomorphism types of balanced sutured manifolds arising from multi-based 3-manifolds, and obtain the following corollary for classical (i.e., non-sutured) Heegaard invariants. Here, we denote by \( [(M, \gamma)] \) the diffeomorphism type of a sutured manifold \( (M, \gamma) \).

\begin{corollary} Let $\cS=\cS_{\mathrm{man*s}}$ be the set of all \(\left[Y(\bz, \bV)\right]\), where $Y$ is a closed oriented $3$-manifold with basepoints $\bz$, and $\bV$ assigns an oriented tangent $2$-plane for each basepoint. Let $\mathcal{C}$ be a category. Then every loose Heegaard invariant $F': \cG'(\cS) \rightarrow \mathcal{C}$ extends to a weak Heegaard invariant $F: \cG(\cS) \rightarrow \mathcal{C}$.
If, furthermore, $F'$ is a tight Heegaard invariant, then $F'$ uniquely extends to a strong Heegaard invariant $F: \cG(\cS) \rightarrow \mathcal{C}$.
\end{corollary}
\begin{remark}
Let $\mathrm{\bf Man}_\mathrm{*s}$ be the category of multi-based $3$-manifolds, and let $\mathrm{\bf Man}_{\mathrm{*s},V}$ be the category of multi-based $3$-manifolds with a choice of oriented tangent $2$-plane at each basepoint.
Since each fiber of the forgetful functor $\mathrm{\bf Man}_{\mathrm{*s},V}\rightarrow\mathrm{\bf Man}_\mathrm{*s}$ is a {product of spheres}, which is simply-connected, there is no monodromy along any loop of oriented tangent $2$-planes (cf. \cite[Lemma 2.45]{JTZ} and the proof of \cite[Theorem 1.5]{JTZ}).
\end{remark}
A \emph{multi-based oriented link} \( \mathbb{L} = (L, \bz, \bw) \) inside a closed, oriented 3-manifold \( Y \) consists of an oriented \( \ell \)-component link \( L \subset Y \), together with two disjoint collections of basepoints \( \bz = (z_1, \dots, z_k) \) and \( \bw = (w_1, \dots, w_k) \) placed on \( L \). Each component of \( L \) is required to contain at least one \(\bz \)-basepoint and one \( \bw \)-basepoint, and the basepoints alternate between \( \bz \) and \( \bw \) as one traverses each component according to its orientation. In particular, we have \( k \ge \ell \).

Similar to the case of multi-based 3-manifolds, a multi-based oriented link defines a balanced sutured manifold by adding one suture circle for each basepoint to the torus boundary of the link complement (cf. \cite[Definition 2.5]{JTZ}). Thus, we can apply Theorem \ref{thm:mom} to the set of diffeomorphism types corresponding to multi-based links and obtain the following corollary.

\begin{corollary} Let $\cS=\cS_{\mathrm{link*s}}$ be the set of all \(\left[Y(\mathbb{L})\right]\), where $Y$ is a closed oriented $3$-manifold and \(\mathbb{L}=(L, \bz, \bw)\) is a multi-based oriented link in \( Y \). Let $\mathcal{C}$ be a category. Then every loose Heegaard invariant $F': \cG'(\cS) \rightarrow \mathcal{C}$ extends to a weak Heegaard invariant $F: \cG(\cS) \rightarrow \mathcal{C}$.
If, furthermore, $F'$ is a tight Heegaard invariant, then $F'$ uniquely extends to a strong Heegaard invariant $F: \cG(\cS) \rightarrow \mathcal{C}$.
\end{corollary}

\begin{remark}
We can obtain sutured diagrams from multi-based Heegaard diagrams by removing disks around each basepoint that are small enough to be disjoint from \( \bal \) and \( \bbe \) on the Heegaard surface \( \Sigma \).
\end{remark}

\bibliography{compress_body}
\bibliographystyle{custom}
\end{document}

%% file: sea_leaf.pdf_tex
\begingroup%
  \makeatletter%
  \providecommand\color[2][]{%
    \errmessage{(Inkscape) Color is used for the text in Inkscape, but the package 'color.sty' is not loaded}%
    \renewcommand\color[2][]{}%
  }%
  \providecommand\transparent[1]{%
    \errmessage{(Inkscape) Transparency is used (non-zero) for the text in Inkscape, but the package 'transparent.sty' is not loaded}%
    \renewcommand\transparent[1]{}%
  }%
  \providecommand\rotatebox[2]{#2}%
  \newcommand*\fsize{\dimexpr\f@size pt\relax}%
  \newcommand*\lineheight[1]{\fontsize{\fsize}{#1\fsize}\selectfont}%
  \ifx\svgwidth\undefined%
    \setlength{\unitlength}{438.11737781bp}%
    \ifx\svgscale\undefined%
      \relax%
    \else%
      \setlength{\unitlength}{\unitlength * \real{\svgscale}}%
    \fi%
  \else%
    \setlength{\unitlength}{\svgwidth}%
  \fi%
  \global\let\svgwidth\undefined%
  \global\let\svgscale\undefined%
  \makeatother%
  \begin{picture}(1,0.45964037)%
    \lineheight{1}%
    \setlength\tabcolsep{0pt}%
    \put(0,0){\includegraphics[width=\unitlength,page=1]{sea_leaf.pdf}}%
    \put(0.98185852,0.20434567){\color[rgb]{0,0,0}\makebox(0,0)[rt]{\lineheight{1.25}\smash{\begin{tabular}[t]{r}Minimal resolution\end{tabular}}}}%
    \put(0,0){\includegraphics[width=\unitlength,page=2]{sea_leaf.pdf}}%
    \put(0.98112766,0.13407252){\color[rgb]{0,0,0}\makebox(0,0)[rt]{\lineheight{1.25}\smash{\begin{tabular}[t]{r}Sum of $2$-cells in $Y_2(C(\bdel))$\end{tabular}}}}%
    \put(0,0){\includegraphics[width=\unitlength,page=3]{sea_leaf.pdf}}%
  \end{picture}%
\endgroup%

%% file: attaching_set.pdf_tex
\begingroup%
  \makeatletter%
  \providecommand\color[2][]{%
    \errmessage{(Inkscape) Color is used for the text in Inkscape, but the package 'color.sty' is not loaded}%
    \renewcommand\color[2][]{}%
  }%
  \providecommand\transparent[1]{%
    \errmessage{(Inkscape) Transparency is used (non-zero) for the text in Inkscape, but the package 'transparent.sty' is not loaded}%
    \renewcommand\transparent[1]{}%
  }%
  \providecommand\rotatebox[2]{#2}%
  \newcommand*\fsize{\dimexpr\f@size pt\relax}%
  \newcommand*\lineheight[1]{\fontsize{\fsize}{#1\fsize}\selectfont}%
  \ifx\svgwidth\undefined%
    \setlength{\unitlength}{497.09812915bp}%
    \ifx\svgscale\undefined%
      \relax%
    \else%
      \setlength{\unitlength}{\unitlength * \real{\svgscale}}%
    \fi%
  \else%
    \setlength{\unitlength}{\svgwidth}%
  \fi%
  \global\let\svgwidth\undefined%
  \global\let\svgscale\undefined%
  \makeatother%
  \begin{picture}(1,0.42948374)%
    \lineheight{1}%
    \setlength\tabcolsep{0pt}%
    \put(0,0){\includegraphics[width=\unitlength,page=1]{attaching_set.pdf}}%
    \put(0.62103025,0.03632706){\color[rgb]{0.0627451,0.60392157,0.6745098}\makebox(0,0)[lt]{\lineheight{1.25}\smash{\begin{tabular}[t]{l}$\bdel$\end{tabular}}}}%
    \put(-0.00103803,0.17552003){\color[rgb]{0,0,0}\makebox(0,0)[lt]{\lineheight{1.25}\smash{\begin{tabular}[t]{l}$\Sigma$\end{tabular}}}}%
    \put(0,0){\includegraphics[width=\unitlength,page=2]{attaching_set.pdf}}%
  \end{picture}%
\endgroup%

%% file: boundary.pdf_tex
\begingroup%
  \makeatletter%
  \providecommand\color[2][]{%
    \errmessage{(Inkscape) Color is used for the text in Inkscape, but the package 'color.sty' is not loaded}%
    \renewcommand\color[2][]{}%
  }%
  \providecommand\transparent[1]{%
    \errmessage{(Inkscape) Transparency is used (non-zero) for the text in Inkscape, but the package 'transparent.sty' is not loaded}%
    \renewcommand\transparent[1]{}%
  }%
  \providecommand\rotatebox[2]{#2}%
  \newcommand*\fsize{\dimexpr\f@size pt\relax}%
  \newcommand*\lineheight[1]{\fontsize{\fsize}{#1\fsize}\selectfont}%
  \ifx\svgwidth\undefined%
    \setlength{\unitlength}{462.06441564bp}%
    \ifx\svgscale\undefined%
      \relax%
    \else%
      \setlength{\unitlength}{\unitlength * \real{\svgscale}}%
    \fi%
  \else%
    \setlength{\unitlength}{\svgwidth}%
  \fi%
  \global\let\svgwidth\undefined%
  \global\let\svgscale\undefined%
  \makeatother%
  \begin{picture}(1,0.37719303)%
    \lineheight{1}%
    \setlength\tabcolsep{0pt}%
    \put(0,0){\includegraphics[width=\unitlength,page=1]{boundary.pdf}}%
  \end{picture}%
\endgroup%

%% file: self_sep_square.pdf_tex
\begingroup%
  \makeatletter%
  \providecommand\color[2][]{%
    \errmessage{(Inkscape) Color is used for the text in Inkscape, but the package 'color.sty' is not loaded}%
    \renewcommand\color[2][]{}%
  }%
  \providecommand\transparent[1]{%
    \errmessage{(Inkscape) Transparency is used (non-zero) for the text in Inkscape, but the package 'transparent.sty' is not loaded}%
    \renewcommand\transparent[1]{}%
  }%
  \providecommand\rotatebox[2]{#2}%
  \newcommand*\fsize{\dimexpr\f@size pt\relax}%
  \newcommand*\lineheight[1]{\fontsize{\fsize}{#1\fsize}\selectfont}%
  \ifx\svgwidth\undefined%
    \setlength{\unitlength}{563.86191541bp}%
    \ifx\svgscale\undefined%
      \relax%
    \else%
      \setlength{\unitlength}{\unitlength * \real{\svgscale}}%
    \fi%
  \else%
    \setlength{\unitlength}{\svgwidth}%
  \fi%
  \global\let\svgwidth\undefined%
  \global\let\svgscale\undefined%
  \makeatother%
  \begin{picture}(1,0.46103515)%
    \lineheight{1}%
    \setlength\tabcolsep{0pt}%
    \put(0,0){\includegraphics[width=\unitlength,page=1]{self_sep_square.pdf}}%
    \put(0.26075073,0.29108907){\color[rgb]{0.24313725,0.6745098,0.0627451}\makebox(0,0)[lt]{\lineheight{1.25}\smash{\begin{tabular}[t]{l}$\bal$\end{tabular}}}}%
    \put(0.82072828,0.29108907){\color[rgb]{0.24313725,0.6745098,0.0627451}\makebox(0,0)[lt]{\lineheight{1.25}\smash{\begin{tabular}[t]{l}$\bal$\end{tabular}}}}%
    \put(0.26075073,0.01974586){\color[rgb]{0.24313725,0.6745098,0.0627451}\makebox(0,0)[lt]{\lineheight{1.25}\smash{\begin{tabular}[t]{l}$\bal$\end{tabular}}}}%
    \put(0.82072828,0.01974586){\color[rgb]{0.24313725,0.6745098,0.0627451}\makebox(0,0)[lt]{\lineheight{1.25}\smash{\begin{tabular}[t]{l}$\bal$\end{tabular}}}}%
    \put(0.7525087,0.44417998){\color[rgb]{0.0627451,0.60392157,0.6745098}\makebox(0,0)[lt]{\lineheight{1.25}\smash{\begin{tabular}[t]{l}$\alpha_4$\end{tabular}}}}%
    \put(0.33299523,0.27262534){\color[rgb]{0.0627451,0.60392157,0.6745098}\makebox(0,0)[lt]{\lineheight{1.25}\smash{\begin{tabular}[t]{l}$\alpha_1$\end{tabular}}}}%
    \put(0.19459508,0.00612917){\color[rgb]{0.0627451,0.60392157,0.6745098}\makebox(0,0)[lt]{\lineheight{1.25}\smash{\begin{tabular}[t]{l}$\alpha_3$\end{tabular}}}}%
    \put(0.89503832,0.17608236){\color[rgb]{0.0627451,0.60392157,0.6745098}\makebox(0,0)[lt]{\lineheight{1.25}\smash{\begin{tabular}[t]{l}$\alpha_2$\end{tabular}}}}%
    \put(0.33299523,0.17608236){\color[rgb]{0.0627451,0.60392157,0.6745098}\makebox(0,0)[lt]{\lineheight{1.25}\smash{\begin{tabular}[t]{l}$\alpha_2$\end{tabular}}}}%
    \put(0.89503832,0.27262534){\color[rgb]{0.0627451,0.60392157,0.6745098}\makebox(0,0)[lt]{\lineheight{1.25}\smash{\begin{tabular}[t]{l}$\alpha_1$\end{tabular}}}}%
    \put(0.19459508,0.27123349){\color[rgb]{0.0627451,0.60392157,0.6745098}\makebox(0,0)[lt]{\lineheight{1.25}\smash{\begin{tabular}[t]{l}$\alpha_3$\end{tabular}}}}%
    \put(0.7525087,0.17608236){\color[rgb]{0.0627451,0.60392157,0.6745098}\makebox(0,0)[lt]{\lineheight{1.25}\smash{\begin{tabular}[t]{l}$\alpha_4$\end{tabular}}}}%
  \end{picture}%
\endgroup%

%% file: inner_d.pdf_tex
\begingroup%
  \makeatletter%
  \providecommand\color[2][]{%
    \errmessage{(Inkscape) Color is used for the text in Inkscape, but the package 'color.sty' is not loaded}%
    \renewcommand\color[2][]{}%
  }%
  \providecommand\transparent[1]{%
    \errmessage{(Inkscape) Transparency is used (non-zero) for the text in Inkscape, but the package 'transparent.sty' is not loaded}%
    \renewcommand\transparent[1]{}%
  }%
  \providecommand\rotatebox[2]{#2}%
  \newcommand*\fsize{\dimexpr\f@size pt\relax}%
  \newcommand*\lineheight[1]{\fontsize{\fsize}{#1\fsize}\selectfont}%
  \ifx\svgwidth\undefined%
    \setlength{\unitlength}{323.73263862bp}%
    \ifx\svgscale\undefined%
      \relax%
    \else%
      \setlength{\unitlength}{\unitlength * \real{\svgscale}}%
    \fi%
  \else%
    \setlength{\unitlength}{\svgwidth}%
  \fi%
  \global\let\svgwidth\undefined%
  \global\let\svgscale\undefined%
  \makeatother%
  \begin{picture}(1,0.63677124)%
    \lineheight{1}%
    \setlength\tabcolsep{0pt}%
    \put(0,0){\includegraphics[width=\unitlength,page=1]{inner_d.pdf}}%
    \put(0.24308191,0.58038765){\color[rgb]{0.17647059,0.25490196,0.74509804}\makebox(0,0)[lt]{\lineheight{1.25}\smash{\begin{tabular}[t]{l}$\alpha$\end{tabular}}}}%
    \put(0,0){\includegraphics[width=\unitlength,page=2]{inner_d.pdf}}%
    \put(0.42996459,0.27575831){\color[rgb]{0.25882353,0.46666667,0.16078431}\makebox(0,0)[lt]{\lineheight{1.25}\smash{\begin{tabular}[t]{l}$d$\end{tabular}}}}%
    \put(0.73318193,0.04817083){\color[rgb]{0.94509804,0.44313725,0.20784314}\makebox(0,0)[lt]{\lineheight{1.25}\smash{\begin{tabular}[t]{l}$\varepsilon$\end{tabular}}}}%
    \put(0.57633406,0.27618322){\color[rgb]{0.90980392,0.24705882,0.42352941}\makebox(0,0)[lt]{\lineheight{1.25}\smash{\begin{tabular}[t]{l}$D_{\alpha}^{\prime}$\end{tabular}}}}%
    \put(0.24351975,0.03427473){\color[rgb]{0,0,0}\makebox(0,0)[lt]{\lineheight{1.25}\smash{\begin{tabular}[t]{l}$\beta$\end{tabular}}}}%
    \put(0,0){\includegraphics[width=\unitlength,page=3]{inner_d.pdf}}%
  \end{picture}%
\endgroup%

%% file: watermelon.pdf_tex
\begingroup%
  \makeatletter%
  \providecommand\color[2][]{%
    \errmessage{(Inkscape) Color is used for the text in Inkscape, but the package 'color.sty' is not loaded}%
    \renewcommand\color[2][]{}%
  }%
  \providecommand\transparent[1]{%
    \errmessage{(Inkscape) Transparency is used (non-zero) for the text in Inkscape, but the package 'transparent.sty' is not loaded}%
    \renewcommand\transparent[1]{}%
  }%
  \providecommand\rotatebox[2]{#2}%
  \newcommand*\fsize{\dimexpr\f@size pt\relax}%
  \newcommand*\lineheight[1]{\fontsize{\fsize}{#1\fsize}\selectfont}%
  \ifx\svgwidth\undefined%
    \setlength{\unitlength}{367.80552661bp}%
    \ifx\svgscale\undefined%
      \relax%
    \else%
      \setlength{\unitlength}{\unitlength * \real{\svgscale}}%
    \fi%
  \else%
    \setlength{\unitlength}{\svgwidth}%
  \fi%
  \global\let\svgwidth\undefined%
  \global\let\svgscale\undefined%
  \makeatother%
  \begin{picture}(1,0.64267558)%
    \lineheight{1}%
    \setlength\tabcolsep{0pt}%
    \put(0,0){\includegraphics[width=\unitlength,page=1]{watermelon.pdf}}%
    \put(0.59562989,0.62175336){\color[rgb]{0,0,0}\makebox(0,0)[lt]{\lineheight{1.25}\smash{\begin{tabular}[t]{l}$\bu_1$\end{tabular}}}}%
    \put(0.59618772,0.00479253){\color[rgb]{0,0,0}\makebox(0,0)[lt]{\lineheight{1.25}\smash{\begin{tabular}[t]{l}$\bv_1$\end{tabular}}}}%
    \put(0.37655383,0.00479253){\color[rgb]{0,0,0}\makebox(0,0)[lt]{\lineheight{1.25}\smash{\begin{tabular}[t]{l}$\bu_2$\end{tabular}}}}%
    \put(0.37711142,0.62175336){\color[rgb]{0,0,0}\makebox(0,0)[lt]{\lineheight{1.25}\smash{\begin{tabular}[t]{l}$\bv_2$\end{tabular}}}}%
    \put(0.42309366,0.31819239){\color[rgb]{0,0,0}\makebox(0,0)[lt]{\lineheight{1.25}\smash{\begin{tabular}[t]{l}$\mathrm{Lemma}$\end{tabular}}}}%
    \put(0.74374546,0.51297872){\color[rgb]{0,0,0}\makebox(0,0)[lt]{\lineheight{1.25}\smash{\begin{tabular}[t]{l}$\iso{\alpha}\mathrm{-segment}$\end{tabular}}}}%
    \put(-0.00205774,0.09023084){\color[rgb]{0,0,0}\makebox(0,0)[lt]{\lineheight{1.25}\smash{\begin{tabular}[t]{l}$\iso{\beta}\mathrm{-segment}$\end{tabular}}}}%
    \put(0.23067335,0.31819239){\color[rgb]{0,0,0}\makebox(0,0)[lt]{\lineheight{1.25}\smash{\begin{tabular}[t]{l}$\mathrm{Lemma}$\end{tabular}}}}%
    \put(0.61505211,0.31819239){\color[rgb]{0,0,0}\makebox(0,0)[lt]{\lineheight{1.25}\smash{\begin{tabular}[t]{l}$\mathrm{Lemma}$\end{tabular}}}}%
    \put(0.46387613,0.25319905){\color[rgb]{0,0,0}\makebox(0,0)[lt]{\lineheight{1.25}\smash{\begin{tabular}[t]{l}$\mathrm{\ref{lem11}}$\end{tabular}}}}%
    \put(0.28796919,0.25340479){\color[rgb]{0,0,0}\makebox(0,0)[lt]{\lineheight{1.25}\smash{\begin{tabular}[t]{l}$\mathrm{\ref{lem:seg}}$\end{tabular}}}}%
    \put(0.66053244,0.25340478){\color[rgb]{0,0,0}\makebox(0,0)[lt]{\lineheight{1.25}\smash{\begin{tabular}[t]{l}$\mathrm{\ref{lem:seg}}$\end{tabular}}}}%
  \end{picture}%
\endgroup%

%% file: slide_tri.pdf_tex
\begingroup%
  \makeatletter%
  \providecommand\color[2][]{%
    \errmessage{(Inkscape) Color is used for the text in Inkscape, but the package 'color.sty' is not loaded}%
    \renewcommand\color[2][]{}%
  }%
  \providecommand\transparent[1]{%
    \errmessage{(Inkscape) Transparency is used (non-zero) for the text in Inkscape, but the package 'transparent.sty' is not loaded}%
    \renewcommand\transparent[1]{}%
  }%
  \providecommand\rotatebox[2]{#2}%
  \newcommand*\fsize{\dimexpr\f@size pt\relax}%
  \newcommand*\lineheight[1]{\fontsize{\fsize}{#1\fsize}\selectfont}%
  \ifx\svgwidth\undefined%
    \setlength{\unitlength}{330.44182485bp}%
    \ifx\svgscale\undefined%
      \relax%
    \else%
      \setlength{\unitlength}{\unitlength * \real{\svgscale}}%
    \fi%
  \else%
    \setlength{\unitlength}{\svgwidth}%
  \fi%
  \global\let\svgwidth\undefined%
  \global\let\svgscale\undefined%
  \makeatother%
  \begin{picture}(1,0.38143464)%
    \lineheight{1}%
    \setlength\tabcolsep{0pt}%
    \put(0,0){\includegraphics[width=\unitlength,page=1]{slide_tri.pdf}}%
    \put(0.13821135,0.20363613){\color[rgb]{1,0.0745098,0.49803922}\makebox(0,0)[lt]{\lineheight{1.25}\smash{\begin{tabular}[t]{l}$\alpha_2$\end{tabular}}}}%
    \put(0.06777258,0.10524104){\color[rgb]{1,0.0745098,0.49803922}\makebox(0,0)[lt]{\lineheight{1.25}\smash{\begin{tabular}[t]{l}$\alpha_1$\end{tabular}}}}%
    \put(0,0){\includegraphics[width=\unitlength,page=2]{slide_tri.pdf}}%
    \put(0.18470795,0.1052411){\color[rgb]{1,0.0745098,0.49803922}\makebox(0,0)[lt]{\lineheight{1.25}\smash{\begin{tabular}[t]{l}$\alpha_3$\end{tabular}}}}%
  \end{picture}%
\endgroup%

%% file: slide_square.pdf_tex
\begingroup%
  \makeatletter%
  \providecommand\color[2][]{%
    \errmessage{(Inkscape) Color is used for the text in Inkscape, but the package 'color.sty' is not loaded}%
    \renewcommand\color[2][]{}%
  }%
  \providecommand\transparent[1]{%
    \errmessage{(Inkscape) Transparency is used (non-zero) for the text in Inkscape, but the package 'transparent.sty' is not loaded}%
    \renewcommand\transparent[1]{}%
  }%
  \providecommand\rotatebox[2]{#2}%
  \newcommand*\fsize{\dimexpr\f@size pt\relax}%
  \newcommand*\lineheight[1]{\fontsize{\fsize}{#1\fsize}\selectfont}%
  \ifx\svgwidth\undefined%
    \setlength{\unitlength}{443.22247363bp}%
    \ifx\svgscale\undefined%
      \relax%
    \else%
      \setlength{\unitlength}{\unitlength * \real{\svgscale}}%
    \fi%
  \else%
    \setlength{\unitlength}{\svgwidth}%
  \fi%
  \global\let\svgwidth\undefined%
  \global\let\svgscale\undefined%
  \makeatother%
  \begin{picture}(1,1.34857559)%
    \lineheight{1}%
    \setlength\tabcolsep{0pt}%
    \put(0.68535115,1.20196592){\makebox(0,0)[lt]{\lineheight{1.25}\smash{\begin{tabular}[t]{l}Type-I\end{tabular}}}}%
    \put(0,0){\includegraphics[width=\unitlength,page=1]{slide_square.pdf}}%
    \put(0.68535115,0.8583602){\makebox(0,0)[lt]{\lineheight{1.25}\smash{\begin{tabular}[t]{l}Type-II\end{tabular}}}}%
    \put(0,0){\includegraphics[width=\unitlength,page=2]{slide_square.pdf}}%
    \put(0.67858263,0.51304112){\makebox(0,0)[lt]{\lineheight{1.25}\smash{\begin{tabular}[t]{l}Type-III\end{tabular}}}}%
    \put(0.67519838,0.15070052){\makebox(0,0)[lt]{\lineheight{1.25}\smash{\begin{tabular}[t]{l}Type-IV\end{tabular}}}}%
    \put(0,0){\includegraphics[width=\unitlength,page=3]{slide_square.pdf}}%
    \put(0.18341539,1.10311998){\color[rgb]{0,0,0}\makebox(0,0)[lt]{\lineheight{1.25}\smash{\begin{tabular}[t]{l}$(2)$ \end{tabular}}}}%
    \put(0.18341539,0.74820533){\color[rgb]{0,0,0}\makebox(0,0)[lt]{\lineheight{1.25}\smash{\begin{tabular}[t]{l}$(4)$ \end{tabular}}}}%
    \put(0.18341539,0.37954951){\color[rgb]{0,0,0}\makebox(0,0)[lt]{\lineheight{1.25}\smash{\begin{tabular}[t]{l}$(3)$ \end{tabular}}}}%
    \put(0.18341539,0.02350325){\color[rgb]{0,0,0}\makebox(0,0)[lt]{\lineheight{1.25}\smash{\begin{tabular}[t]{l}$(5)$ \end{tabular}}}}%
  \end{picture}%
\endgroup%

%% file: pentagon.pdf_tex
\begingroup%
  \makeatletter%
  \providecommand\color[2][]{%
    \errmessage{(Inkscape) Color is used for the text in Inkscape, but the package 'color.sty' is not loaded}%
    \renewcommand\color[2][]{}%
  }%
  \providecommand\transparent[1]{%
    \errmessage{(Inkscape) Transparency is used (non-zero) for the text in Inkscape, but the package 'transparent.sty' is not loaded}%
    \renewcommand\transparent[1]{}%
  }%
  \providecommand\rotatebox[2]{#2}%
  \newcommand*\fsize{\dimexpr\f@size pt\relax}%
  \newcommand*\lineheight[1]{\fontsize{\fsize}{#1\fsize}\selectfont}%
  \ifx\svgwidth\undefined%
    \setlength{\unitlength}{288.40155894bp}%
    \ifx\svgscale\undefined%
      \relax%
    \else%
      \setlength{\unitlength}{\unitlength * \real{\svgscale}}%
    \fi%
  \else%
    \setlength{\unitlength}{\svgwidth}%
  \fi%
  \global\let\svgwidth\undefined%
  \global\let\svgscale\undefined%
  \makeatother%
  \begin{picture}(1,0.92052931)%
    \lineheight{1}%
    \setlength\tabcolsep{0pt}%
    \put(0.39043504,0.43492555){\makebox(0,0)[lt]{\lineheight{1.25}\smash{\begin{tabular}[t]{l}pentagon\end{tabular}}}}%
    \put(0,0){\includegraphics[width=\unitlength,page=1]{pentagon.pdf}}%
  \end{picture}%
\endgroup%

%% file: regular.pdf_tex
\begingroup%
  \makeatletter%
  \providecommand\color[2][]{%
    \errmessage{(Inkscape) Color is used for the text in Inkscape, but the package 'color.sty' is not loaded}%
    \renewcommand\color[2][]{}%
  }%
  \providecommand\transparent[1]{%
    \errmessage{(Inkscape) Transparency is used (non-zero) for the text in Inkscape, but the package 'transparent.sty' is not loaded}%
    \renewcommand\transparent[1]{}%
  }%
  \providecommand\rotatebox[2]{#2}%
  \newcommand*\fsize{\dimexpr\f@size pt\relax}%
  \newcommand*\lineheight[1]{\fontsize{\fsize}{#1\fsize}\selectfont}%
  \ifx\svgwidth\undefined%
    \setlength{\unitlength}{807.55108691bp}%
    \ifx\svgscale\undefined%
      \relax%
    \else%
      \setlength{\unitlength}{\unitlength * \real{\svgscale}}%
    \fi%
  \else%
    \setlength{\unitlength}{\svgwidth}%
  \fi%
  \global\let\svgwidth\undefined%
  \global\let\svgscale\undefined%
  \makeatother%
  \begin{picture}(1,0.3574399)%
    \lineheight{1}%
    \setlength\tabcolsep{0pt}%
    \put(0,0){\includegraphics[width=\unitlength,page=1]{regular.pdf}}%
    \put(0.26196166,0.23371606){\color[rgb]{0.24313725,0.24313725,0.24313725}\makebox(0,0)[lt]{\lineheight{1.25}\smash{\begin{tabular}[t]{l}$\gamma_1$\end{tabular}}}}%
    \put(0.38910982,0.23371606){\color[rgb]{0.24313725,0.24313725,0.24313725}\makebox(0,0)[lt]{\lineheight{1.25}\smash{\begin{tabular}[t]{l}$\gamma_2$\end{tabular}}}}%
    \put(0.14786358,0.19623581){\color[rgb]{0.03921569,0.01176471,0.01568627}\makebox(0,0)[lt]{\lineheight{1.25}\smash{\begin{tabular}[t]{l}$\alpha_1$\end{tabular}}}}%
    \put(0.49545729,0.19871629){\color[rgb]{0.03921569,0.01176471,0.01568627}\makebox(0,0)[lt]{\lineheight{1.25}\smash{\begin{tabular}[t]{l}$\alpha_2$\end{tabular}}}}%
    \put(0.25866272,0.08870043){\color[rgb]{0.56862745,0.45098039,0.02745098}\makebox(0,0)[lt]{\lineheight{1.25}\smash{\begin{tabular}[t]{l}$P_1$\end{tabular}}}}%
    \put(0.32235846,0.1966281){\color[rgb]{0.44313725,0.41176471,0.01960784}\makebox(0,0)[lt]{\lineheight{1.25}\smash{\begin{tabular}[t]{l}$P_2$\end{tabular}}}}%
    \put(0.32727533,0.32188711){\color[rgb]{0,0,0}\makebox(0,0)[lt]{\lineheight{1.25}\smash{\begin{tabular}[t]{l}$\beta_0'$\end{tabular}}}}%
    \put(0.19706195,0.01601049){\color[rgb]{0,0,0}\makebox(0,0)[lt]{\lineheight{1.25}\smash{\begin{tabular}[t]{l}$\beta_0''$\end{tabular}}}}%
    \put(0.44526518,0.01808325){\color[rgb]{0,0,0}\makebox(0,0)[lt]{\lineheight{1.25}\smash{\begin{tabular}[t]{l}$\beta_0'''$\end{tabular}}}}%
    \put(0.38623893,0.08870043){\color[rgb]{1,0.0745098,0.49803922}\makebox(0,0)[lt]{\lineheight{1.25}\smash{\begin{tabular}[t]{l}$P_3$\end{tabular}}}}%
    \put(0.59627102,0.178938){\color[rgb]{0,0,0}\makebox(0,0)[lt]{\lineheight{1.25}\smash{\begin{tabular}[t]{l}$\mathcal{P}=\{P_1,P_2,P_3\}$\end{tabular}}}}%
    \put(0.59627102,0.13654675){\color[rgb]{0,0,0}\makebox(0,0)[lt]{\lineheight{1.25}\smash{\begin{tabular}[t]{l}$\partial\mathcal{P}=\{\gamma_1,\gamma_2\}$\end{tabular}}}}%
    \put(0.59627102,0.09219405){\color[rgb]{0,0,0}\makebox(0,0)[lt]{\lineheight{1.25}\smash{\begin{tabular}[t]{l}$\Gamma(\mathcal{P})=(\gamma_1,\gamma_2)$\end{tabular}}}}%
    \put(0,0){\includegraphics[width=\unitlength,page=2]{regular.pdf}}%
  \end{picture}%
\endgroup%

%% file: bridge.pdf_tex
\begingroup%
  \makeatletter%
  \providecommand\color[2][]{%
    \errmessage{(Inkscape) Color is used for the text in Inkscape, but the package 'color.sty' is not loaded}%
    \renewcommand\color[2][]{}%
  }%
  \providecommand\transparent[1]{%
    \errmessage{(Inkscape) Transparency is used (non-zero) for the text in Inkscape, but the package 'transparent.sty' is not loaded}%
    \renewcommand\transparent[1]{}%
  }%
  \providecommand\rotatebox[2]{#2}%
  \newcommand*\fsize{\dimexpr\f@size pt\relax}%
  \newcommand*\lineheight[1]{\fontsize{\fsize}{#1\fsize}\selectfont}%
  \ifx\svgwidth\undefined%
    \setlength{\unitlength}{290.48643097bp}%
    \ifx\svgscale\undefined%
      \relax%
    \else%
      \setlength{\unitlength}{\unitlength * \real{\svgscale}}%
    \fi%
  \else%
    \setlength{\unitlength}{\svgwidth}%
  \fi%
  \global\let\svgwidth\undefined%
  \global\let\svgscale\undefined%
  \makeatother%
  \begin{picture}(1,0.80907685)%
    \lineheight{1}%
    \setlength\tabcolsep{0pt}%
    \put(0.49428221,0.45792421){\color[rgb]{0.48235294,0.24313725,0.69803922}\makebox(0,0)[lt]{\lineheight{1.25}\smash{\begin{tabular}[t]{l}$\gamma_1$\end{tabular}}}}%
    \put(0.73452837,0.20038145){\color[rgb]{1,0.34509804,0.09411765}\makebox(0,0)[lt]{\lineheight{1.25}\smash{\begin{tabular}[t]{l}$\gamma_3$\end{tabular}}}}%
    \put(0,0){\includegraphics[width=\unitlength,page=1]{bridge.pdf}}%
    \put(0.10595208,0.32378809){\color[rgb]{0.29019608,0.57254902,0.18039216}\makebox(0,0)[lt]{\lineheight{1.25}\smash{\begin{tabular}[t]{l}$\gamma_2$\end{tabular}}}}%
    \put(0,0){\includegraphics[width=\unitlength,page=2]{bridge.pdf}}%
    \put(0.21664749,0.57125721){\color[rgb]{0,0,0}\makebox(0,0)[lt]{\lineheight{1.25}\smash{\begin{tabular}[t]{l}$\alpha_1$\end{tabular}}}}%
    \put(0,0){\includegraphics[width=\unitlength,page=3]{bridge.pdf}}%
    \put(0.49965182,0.1505052){\color[rgb]{0,0,0}\makebox(0,0)[lt]{\lineheight{1.25}\smash{\begin{tabular}[t]{l}$\rho_2$\end{tabular}}}}%
    \put(0.73545795,0.57007741){\color[rgb]{0,0,0}\makebox(0,0)[lt]{\lineheight{1.25}\smash{\begin{tabular}[t]{l}$\rho_1$\end{tabular}}}}%
    \put(0.29031406,0.26686563){\color[rgb]{0,0,0}\makebox(0,0)[lt]{\lineheight{1.25}\smash{\begin{tabular}[t]{l}$c_2$\end{tabular}}}}%
    \put(0.46239816,0.66176178){\color[rgb]{0,0,0}\makebox(0,0)[lt]{\lineheight{1.25}\smash{\begin{tabular}[t]{l}$c_1$\end{tabular}}}}%
    \put(0,0){\includegraphics[width=\unitlength,page=4]{bridge.pdf}}%
  \end{picture}%
\endgroup%

%% file: rho.pdf_tex
\begingroup%
  \makeatletter%
  \providecommand\color[2][]{%
    \errmessage{(Inkscape) Color is used for the text in Inkscape, but the package 'color.sty' is not loaded}%
    \renewcommand\color[2][]{}%
  }%
  \providecommand\transparent[1]{%
    \errmessage{(Inkscape) Transparency is used (non-zero) for the text in Inkscape, but the package 'transparent.sty' is not loaded}%
    \renewcommand\transparent[1]{}%
  }%
  \providecommand\rotatebox[2]{#2}%
  \newcommand*\fsize{\dimexpr\f@size pt\relax}%
  \newcommand*\lineheight[1]{\fontsize{\fsize}{#1\fsize}\selectfont}%
  \ifx\svgwidth\undefined%
    \setlength{\unitlength}{228.28396054bp}%
    \ifx\svgscale\undefined%
      \relax%
    \else%
      \setlength{\unitlength}{\unitlength * \real{\svgscale}}%
    \fi%
  \else%
    \setlength{\unitlength}{\svgwidth}%
  \fi%
  \global\let\svgwidth\undefined%
  \global\let\svgscale\undefined%
  \makeatother%
  \begin{picture}(1,0.70609017)%
    \lineheight{1}%
    \setlength\tabcolsep{0pt}%
    \put(0,0){\includegraphics[width=\unitlength,page=1]{rho.pdf}}%
    \put(0.47255199,0.56985256){\color[rgb]{0,0,0}\makebox(0,0)[lt]{\lineheight{1.25}\smash{\begin{tabular}[t]{l}$c'$\end{tabular}}}}%
    \put(0,0){\includegraphics[width=\unitlength,page=2]{rho.pdf}}%
    \put(0.0071899,0.22326104){\color[rgb]{0.47843137,0.10196078,0.37647059}\makebox(0,0)[lt]{\lineheight{1.25}\smash{\begin{tabular}[t]{l}$c_1$\end{tabular}}}}%
    \put(0.27542907,0.21669027){\color[rgb]{0.76862745,0.42745098,0.56862745}\makebox(0,0)[lt]{\lineheight{1.25}\smash{\begin{tabular}[t]{l}$c_2$\end{tabular}}}}%
    \put(0.61144526,0.38613365){\color[rgb]{0.24705882,0.4627451,0.16078431}\makebox(0,0)[lt]{\lineheight{1.25}\smash{\begin{tabular}[t]{l}$T'$\end{tabular}}}}%
    \put(0,0){\includegraphics[width=\unitlength,page=3]{rho.pdf}}%
    \put(0.15874175,0.06776468){\color[rgb]{0,0,0}\makebox(0,0)[lt]{\lineheight{1.25}\smash{\begin{tabular}[t]{l}$\rho$\end{tabular}}}}%
    \put(0.79517018,0.51582411){\color[rgb]{0,0,0}\makebox(0,0)[lt]{\lineheight{1.25}\smash{\begin{tabular}[t]{l}$\rho'$\end{tabular}}}}%
    \put(0.14432374,0.51636841){\color[rgb]{0,0,0}\makebox(0,0)[lt]{\lineheight{1.25}\smash{\begin{tabular}[t]{l}$\alpha_1$\end{tabular}}}}%
  \end{picture}%
\endgroup%

%% file: bigon.pdf_tex
\begingroup%
  \makeatletter%
  \providecommand\color[2][]{%
    \errmessage{(Inkscape) Color is used for the text in Inkscape, but the package 'color.sty' is not loaded}%
    \renewcommand\color[2][]{}%
  }%
  \providecommand\transparent[1]{%
    \errmessage{(Inkscape) Transparency is used (non-zero) for the text in Inkscape, but the package 'transparent.sty' is not loaded}%
    \renewcommand\transparent[1]{}%
  }%
  \providecommand\rotatebox[2]{#2}%
  \newcommand*\fsize{\dimexpr\f@size pt\relax}%
  \newcommand*\lineheight[1]{\fontsize{\fsize}{#1\fsize}\selectfont}%
  \ifx\svgwidth\undefined%
    \setlength{\unitlength}{346.89999318bp}%
    \ifx\svgscale\undefined%
      \relax%
    \else%
      \setlength{\unitlength}{\unitlength * \real{\svgscale}}%
    \fi%
  \else%
    \setlength{\unitlength}{\svgwidth}%
  \fi%
  \global\let\svgwidth\undefined%
  \global\let\svgscale\undefined%
  \makeatother%
  \begin{picture}(1,0.5347032)%
    \lineheight{1}%
    \setlength\tabcolsep{0pt}%
    \put(0,0){\includegraphics[width=\unitlength,page=1]{bigon.pdf}}%
    \put(0.26953382,0.17275553){\color[rgb]{0.23921569,0.23921569,0.23921569}\makebox(0,0)[lt]{\lineheight{1.25}\smash{\begin{tabular}[t]{l}$\lambda'$\end{tabular}}}}%
    \put(0.64365071,0.02541209){\color[rgb]{0.32941176,0.49803922,0.91372549}\makebox(0,0)[lt]{\lineheight{1.25}\smash{\begin{tabular}[t]{l}$\lambda$\end{tabular}}}}%
    \put(0,0){\includegraphics[width=\unitlength,page=2]{bigon.pdf}}%
  \end{picture}%
\endgroup%

%% file: gons.pdf_tex
\begingroup%
  \makeatletter%
  \providecommand\color[2][]{%
    \errmessage{(Inkscape) Color is used for the text in Inkscape, but the package 'color.sty' is not loaded}%
    \renewcommand\color[2][]{}%
  }%
  \providecommand\transparent[1]{%
    \errmessage{(Inkscape) Transparency is used (non-zero) for the text in Inkscape, but the package 'transparent.sty' is not loaded}%
    \renewcommand\transparent[1]{}%
  }%
  \providecommand\rotatebox[2]{#2}%
  \newcommand*\fsize{\dimexpr\f@size pt\relax}%
  \newcommand*\lineheight[1]{\fontsize{\fsize}{#1\fsize}\selectfont}%
  \ifx\svgwidth\undefined%
    \setlength{\unitlength}{564.57044707bp}%
    \ifx\svgscale\undefined%
      \relax%
    \else%
      \setlength{\unitlength}{\unitlength * \real{\svgscale}}%
    \fi%
  \else%
    \setlength{\unitlength}{\svgwidth}%
  \fi%
  \global\let\svgwidth\undefined%
  \global\let\svgscale\undefined%
  \makeatother%
  \begin{picture}(1,0.42602037)%
    \lineheight{1}%
    \setlength\tabcolsep{0pt}%
    \put(0,0){\includegraphics[width=\unitlength,page=1]{gons.pdf}}%
    \put(0.29163551,0.21513756){\color[rgb]{0.23921569,0.23921569,0.23921569}\makebox(0,0)[lt]{\lineheight{1.25}\smash{\begin{tabular}[t]{l}$\lambda'$\end{tabular}}}}%
    \put(0.2050426,0.12967986){\color[rgb]{0.32941176,0.49803922,0.91372549}\makebox(0,0)[lt]{\lineheight{1.25}\smash{\begin{tabular}[t]{l}$\lambda_i$\end{tabular}}}}%
    \put(0,0){\includegraphics[width=\unitlength,page=2]{gons.pdf}}%
    \put(0.16792712,0.2914985){\color[rgb]{0.58039216,0.55294118,0.47058824}\makebox(0,0)[lt]{\lineheight{1.25}\smash{\begin{tabular}[t]{l}$\lambda_{i+1}$\end{tabular}}}}%
    \put(0,0){\includegraphics[width=\unitlength,page=3]{gons.pdf}}%
    \put(0.73668272,0.25280472){\color[rgb]{0.23921569,0.23921569,0.23921569}\makebox(0,0)[lt]{\lineheight{1.25}\smash{\begin{tabular}[t]{l}$\lambda'$\end{tabular}}}}%
    \put(0.66807483,0.14024559){\color[rgb]{0.32941176,0.49803922,0.91372549}\makebox(0,0)[lt]{\lineheight{1.25}\smash{\begin{tabular}[t]{l}$\lambda_i$\end{tabular}}}}%
    \put(0,0){\includegraphics[width=\unitlength,page=4]{gons.pdf}}%
    \put(0.6372265,0.32350299){\color[rgb]{0.58039216,0.55294118,0.47058824}\makebox(0,0)[lt]{\lineheight{1.25}\smash{\begin{tabular}[t]{l}$\lambda_{i+1}$\end{tabular}}}}%
    \put(0.55766273,0.20354659){\color[rgb]{0,0,0}\makebox(0,0)[lt]{\lineheight{1.25}\smash{\begin{tabular}[t]{l}$\partial S$\end{tabular}}}}%
    \put(0.78838203,0.10635849){\color[rgb]{0.74509804,0.57647059,0.60392157}\makebox(0,0)[lt]{\lineheight{1.25}\smash{\begin{tabular}[t]{l}distinguished disk\end{tabular}}}}%
    \put(0,0){\includegraphics[width=\unitlength,page=5]{gons.pdf}}%
    \put(0.22569642,0.03298713){\color[rgb]{0,0,0}\makebox(0,0)[lt]{\lineheight{1.25}\smash{\begin{tabular}[t]{l}bigons\end{tabular}}}}%
    \put(0.68690562,0.03298713){\color[rgb]{0,0,0}\makebox(0,0)[lt]{\lineheight{1.25}\smash{\begin{tabular}[t]{l}half bigons\end{tabular}}}}%
    \put(0,0){\includegraphics[width=\unitlength,page=6]{gons.pdf}}%
  \end{picture}%
\endgroup%

%% file: gloves.pdf_tex
\begingroup%
  \makeatletter%
  \providecommand\color[2][]{%
    \errmessage{(Inkscape) Color is used for the text in Inkscape, but the package 'color.sty' is not loaded}%
    \renewcommand\color[2][]{}%
  }%
  \providecommand\transparent[1]{%
    \errmessage{(Inkscape) Transparency is used (non-zero) for the text in Inkscape, but the package 'transparent.sty' is not loaded}%
    \renewcommand\transparent[1]{}%
  }%
  \providecommand\rotatebox[2]{#2}%
  \newcommand*\fsize{\dimexpr\f@size pt\relax}%
  \newcommand*\lineheight[1]{\fontsize{\fsize}{#1\fsize}\selectfont}%
  \ifx\svgwidth\undefined%
    \setlength{\unitlength}{557.91896069bp}%
    \ifx\svgscale\undefined%
      \relax%
    \else%
      \setlength{\unitlength}{\unitlength * \real{\svgscale}}%
    \fi%
  \else%
    \setlength{\unitlength}{\svgwidth}%
  \fi%
  \global\let\svgwidth\undefined%
  \global\let\svgscale\undefined%
  \makeatother%
  \begin{picture}(1,0.46909377)%
    \lineheight{1}%
    \setlength\tabcolsep{0pt}%
    \put(0,0){\includegraphics[width=\unitlength,page=1]{gloves.pdf}}%
    \put(0.22238453,0.03907312){\color[rgb]{0,0,0}\makebox(0,0)[t]{\lineheight{1.25}\smash{\begin{tabular}[t]{c}$k$ other boundary \\components\end{tabular}}}}%
    \put(0.60189173,0.00771871){\color[rgb]{0,0,0}\makebox(0,0)[lt]{\lineheight{1.25}\smash{\begin{tabular}[t]{l}$(k-1)$ components\end{tabular}}}}%
    \put(0,0){\includegraphics[width=\unitlength,page=2]{gloves.pdf}}%
    \put(0.73346497,0.22047868){\color[rgb]{0.50196078,0.52156863,0.2}\makebox(0,0)[lt]{\lineheight{1.25}\smash{\begin{tabular}[t]{l}$F'$\end{tabular}}}}%
    \put(0.56687623,0.39748983){\color[rgb]{0,0,0}\makebox(0,0)[lt]{\lineheight{1.25}\smash{\begin{tabular}[t]{l}$\alpha_0$\end{tabular}}}}%
    \put(0.72998366,0.43442084){\color[rgb]{0,0,0}\makebox(0,0)[lt]{\lineheight{1.25}\smash{\begin{tabular}[t]{l}$\alpha_1$\end{tabular}}}}%
    \put(0.87242609,0.35202963){\color[rgb]{0,0,0}\makebox(0,0)[lt]{\lineheight{1.25}\smash{\begin{tabular}[t]{l}$\alpha_2$\end{tabular}}}}%
    \put(0,0){\includegraphics[width=\unitlength,page=3]{gloves.pdf}}%
    \put(0.65913833,0.34222065){\color[rgb]{0.64313725,0,0.78039216}\makebox(0,0)[lt]{\lineheight{1.25}\smash{\begin{tabular}[t]{l}$\alpha_0'$\end{tabular}}}}%
    \put(0,0){\includegraphics[width=\unitlength,page=4]{gloves.pdf}}%
    \put(0.2236881,0.23013734){\color[rgb]{0,0,0}\makebox(0,0)[lt]{\lineheight{1.25}\smash{\begin{tabular}[t]{l}$F$\end{tabular}}}}%
    \put(0.08607393,0.39748983){\color[rgb]{0,0,0}\makebox(0,0)[lt]{\lineheight{1.25}\smash{\begin{tabular}[t]{l}$\alpha_0$\end{tabular}}}}%
    \put(0.2491813,0.43442084){\color[rgb]{0,0,0}\makebox(0,0)[lt]{\lineheight{1.25}\smash{\begin{tabular}[t]{l}$\alpha_1$\end{tabular}}}}%
    \put(0.39162368,0.35202963){\color[rgb]{0,0,0}\makebox(0,0)[lt]{\lineheight{1.25}\smash{\begin{tabular}[t]{l}$\alpha_2$\end{tabular}}}}%
    \put(0,0){\includegraphics[width=\unitlength,page=5]{gloves.pdf}}%
  \end{picture}%
\endgroup%

%% file: decomp.pdf_tex
\begingroup%
  \makeatletter%
  \providecommand\color[2][]{%
    \errmessage{(Inkscape) Color is used for the text in Inkscape, but the package 'color.sty' is not loaded}%
    \renewcommand\color[2][]{}%
  }%
  \providecommand\transparent[1]{%
    \errmessage{(Inkscape) Transparency is used (non-zero) for the text in Inkscape, but the package 'transparent.sty' is not loaded}%
    \renewcommand\transparent[1]{}%
  }%
  \providecommand\rotatebox[2]{#2}%
  \newcommand*\fsize{\dimexpr\f@size pt\relax}%
  \newcommand*\lineheight[1]{\fontsize{\fsize}{#1\fsize}\selectfont}%
  \ifx\svgwidth\undefined%
    \setlength{\unitlength}{628.5519241bp}%
    \ifx\svgscale\undefined%
      \relax%
    \else%
      \setlength{\unitlength}{\unitlength * \real{\svgscale}}%
    \fi%
  \else%
    \setlength{\unitlength}{\svgwidth}%
  \fi%
  \global\let\svgwidth\undefined%
  \global\let\svgscale\undefined%
  \makeatother%
  \begin{picture}(1,0.73824443)%
    \lineheight{1}%
    \setlength\tabcolsep{0pt}%
    \put(0,0){\includegraphics[width=\unitlength,page=1]{decomp.pdf}}%
    \put(0.50136428,0.54870128){\color[rgb]{0.50196078,0.52156863,0.2}\makebox(0,0)[lt]{\lineheight{1.25}\smash{\begin{tabular}[t]{l}$F'$\end{tabular}}}}%
    \put(0.49204708,0.71033482){\color[rgb]{0.09411765,0.69019608,0.56862745}\makebox(0,0)[lt]{\lineheight{1.25}\smash{\begin{tabular}[t]{l}$\alpha_1$\end{tabular}}}}%
    \put(0.61113608,0.64611451){\color[rgb]{0.09411765,0.69019608,0.56862745}\makebox(0,0)[lt]{\lineheight{1.25}\smash{\begin{tabular}[t]{l}$\alpha_2$\end{tabular}}}}%
    \put(0,0){\includegraphics[width=\unitlength,page=2]{decomp.pdf}}%
    \put(0.20351259,0.12486522){\color[rgb]{0.50196078,0.52156863,0.2}\makebox(0,0)[lt]{\lineheight{1.25}\smash{\begin{tabular}[t]{l}$F'$\end{tabular}}}}%
    \put(0.19419543,0.28649862){\color[rgb]{0.09411765,0.69019608,0.56862745}\makebox(0,0)[lt]{\lineheight{1.25}\smash{\begin{tabular}[t]{l}$\alpha_1$\end{tabular}}}}%
    \put(0,0){\includegraphics[width=\unitlength,page=3]{decomp.pdf}}%
    \put(0.14964486,0.08280512){\color[rgb]{0.64313725,0,0.78039216}\makebox(0,0)[lt]{\lineheight{1.25}\smash{\begin{tabular}[t]{l}$\alpha_0'$\end{tabular}}}}%
    \put(0,0){\includegraphics[width=\unitlength,page=4]{decomp.pdf}}%
    \put(0.79038797,0.12486515){\color[rgb]{0.50196078,0.52156863,0.2}\makebox(0,0)[lt]{\lineheight{1.25}\smash{\begin{tabular}[t]{l}$F'$\end{tabular}}}}%
    \put(0.9001598,0.22227844){\color[rgb]{0.09411765,0.69019608,0.56862745}\makebox(0,0)[lt]{\lineheight{1.25}\smash{\begin{tabular}[t]{l}$\alpha_2$\end{tabular}}}}%
    \put(0,0){\includegraphics[width=\unitlength,page=5]{decomp.pdf}}%
    \put(0.73652026,0.08280512){\color[rgb]{0.64313725,0,0.78039216}\makebox(0,0)[lt]{\lineheight{1.25}\smash{\begin{tabular}[t]{l}$\alpha_0'$\end{tabular}}}}%
    \put(0,0){\includegraphics[width=\unitlength,page=6]{decomp.pdf}}%
    \put(0.48862864,0.30483766){\color[rgb]{0,0,0}\makebox(0,0)[t]{\lineheight{1.25}\smash{\begin{tabular}[t]{c}null-homotopic \\by induction\end{tabular}}}}%
    \put(0.70212351,0.43088989){\color[rgb]{0.38039216,0.38039216,0.38039216}\makebox(0,0)[lt]{\lineheight{1.25}\smash{\begin{tabular}[t]{l}minimal\\\end{tabular}}}}%
    \put(0,0){\includegraphics[width=\unitlength,page=7]{decomp.pdf}}%
    \put(0.46549546,0.0647571){\color[rgb]{0,0,0}\makebox(0,0)[lt]{\lineheight{1.25}\smash{\begin{tabular}[t]{l}$\bp_\be$\end{tabular}}}}%
    \put(0.72844561,0.38719256){\color[rgb]{0.38039216,0.38039216,0.38039216}\makebox(0,0)[lt]{\lineheight{1.25}\smash{\begin{tabular}[t]{l}resolution\end{tabular}}}}%
  \end{picture}%
\endgroup%

%% file: pinkpan.pdf_tex
\begingroup%
  \makeatletter%
  \providecommand\color[2][]{%
    \errmessage{(Inkscape) Color is used for the text in Inkscape, but the package 'color.sty' is not loaded}%
    \renewcommand\color[2][]{}%
  }%
  \providecommand\transparent[1]{%
    \errmessage{(Inkscape) Transparency is used (non-zero) for the text in Inkscape, but the package 'transparent.sty' is not loaded}%
    \renewcommand\transparent[1]{}%
  }%
  \providecommand\rotatebox[2]{#2}%
  \newcommand*\fsize{\dimexpr\f@size pt\relax}%
  \newcommand*\lineheight[1]{\fontsize{\fsize}{#1\fsize}\selectfont}%
  \ifx\svgwidth\undefined%
    \setlength{\unitlength}{640.85829847bp}%
    \ifx\svgscale\undefined%
      \relax%
    \else%
      \setlength{\unitlength}{\unitlength * \real{\svgscale}}%
    \fi%
  \else%
    \setlength{\unitlength}{\svgwidth}%
  \fi%
  \global\let\svgwidth\undefined%
  \global\let\svgscale\undefined%
  \makeatother%
  \begin{picture}(1,0.60240769)%
    \lineheight{1}%
    \setlength\tabcolsep{0pt}%
    \put(0,0){\includegraphics[width=\unitlength,page=1]{pinkpan.pdf}}%
    \put(0.32019473,0.28112712){\color[rgb]{0,0,0}\makebox(0,0)[t]{\lineheight{1.25}\smash{\begin{tabular}[t]{c}pentagon over the\\pink T-shirt\end{tabular}}}}%
    \put(0.71817333,0.3929381){\color[rgb]{0,0,0}\makebox(0,0)[rt]{\lineheight{1.25}\smash{\begin{tabular}[t]{r}a row of type\\ I/II squares\end{tabular}}}}%
    \put(0.19559005,0.58647178){\color[rgb]{0.09411765,0.69019608,0.56862745}\makebox(0,0)[lt]{\lineheight{1.25}\smash{\begin{tabular}[t]{l}$\alpha_1$\end{tabular}}}}%
    \put(0.46282403,0.46308968){\color[rgb]{1,0,0.49411765}\makebox(0,0)[lt]{\lineheight{1.25}\smash{\begin{tabular}[t]{l}$\alpha_1'$\end{tabular}}}}%
    \put(0.55443468,0.2259407){\color[rgb]{1,0,0.49411765}\makebox(0,0)[lt]{\lineheight{1.25}\smash{\begin{tabular}[t]{l}$\alpha_1'$\end{tabular}}}}%
    \put(0.22070746,0.0751192){\color[rgb]{0.56078431,0.23137255,0.06666667}\makebox(0,0)[lt]{\lineheight{1.25}\smash{\begin{tabular}[t]{l}$\alpha_1''$\end{tabular}}}}%
    \put(0.0944057,0.34678618){\color[rgb]{0.09411765,0.69019608,0.56862745}\makebox(0,0)[lt]{\lineheight{1.25}\smash{\begin{tabular}[t]{l}$\alpha_1$\end{tabular}}}}%
    \put(0.00030063,0.28752961){\color[rgb]{0.09411765,0.69019608,0.56862745}\makebox(0,0)[lt]{\lineheight{1.25}\smash{\begin{tabular}[t]{l}$\alpha_0$\end{tabular}}}}%
    \put(0.2171035,0.1074231){\color[rgb]{0.09411765,0.69019608,0.56862745}\makebox(0,0)[lt]{\lineheight{1.25}\smash{\begin{tabular}[t]{l}$\alpha_0$\end{tabular}}}}%
    \put(0.15588441,0.55503064){\color[rgb]{0.09411765,0.69019608,0.56862745}\makebox(0,0)[lt]{\lineheight{1.25}\smash{\begin{tabular}[t]{l}$\alpha_0'$\end{tabular}}}}%
    \put(0.42569734,0.55503064){\color[rgb]{0.09411765,0.69019608,0.56862745}\makebox(0,0)[lt]{\lineheight{1.25}\smash{\begin{tabular}[t]{l}$\alpha_0'$\end{tabular}}}}%
    \put(0.76962154,0.55503064){\color[rgb]{0.09411765,0.69019608,0.56862745}\makebox(0,0)[lt]{\lineheight{1.25}\smash{\begin{tabular}[t]{l}$\alpha_0'$\end{tabular}}}}%
    \put(0.46185736,0.28752961){\color[rgb]{0.09411765,0.69019608,0.56862745}\makebox(0,0)[lt]{\lineheight{1.25}\smash{\begin{tabular}[t]{l}$\alpha_0$\end{tabular}}}}%
    \put(0.80863576,0.28752961){\color[rgb]{0.09411765,0.69019608,0.56862745}\makebox(0,0)[lt]{\lineheight{1.25}\smash{\begin{tabular}[t]{l}$\alpha_0$\end{tabular}}}}%
    \put(0.88146586,0.54395927){\color[rgb]{0.09411765,0.69019608,0.56862745}\makebox(0,0)[lt]{\lineheight{1.25}\smash{\begin{tabular}[t]{l}$\alpha_2$\end{tabular}}}}%
    \put(0.9744735,0.30481234){\color[rgb]{0.09411765,0.69019608,0.56862745}\makebox(0,0)[lt]{\lineheight{1.25}\smash{\begin{tabular}[t]{l}$\alpha_2$\end{tabular}}}}%
    \put(0,0){\includegraphics[width=\unitlength,page=2]{pinkpan.pdf}}%
  \end{picture}%
\endgroup%

%% file: grid.pdf_tex
\begingroup%
  \makeatletter%
  \providecommand\color[2][]{%
    \errmessage{(Inkscape) Color is used for the text in Inkscape, but the package 'color.sty' is not loaded}%
    \renewcommand\color[2][]{}%
  }%
  \providecommand\transparent[1]{%
    \errmessage{(Inkscape) Transparency is used (non-zero) for the text in Inkscape, but the package 'transparent.sty' is not loaded}%
    \renewcommand\transparent[1]{}%
  }%
  \providecommand\rotatebox[2]{#2}%
  \newcommand*\fsize{\dimexpr\f@size pt\relax}%
  \newcommand*\lineheight[1]{\fontsize{\fsize}{#1\fsize}\selectfont}%
  \ifx\svgwidth\undefined%
    \setlength{\unitlength}{413.49096103bp}%
    \ifx\svgscale\undefined%
      \relax%
    \else%
      \setlength{\unitlength}{\unitlength * \real{\svgscale}}%
    \fi%
  \else%
    \setlength{\unitlength}{\svgwidth}%
  \fi%
  \global\let\svgwidth\undefined%
  \global\let\svgscale\undefined%
  \makeatother%
  \begin{picture}(1,0.64438196)%
    \lineheight{1}%
    \setlength\tabcolsep{0pt}%
    \put(0,0){\includegraphics[width=\unitlength,page=1]{grid.pdf}}%
    \put(0.23557616,0.53889168){\color[rgb]{0.3254902,0.29411765,0.17254902}\makebox(0,0)[lt]{\lineheight{1.25}\smash{\begin{tabular}[t]{l}$h_1$\end{tabular}}}}%
    \put(0.39577944,0.53889168){\color[rgb]{0.3254902,0.29411765,0.17254902}\makebox(0,0)[lt]{\lineheight{1.25}\smash{\begin{tabular}[t]{l}$h_2$\end{tabular}}}}%
    \put(0.72952283,0.53889168){\color[rgb]{0.3254902,0.29411765,0.17254902}\makebox(0,0)[lt]{\lineheight{1.25}\smash{\begin{tabular}[t]{l}$h_n$\end{tabular}}}}%
    \put(0.55714462,0.5382968){\color[rgb]{0.3254902,0.29411765,0.17254902}\makebox(0,0)[lt]{\lineheight{1.25}\smash{\begin{tabular}[t]{l}$\cdots$\end{tabular}}}}%
    \put(0.10091153,0.50524562){\color[rgb]{0.16862745,0.11764706,0.69803922}\makebox(0,0)[lt]{\lineheight{1.25}\smash{\begin{tabular}[t]{l}$h_1'$\end{tabular}}}}%
    \put(0.10091153,0.3736452){\color[rgb]{0.16862745,0.11764706,0.69803922}\makebox(0,0)[lt]{\lineheight{1.25}\smash{\begin{tabular}[t]{l}$h_2'$\end{tabular}}}}%
    \put(0.10091153,0.12039094){\color[rgb]{0.16862745,0.11764706,0.69803922}\makebox(0,0)[lt]{\lineheight{1.25}\smash{\begin{tabular}[t]{l}$h_m'$\end{tabular}}}}%
    \put(0.11179449,0.25309337){\color[rgb]{0.16862745,0.11764706,0.69803922}\makebox(0,0)[lt]{\lineheight{1.25}\smash{\begin{tabular}[t]{l}$\vdots$\end{tabular}}}}%
    \put(0.05819946,0.61934049){\color[rgb]{0,0,0}\makebox(0,0)[lt]{\lineheight{1.25}\smash{\begin{tabular}[t]{l}$(\Sigma, \bal, \bbe)$\end{tabular}}}}%
    \put(0.76500731,0.6193405){\color[rgb]{0,0,0}\makebox(0,0)[lt]{\lineheight{1.25}\smash{\begin{tabular}[t]{l}$(\Sigma, \bal', \bbe)$\end{tabular}}}}%
    \put(0.05819946,0.00902354){\color[rgb]{0,0,0}\makebox(0,0)[lt]{\lineheight{1.25}\smash{\begin{tabular}[t]{l}$(\Sigma, \bal, \bbe')$\end{tabular}}}}%
    \put(0.76504239,0.00902161){\color[rgb]{0,0,0}\makebox(0,0)[lt]{\lineheight{1.25}\smash{\begin{tabular}[t]{l}$(\Sigma, \bal', \bbe')$\end{tabular}}}}%
    \put(0,0){\includegraphics[width=\unitlength,page=2]{grid.pdf}}%
    \put(0.36483484,0.61968328){\color[rgb]{0.41960784,0.6,0.19215686}\transparent{0.69411802}\makebox(0,0)[lt]{\lineheight{1.25}\smash{\begin{tabular}[t]{l}$\alpha$-equivalence\end{tabular}}}}%
    \put(0.87527102,0.32129122){\color[rgb]{0.37647059,0.50196078,0.91764706}\makebox(0,0)[lt]{\lineheight{1.25}\smash{\begin{tabular}[t]{l}$\beta$-equiv.\end{tabular}}}}%
    \put(0,0){\includegraphics[width=\unitlength,page=3]{grid.pdf}}%
  \end{picture}%
\endgroup%

%% file: row.pdf_tex
\begingroup%
  \makeatletter%
  \providecommand\color[2][]{%
    \errmessage{(Inkscape) Color is used for the text in Inkscape, but the package 'color.sty' is not loaded}%
    \renewcommand\color[2][]{}%
  }%
  \providecommand\transparent[1]{%
    \errmessage{(Inkscape) Transparency is used (non-zero) for the text in Inkscape, but the package 'transparent.sty' is not loaded}%
    \renewcommand\transparent[1]{}%
  }%
  \providecommand\rotatebox[2]{#2}%
  \newcommand*\fsize{\dimexpr\f@size pt\relax}%
  \newcommand*\lineheight[1]{\fontsize{\fsize}{#1\fsize}\selectfont}%
  \ifx\svgwidth\undefined%
    \setlength{\unitlength}{362.2099172bp}%
    \ifx\svgscale\undefined%
      \relax%
    \else%
      \setlength{\unitlength}{\unitlength * \real{\svgscale}}%
    \fi%
  \else%
    \setlength{\unitlength}{\svgwidth}%
  \fi%
  \global\let\svgwidth\undefined%
  \global\let\svgscale\undefined%
  \makeatother%
  \begin{picture}(1,0.35461597)%
    \lineheight{1}%
    \setlength\tabcolsep{0pt}%
    \put(0,0){\includegraphics[width=\unitlength,page=1]{row.pdf}}%
    \put(0.06765752,0.15757729){\color[rgb]{0.91764706,0.52941176,0.37647059}\makebox(0,0)[lt]{\lineheight{1.25}\smash{\begin{tabular}[t]{l}$d$\end{tabular}}}}%
    \put(0.91410769,0.15757729){\color[rgb]{0.91764706,0.52941176,0.37647059}\makebox(0,0)[lt]{\lineheight{1.25}\smash{\begin{tabular}[t]{l}$d$\end{tabular}}}}%
    \put(0.20341262,0.23419054){\color[rgb]{0.42352941,0.4,0.30980392}\makebox(0,0)[lt]{\lineheight{1.25}\smash{\begin{tabular}[t]{l}$h_1$\end{tabular}}}}%
    \put(0.38629719,0.23419054){\color[rgb]{0.42352941,0.4,0.30980392}\makebox(0,0)[lt]{\lineheight{1.25}\smash{\begin{tabular}[t]{l}$h_2$\end{tabular}}}}%
    \put(0.76729138,0.23419054){\color[rgb]{0.42352941,0.4,0.30980392}\makebox(0,0)[lt]{\lineheight{1.25}\smash{\begin{tabular}[t]{l}$h_n$\end{tabular}}}}%
    \put(0.57050816,0.23351143){\color[rgb]{0.42352941,0.4,0.30980392}\makebox(0,0)[lt]{\lineheight{1.25}\smash{\begin{tabular}[t]{l}$\cdots$\end{tabular}}}}%
    \put(0.20341259,0.10001274){\color[rgb]{0.42352941,0.4,0.30980392}\makebox(0,0)[lt]{\lineheight{1.25}\smash{\begin{tabular}[t]{l}$h_1'$\end{tabular}}}}%
    \put(0.38629716,0.10001274){\color[rgb]{0.42352941,0.4,0.30980392}\makebox(0,0)[lt]{\lineheight{1.25}\smash{\begin{tabular}[t]{l}$h_2'$\end{tabular}}}}%
    \put(0.76729141,0.10001274){\color[rgb]{0.42352941,0.4,0.30980392}\makebox(0,0)[lt]{\lineheight{1.25}\smash{\begin{tabular}[t]{l}$h_n'$\end{tabular}}}}%
    \put(0.57050807,0.09933363){\color[rgb]{0.42352941,0.4,0.30980392}\makebox(0,0)[lt]{\lineheight{1.25}\smash{\begin{tabular}[t]{l}$\cdots$\end{tabular}}}}%
    \put(0.00920572,0.3260292){\color[rgb]{0,0,0}\makebox(0,0)[lt]{\lineheight{1.25}\smash{\begin{tabular}[t]{l}$(\Sigma, \bal, \bbe)$\end{tabular}}}}%
    \put(0.8088591,0.32602917){\color[rgb]{0,0,0}\makebox(0,0)[lt]{\lineheight{1.25}\smash{\begin{tabular}[t]{l}$(\Sigma, \overline{\bal}, \bbe)$\end{tabular}}}}%
    \put(0.00092322,0.01029895){\color[rgb]{0,0,0}\makebox(0,0)[lt]{\lineheight{1.25}\smash{\begin{tabular}[t]{l}$(\Sigma', \bal', \bbe')$\end{tabular}}}}%
    \put(0.79593517,0.01029964){\color[rgb]{0,0,0}\makebox(0,0)[lt]{\lineheight{1.25}\smash{\begin{tabular}[t]{l}$(\Sigma', \overline{\bal}', \bbe')$\end{tabular}}}}%
  \end{picture}%
\endgroup%

%% file: stab_slides.pdf_tex
\begingroup%
  \makeatletter%
  \providecommand\color[2][]{%
    \errmessage{(Inkscape) Color is used for the text in Inkscape, but the package 'color.sty' is not loaded}%
    \renewcommand\color[2][]{}%
  }%
  \providecommand\transparent[1]{%
    \errmessage{(Inkscape) Transparency is used (non-zero) for the text in Inkscape, but the package 'transparent.sty' is not loaded}%
    \renewcommand\transparent[1]{}%
  }%
  \providecommand\rotatebox[2]{#2}%
  \newcommand*\fsize{\dimexpr\f@size pt\relax}%
  \newcommand*\lineheight[1]{\fontsize{\fsize}{#1\fsize}\selectfont}%
  \ifx\svgwidth\undefined%
    \setlength{\unitlength}{396.80653489bp}%
    \ifx\svgscale\undefined%
      \relax%
    \else%
      \setlength{\unitlength}{\unitlength * \real{\svgscale}}%
    \fi%
  \else%
    \setlength{\unitlength}{\svgwidth}%
  \fi%
  \global\let\svgwidth\undefined%
  \global\let\svgscale\undefined%
  \makeatother%
  \begin{picture}(1,0.35395567)%
    \lineheight{1}%
    \setlength\tabcolsep{0pt}%
    \put(0,0){\includegraphics[width=\unitlength,page=1]{stab_slides.pdf}}%
    \put(0.02017657,0.19776251){\color[rgb]{0.91764706,0.24705882,0.40392157}\makebox(0,0)[lt]{\lineheight{1.25}\smash{\begin{tabular}[t]{l}stab\end{tabular}}}}%
    \put(0.14409551,0.26013567){\color[rgb]{0.42352941,0.4,0.30980392}\makebox(0,0)[lt]{\lineheight{1.25}\smash{\begin{tabular}[t]{l}$h_1$\end{tabular}}}}%
    \put(0.28079337,0.26013567){\color[rgb]{0.42352941,0.4,0.30980392}\makebox(0,0)[lt]{\lineheight{1.25}\smash{\begin{tabular}[t]{l}$h_2$\end{tabular}}}}%
    \put(0.54918644,0.26013567){\color[rgb]{0.42352941,0.4,0.30980392}\makebox(0,0)[lt]{\lineheight{1.25}\smash{\begin{tabular}[t]{l}$h_n$\end{tabular}}}}%
    \put(0.40736157,0.25951577){\color[rgb]{0.42352941,0.4,0.30980392}\makebox(0,0)[lt]{\lineheight{1.25}\smash{\begin{tabular}[t]{l}$\cdots$\end{tabular}}}}%
    \put(0.14409548,0.13765651){\color[rgb]{0.42352941,0.4,0.30980392}\makebox(0,0)[lt]{\lineheight{1.25}\smash{\begin{tabular}[t]{l}$h_1'$\end{tabular}}}}%
    \put(0.28079331,0.13765651){\color[rgb]{0.42352941,0.4,0.30980392}\makebox(0,0)[lt]{\lineheight{1.25}\smash{\begin{tabular}[t]{l}$h_2'$\end{tabular}}}}%
    \put(0.54918644,0.13765651){\color[rgb]{0.42352941,0.4,0.30980392}\makebox(0,0)[lt]{\lineheight{1.25}\smash{\begin{tabular}[t]{l}$h_n'$\end{tabular}}}}%
    \put(0.40736157,0.13703661){\color[rgb]{0.42352941,0.4,0.30980392}\makebox(0,0)[lt]{\lineheight{1.25}\smash{\begin{tabular}[t]{l}$\cdots$\end{tabular}}}}%
    \put(0.0008427,0.32786139){\color[rgb]{0,0,0}\makebox(0,0)[lt]{\lineheight{1.25}\smash{\begin{tabular}[t]{l}$(\Sigma, \bal, \bbe)$\end{tabular}}}}%
    \put(0.56608395,0.32786141){\color[rgb]{0,0,0}\makebox(0,0)[lt]{\lineheight{1.25}\smash{\begin{tabular}[t]{l}$(\Sigma, \overline{\bal}, \bbe)$\end{tabular}}}}%
    \put(0.91963111,0.05954551){\color[rgb]{0,0,0}\makebox(0,0)[lt]{\lineheight{1.25}\smash{\begin{tabular}[t]{l}$(\Sigma', \overline{\bal}', \bbe')$\end{tabular}}}}%
    \put(0,0){\includegraphics[width=\unitlength,page=2]{stab_slides.pdf}}%
    \put(0.49787506,0.06027125){\color[rgb]{0.63529412,0.17254902,0.28235294}\makebox(0,0)[lt]{\lineheight{1.25}\smash{\begin{tabular}[t]{l}$(\Sigma, \overline{\bal}, \bbe)\#(T^2,\alpha_0,\beta_0)$\end{tabular}}}}%
    \put(0.75228582,0.13703655){\color[rgb]{0.22745098,0.31764706,0.56862745}\makebox(0,0)[lt]{\lineheight{1.25}\smash{\begin{tabular}[t]{l}$\cdots$\end{tabular}}}}%
    \put(0.98324745,0.26814371){\color[rgb]{0.22745098,0.31764706,0.56862745}\makebox(0,0)[rt]{\lineheight{1.25}\smash{\begin{tabular}[t]{r}stabilization\\slides\end{tabular}}}}%
    \put(-0.00130039,0.05954538){\color[rgb]{0,0,0}\makebox(0,0)[lt]{\lineheight{1.25}\smash{\begin{tabular}[t]{l}$(\Sigma', \bal', \bbe')=$\end{tabular}}}}%
    \put(0.07719077,0.00870958){\color[rgb]{0,0,0}\makebox(0,0)[lt]{\lineheight{1.25}\smash{\begin{tabular}[t]{l}$(\Sigma,\bal,\bbe)\#(T^2,\alpha_0,\beta_0)$\end{tabular}}}}%
  \end{picture}%
\endgroup%